\documentclass{amsart}

\usepackage{amsmath, amssymb, amsthm}
\usepackage{amsfonts,amsbsy}
\usepackage{array}
\usepackage{mathtools}

\usepackage{booktabs} 

\usepackage{dutchcal} 
\usepackage{dsfont} 

\usepackage{graphicx}
\usepackage{hyperref}
\usepackage[all]{xy}
\usepackage{color}
\usepackage{fullpage}
\usepackage{xspace} 
\usepackage{enumitem}
\usepackage[normalem]{ulem} 

\makeatletter
\newcommand*\bigcdot{\mathpalette\bigcdot@{.5}}
\newcommand*\bigcdot@[2]{\mathbin{\vcenter{\hbox{\scalebox{#2}{$\m@th#1\bullet$}}}}}
\makeatother

\makeatletter
\def\zz\ignorespaces{\@ifnextchar-{}{\phantom{-}}}
\newcolumntype{C}{>{\zz}{c}}
\makeatother

\newtheorem{theorem}{Theorem}[section]
\newtheorem{corollary}[theorem]{Corollary}
\newtheorem{conjecture}[theorem]{Conjecture}
\newtheorem{proposition}[theorem]{Proposition}
\newtheorem{lemma}[theorem]{Lemma}
\newtheorem{claim}[theorem]{Claim}

\theoremstyle{definition}
\newtheorem{defn}[theorem]{Definition}
\newtheorem{remark}[theorem]{Remark}
\newtheorem{example}[theorem]{Example}

\newcommand{\bdry}{\partial}

\newcommand{\R}{\mathbb{R}}

\renewcommand{\L}{\mathcal{L}}

\newcommand{\calA}{\mathcal{A}}  
\newcommand{\calC}{\mathcal{C}}

\newcommand{\calt}{\mathcal{t}}

\newcommand{\calc}{\mathcal{c}}

\newcommand{\hatR}{\widehat{R}}

\newcommand{\Z}{\mathbb{Z}}

\newcommand{\HF}{\widehat{HF}}

\newcommand{\cut}{\setminus}
\newcommand{\nbhd}{\mathcal{N}}

\DeclareMathOperator{\fix}{fix}

\DeclareMathOperator{\rk}{rk}
\DeclareMathOperator{\Diff}{Diff}

\DeclareMathOperator{\Def}{Def}

\title{Asymmetric L-space knots}

\author{Kenneth L.\ Baker  \and John Luecke}
\address{Department of Mathematics, University of Miami, 
Coral Gables, FL 33146, USA}
\email{k.baker@math.miami.edu}
\address{Department of Mathematics, University of Texas, 
Austin, TX  78712, USA}
\email{luecke@math.utexas.edu}

\begin{document}


\begin{abstract}
We construct the first examples of asymmetric L-space knots in $S^3$.  More specifically, we exhibit a construction of hyperbolic knots in $S^3$ with both (i) a surgery that may be realized as a surgery on a strongly invertible link such that the result of the surgery is the double branched cover of an alternating link and (ii) trivial isometry group. In particular, this produces L-space knots in $S^3$ which are not strongly invertible.  The construction also immediately extends to produce asymmetric L-space knots in any lens space, including $S^1 \times S^2$.
\end{abstract}

\maketitle

\tableofcontents

\section{Introduction}

A knot $K$ in $S^3$ is an {\em L-space knot} if it admits a non-trivial Dehn surgery to a manifold $Y$ such that $\rk \HF{Y} = |H_1(Y, \Z)|$ \cite{OS-lens}. 
The knot $K$ is  {\em strongly-invertible} if there is an orientation preserving involution of $S^3$ whose fixed point set is a circle that intersects the knot in two points and takes the knot $K$ to itself.
Watson makes a passing comment that ``many L-space knots are strongly invertible'' preceding  \cite[Theorem 1.2]{watsonsurgicalperspective}.  In the absence of evidence of L-space knots that are not strongly invertible, this has been promoted to a question (e.g.\ \cite{motegi-twisting} and \cite{hom-problemlist}) and eventually a conjecture.

\begin{conjecture}[{e.g.\ \cite[Example 20]{lidmanmoore} and \cite[Conjecture 30]{watsonsymmetry}}]\label{conj:main}
An L-space knot in $S^3$ is strongly-invertible.
\end{conjecture}

We define the {\em symmetry group} of a 3-manifold, $N$, to be $\pi_0(\Diff(N))$, the group of 
isotopy classes of diffeomorphisms of $N$. Let $K$ be a knot in a 3-manifold $Y$ whose 
complement admits a complete hyperbolic metric of finite volume. We refer to such a knot as
a {\em hyperbolic knot} in $Y$. We say that $K$ is {\em asymmetric} if the symmetry group
of its complement is trivial, that is if any diffeomorphism of the the knot complement is isotopic
to the identity. 

The main result of this paper is

\begin{theorem}\label{thm:main}
There exist asymmetric hyperbolic L-space knots $K$ in $S^3$.
\end{theorem}

The inclusion of the isometry group of a hyperbolic knot complement into the symmetry group of
the knot complement is an isomorphism, see for example the paragraph preceding Theorem 6.2
in \cite{Bon}. On the other hand, the Orbifold Theorem shows that the isometry group of a
strongly invertible hyperbolic knot is non-trivial. The symmetry group of a strongly invertible
hyperbolic knot then is non-trivial. Thus Theorem~\ref{thm:main} shows that Conjecture~\ref{conj:main}
is false, by giving examples of L-space knots in $S^3$ which cannot be strongly invertible.

\begin{remark}
Let the finite group G act on a 
3-manifold $N$ which admits a finite volume hyperbolic structure with totally geodesic boundary. 
If its fixed set has dimension at least one, then the action is conjugate to an isometric action by a diffeomorphism of $N$ which is isotopic to the identity. See Theorem 6.3 of \cite{Bon}.
\end{remark}

Dunfield, Hoffman, and Licata \cite{dhl} found examples of asymmetric hyperbolic  L-space knots in lens spaces with non-trivial surgeries to other lens spaces, the double branched covers of two-bridge links.  This then allowed them to demonstrate the existence of asymmetric hyperbolic L-spaces, \cite[Theorem 1.2]{dhl}.  Similarly, Theorem~\ref{thm:main} along with \cite[Lemma 2.2]{dhl} (see also \cite{HW}) and
\cite[Proposition 2.1]{OS-lens} shows that such manifolds can be obtained by Dehn surgery on knots in $S^3$.

\begin{corollary}
There exist asymmetric hyperbolic L-spaces that are obtained as Dehn surgery on knots in $S^3$. \qed
\end{corollary}

The double branched cover of a non-split alternating link is an L-space \cite{OS-dbc}. When a non-trivial Dehn
surgery on a knot yields the double branched cover of a non-split alternating link, we refer to this Dehn
surgery as an {\em alternating surgery} on the knot. Thus any knot in $S^3$ admitting an alternating surgery is
an L-space knot.

 Along the lines of Conjecture~\ref{conj:main}, McCoy proposed: 
\begin{conjecture}[{\cite[Conjecture 1]{mccoy-boundsonsurgeriesbranchingoveralternating}}]
\label{conj:alternatingmccoy}
If a knot in $S^3$ has an alternating surgery, then the knot corresponds to the dealternation crossing in an almost alternating diagram of the unknot.
\end{conjecture}

The knot corresponding to a dealternation crossing is strongly invertible.
Our proof of Theorem~\ref{thm:main} is a consequence of the following theorem which shows McCoy's conjecture is also false.
\begin{theorem}\label{thm:dbcaltsurgery}
There exist asymmetric hyperbolic knots in $S^3$ with alternating surgeries.
\end{theorem}

In Conjecture~\ref{conj:alternating} we propose a modification of Conjecture~\ref{conj:alternatingmccoy}.

\begin{example}\label{example:simplest}
	Our simplest example is the genus $119$ knot $K$ that is the closure of the positive $12$--strand braid of length $249$
shown in Figure~\ref{fig:smallestAsymLSK-posbraid} (Left).  The result of $272$--surgery on $K$ produces the double branched cover of the $12$--crossing alternating link $J^*$ shown in Figure~\ref{fig:smallestAsymLSK-posbraid} (Right)  which is also known as the link $L12a955$.

 SnapPy and Sage confirm the complement of $K$ is hyperbolic and asymmetric and compute that its complement has volume $\sim\!10.20098$ with cusp shape $\sim\!0.41433 + 1.19820i$ and canonical triangulation comprised of $13$ tetrahedra  \cite{SnapPy, sagemath}.

\begin{figure}
	\centering
	\includegraphics[width=6in]{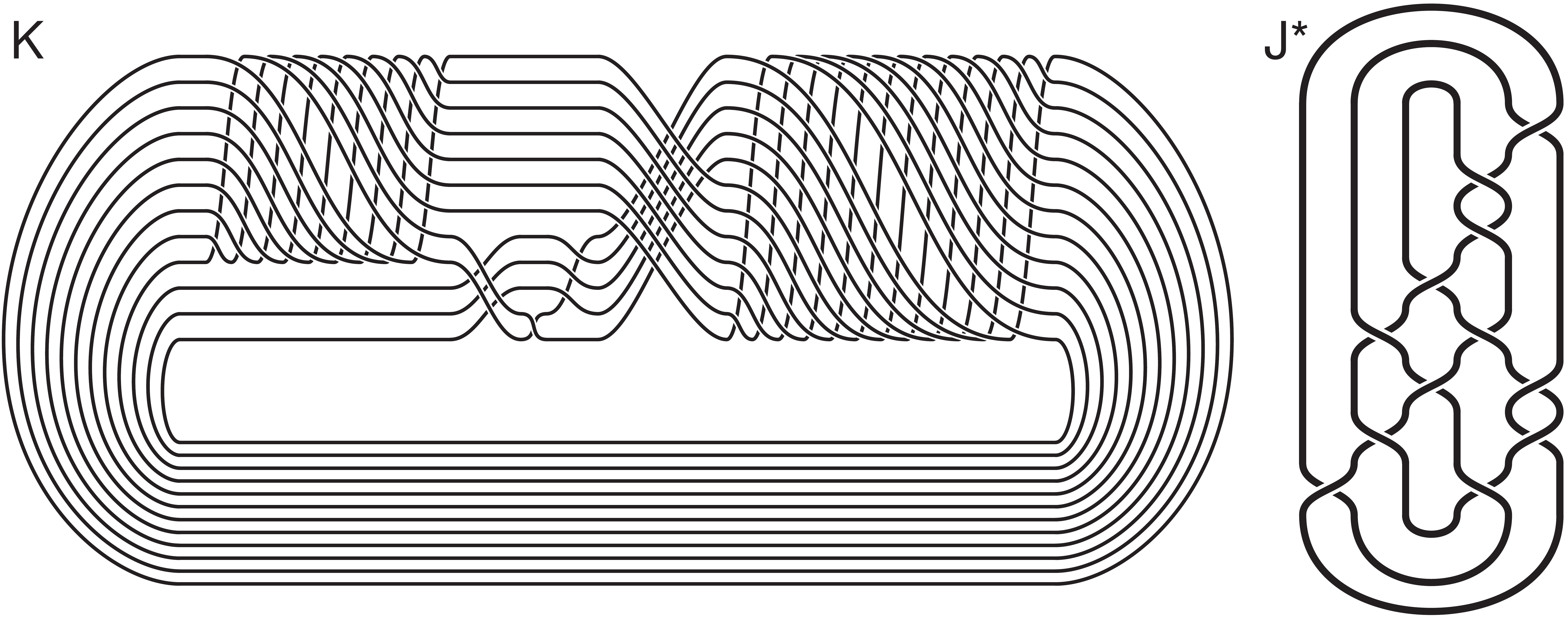}
	\caption{	The knot $K$ (Left) is the closure of a $12$--strand positive braid of length $249$. 
		It is an asymmetric hyperbolic L-space knot.  The result of $272$--surgery on $K$  produces the double branched cover of the $12$--crossing alternating link $J^*$ (Right).  }
	\label{fig:smallestAsymLSK-posbraid}
\end{figure}
\end{example}

\begin{remark}
	Boileau, Gonzalez-Acu\~{n}a, and Montesinos \cite{bgam} and Callahan \cite{callahan} have created knots in $S^3$ with no symmetry of finite order and yet admit a non-trivial Dehn surgery to double branched covers of $S^3$.  However it is not clear if non-strongly invertible L-spaces knots may arise from their constructions.
\end{remark}

Our construction giving Theorem~\ref{thm:dbcaltsurgery} easily adapts with $S^3$ replaced by any of the lens spaces, including $S^1 \times S^2$.  This adaptation is detailed in Theorem~\ref{thm:extension}.  
Following Rasmussen-Rasmussen \cite{RR}, the exterior of a knot in $S^1 \times S^2$ with a non-trivial L-space surgery is a {\em generalized solid torus}; every non-trivial surgery on such a knot is an L-space.  Thus the adaptation gives the following theorem.
\begin{theorem}
	There exist asymmetric hyperbolic generalized solid tori. \qed
\end{theorem}

\subsection{Heegaard genus and tunnel number}

As suggested in \cite{bakermoore}, an upper bound on the tunnel number of hyperbolic L-space knots in $S^3$ is not expected, though the familiar examples all have tunnel number $1$.  Motegi has demonstrated strongly invertible hyperbolic L-space knots with tunnel number $2$ \cite{motegi-twisting}.

 Since tunnel number one knots are all strongly invertible, the tunnel number of an asymmetric hyperbolic L-space knot must be at least $2$.  The families of knots we create in this article have bounded tunnel number.  Indeed, as shown in Proposition~\ref{prop:tnofgenus2lashing}, the lashing construction (section 2) applied to a genus $2$ Heegaard splitting produces knots with tunnel number at most $3$.   Moreover, Proposition~\ref{prop:tn2lashing} shows that the asymmetric L-space knots we construct have tunnel number $2$, and Proposition~\ref{prop:dbcgenus3} shows that the Heegaard genus of the double branched cover of the alternating link obtained by surgery is at most $3$.  As evidenced by the fact that the link $J^*$ of Example~\ref{example:simplest} is a $3$--bridge link,  our asymmetric hyperbolic L-space knots 
sometimes admit surgeries to Heegaard genus $2$ L-spaces.

\subsection{Overview of construction of asymmetric L-space knots}
Our construction begins with a family of {\em lashings} of a pair of pants $P$ embedded in a $3$--manifold $Y$. Let $H_P$ be a genus $2$ handlebody properly containing $P$ that is identified with the closed product neighborhood $P \times [-1,1]$ so that $P$ is identified with $P \times \{0\}$.  Described more explicitly in Section~\ref{sec:lashings}, a {\em $p/q$--lashing} of $P$ is a particular framed knot in $H_P$, relative to a labeling of $\bdry P$ as $\calC = C_\nu \cup C_\mu \cup C_\lambda$.  Lemma~\ref{lem:lashingsurgery} shows that this $p/q$--lashing of $P$ has the property that framed surgery on the lashing is equivalent to $(+1, -q/(p+q), -p/(p+q))$--surgery on $\calC = C_\nu \cup C_\mu \cup C_\lambda$ where the surgery coefficients are taken with respect to the framing by $P$.  Figure~\ref{fig:Kgammasugery} illustrates the surgery calculus that proves this lemma.

While every framed knot in $Y$ can be expressed as a lashing of $P$ for some pair of pants $P$, we find utility in lashings  when $\bdry P$ enjoys properties that are not apparent in the lashings of $P$.  Suppose $Y$ admits a genus $2$ Heegaard surface $\Sigma = P \cup_{\calC} Q$ expressed as the union of two pairs of pants $P$ and $Q$ such that $\bdry P$ and $\bdry Q$ are both identified as the triple of curves $\calC$.  Since each component of $\calC$ is non-separating in $\Sigma$, the hyperelliptic involution  on $\Sigma$ may be isotoped to exchange $P$ and $Q$ while fixing $\calC$ component-wise.  This involution
on $\Sigma$ extends to an involution $\iota$ on $Y$ in which $\calC$ is a strongly invertible link and yet, due to the exchanging of $P$ and $Q$ by  $\iota$, the lashings of $P$ seem unlikely to be isotopic to a knot invariant under $\iota$ --- at least in general. 

We use this set-up to construct knots proving Theorem~\ref{thm:main} as follows.  
\begin{enumerate}
	\item We find a family of $3$--bridge presentations of the unknot parameterized by an alternating $3$--braid $\alpha$ and integer $m$ so that the triple of  $(+1, -q/(p+q), -p/(p+q))$--rational tangle replacements along a triple of arcs $\calc=\calc^{\alpha,m}$ in the bridge sphere produces an alternating link for coprime integers $p$ and $q$ with $p/q>0$.  Using the Montesinos Trick, this lifts to $(+1, -q/(p+q), -p/(p+q))$--surgery on a $3$--component link $\calC=\calC^{\alpha,m}$ in $S^3$.  Being the double branched cover of a non-split alternating link, the resulting manifold is an L-space \cite{OS-dbc}.  The bridge sphere lifts to a genus $2$ Heegaard surface $\Sigma$ divided by $\calC$ into pairs of pants $P=P^{\alpha,m}$ and $Q=Q^{\alpha,m}$.  As in Lemma~\ref{lem:lashingsurgery}, the surgery on $\calC$ may be obtained by framed surgery on a $p/q$--lashing of $P$.  Hence the $p/q$--lashing is an L-space knot.
	\item The manifold $M = Y - \nbhd(P)$ is the union of two genus two handlebodies along the pair of pants $Q$.  The exterior of the $p/q$--lashing of $P$ in $Y$ may be identified as $M[\gamma]$, the result of attaching a $2$--handle to $M$ along a particular non-separating curve $\gamma$ in the genus $2$ boundary of $M$.  For ``suitable'' choices of $P$ and $Q$ in $Y$, based upon the disk-busting and annulus-busting nature of $\calC$ in the two handlebodies $M \cut Q$,  we show that $M$ is a simple $3$--manifold (Lemma~\ref{lem:Msimple}) and that $Q$ is the unique separating, incompressible, boundary incompressible pair of pants in $M$ (Lemmas~\ref{lem:Qunique} and \ref{lem:Qess}).  Then, for ``sufficiently complicated'' lashings of $P$, Theorem~\ref{thm:mainasymmetry} shows that 
	 	the complement of the lashing is a hyperbolic manifold, and any diffeomorphism of its
exterior is isotopic to the identity.  The latter is done by first ensuring any such diffeomorphism may be isotoped to restrict to a diffeomorphism of $M$,  then to one that is the identity on $Q$, and eventually to one that is the identity on $M$ and $\gamma$. 
	\item Finally, we verify that the parameters $\alpha$ and $m$ may be chosen so that  $P=P^{\alpha,m}$ and $Q=Q^{\alpha,m}$ are ``suitable'' and that there are coprime integers $p$ and $q$ with $p/q>0$ for which the $p/q$--lashing of $P$ is ``sufficiently complicated''.
\end{enumerate}

\subsection{Basic notation}
Throughout,  $A \cut B$ means the closure of $A-B$ in the path metric, $\nbhd(B)$ denotes a regular tubular neighborhood of $B$.  When $A$ and $B$ are properly embedded submanifolds, $\Delta(A,B)$ is the geometric intersection number, the minimal number of intersections of $A$ and $B$ up to proper isotopy.

If $L=L_1 \cup \dots \cup L_n$ is an $n$--component link in a $3$--manifold $Y$, $Y_L(r_1, \dots, r_n)$ denotes the result of $r_i$--Dehn surgery on the component $L_i$ for each $i=1, \dots, n$.  If $X = Y\cut L$, then this may also be denoted $X(r_1, \dots, r_n)$.  

\subsection{Acknowledgements}
We would like to thank Duncan McCoy for discussions about bandings of alternating links that produce unknots.

KLB was partially supported by grants from the Simons Foundation (\#209184 and \#523883 to Kenneth L.\ Baker).

\section{Lashings}\label{sec:lashings}
Consider an oriented pair of pants $P$ embedded in a $3$--manifold $Y$ with an oriented once-punctured torus $T$  embedded in a  closed product neighborhood $P \times [-1,1]=P\times I\subset Y$ so that $T$ has a projection to $P$ as shown in Figure~\ref{fig:TinPxIprojection}. We identify $P$ with $P \times \{0\}$.

A basis of curves $\mu,\lambda$ in $T$ are also shown, oriented so that $\mu \cdot \lambda = +1$.  An unoriented, non-trivial, simple closed curve $\tau$ in $T$ has {\em slope} $p/q$ if it is homologous to $p[\mu]+q[\lambda]$ for some orientation.
A curve in $T$ of slope $p/q$ with respect to this basis is a knot $K=K(p/q)$ in $P\times I \subset Y$. 
  We say the knot $K$ with framing given by $T$ is a {\em $p/q$--lashing} of $P$.

\begin{figure}
	\centering
	\includegraphics[width=2in]{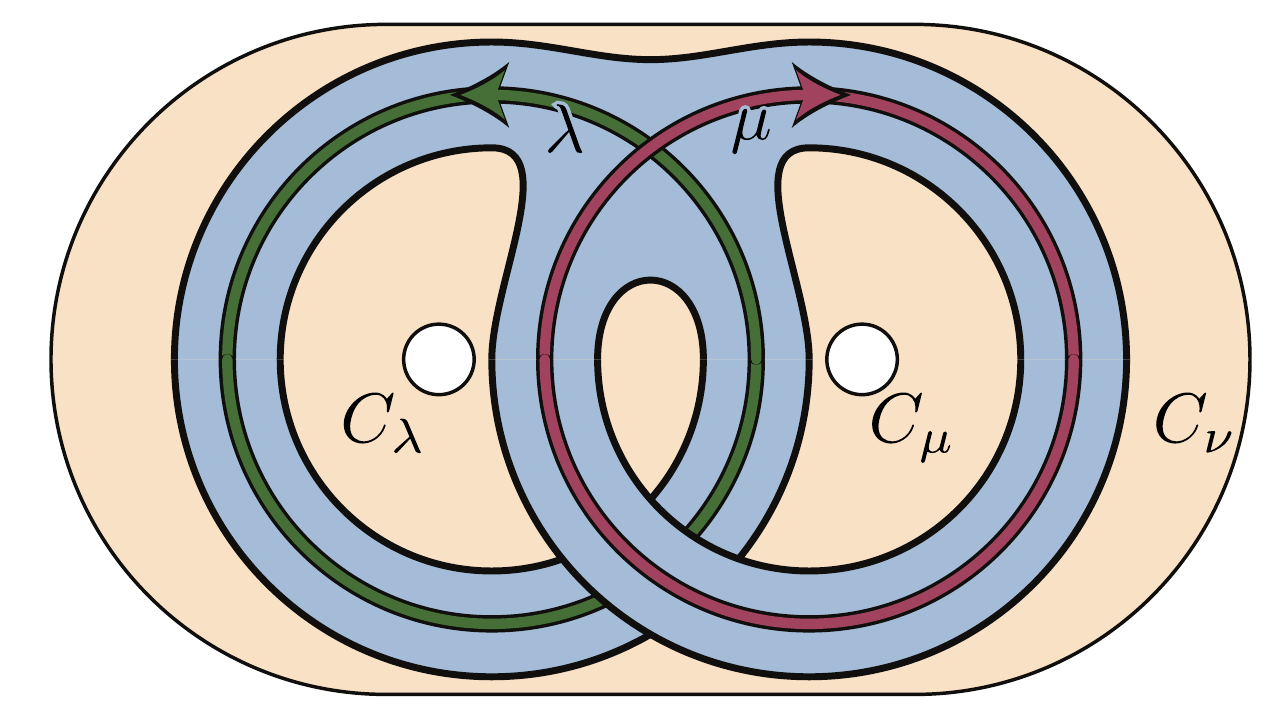}
	\caption{A projection to $P$ of the once-punctured torus $T \subset P\times I$ with its basis of curves $\mu,\lambda$ along with the three boundary curves $C_\nu, C_\mu, C_\lambda$.}
	\label{fig:TinPxIprojection}
\end{figure}

\begin{figure}
	\centering
	\includegraphics[height=1.5in]{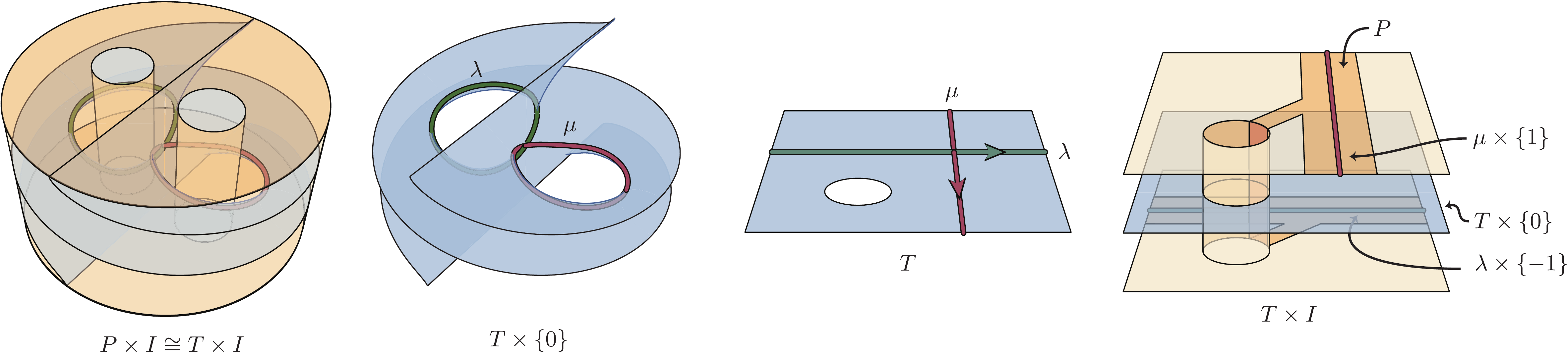}
	\caption{
	(Left two) The handlebody $P\times I (\cong T \times I)$ containing the properly embedded once-punctured torus $T=T\times \{0\}$, along with this configuration of $T\times \{0\}$ isolated and the curves $\lambda, \mu$ labeled.  (Right two) The once-punctured torus $T$ and its curves $\lambda,\mu$ in a more standard form where opposite sides of the rectangle are identifed. Adjacent is a view of the handlebody $T \times I (\cong P \times I)$ with the pair of pants $P = P\times \{1\}$ marked on its boundary. (
	This figure is essentially Figures 1 and 2 from \cite{BBL-handlebody} except that $\mu$ and $\lambda$ are swapped.  Equivalently, one may change the orientation of $T$ in \cite[Figure 2]{BBL-handlebody}.)
	}	
	\label{fig:TinPxI}
\end{figure}

Note that the three boundary components of $P$ are each isotopic in $P \times I$ to one of the lashings $\nu = K(-1/1)$, $\mu = K(1/0)$, $\lambda = K(0/1)$. We call them $C_\nu, C_\mu, C_\lambda$ accordingly and collectively $\calC = \bdry P (=\bdry P \times \{0\})$.

It will be helpful to have another way to view a lashing inside of $P \times I$.
Observe that $T$ may be isotoped in $P\times I$ to be  properly embedded so that $P \times I \cong T \times I$ where the boundary of $T\times \{0\}$  crosses $\bdry P \times \{0\}$ in only two points of $C_\nu$.  This is shown in Figure~\ref{fig:TinPxI}.  Figure~\ref{fig:TinPxI} then also represents $P \times \{1\}$ in $T \times I$. In particular, in $T \times I$, the curves $C_\mu$ and $C_\lambda$ are respectively isotopic in $T\times \{1\}$ and $T \times\{-1\}$  to $\mu \times \{1\}$ and $\lambda \times \{-1\}$.  
As shown in Figure~\ref{fig:nuinTxI}, the curve $C_\nu$ however is isotopic in $\bdry(T \times I)$ to a union of the arcs $-\mu' \times \{1\}$ and $\lambda' \times \{-1\}$ and two arcs in $\bdry T \times I$. (For an essential simple closed curve $\alpha$ in $T$, we let $\alpha'$ denote the essential properly embedded arc in $T$ that is disjoint from $\alpha$.)

\begin{remark}
 Lashings, though without this name, were previously utilized in \cite{bowman} and \cite{BBL-handlebody} to create knots in handlebodies with cosmetic surgeries and bridge number greater than $1$.
Note that we have reversed the roles of $\mu, \lambda$ from the convention in \cite{BBL-handlebody}. In 
particular, our $K(p/q)$ would be $K(-q/p)$ in \cite{BBL-handlebody}.
\end{remark}

\begin{lemma}\label{lem:lashingsurgery}
	Let $K$ be the $p/q$--lashing of the pair of pants $P$.  Then $T$--framed surgery on $K$ is equivalent to $(+1, -q/(p+q), -p/(p+q))$--surgery on the link $C_\nu \cup C_\mu \cup C_\lambda = \bdry P$  with respect to the framing by $P$. 
\end{lemma}

Before proving the lemma, we first establish continued fraction conventions and recall surgery realizations of Dehn twists that facilitate a surgery description of a lashing.  From this surgery description,  Figure~\ref{fig:Kgammasugery}  illustrates the relevant surgery calculus from which the lemma follows.

We use the {\em continued fraction} convention
\[ [ r_n, r_{n-1}, \dots, r_2, r_1 ] = 
-\frac{1}{r_n - \frac{1}{r_{n-1} - \frac{1}{\dots-\frac{1}{r_2-\frac{1}{r_1}}}}}\]
used in \cite{KirbyMelvin-Dedekindsums}.   This corresponds to the continued fraction $[0, r_n,  \dots,  r_1]$ used in \cite{Saveliev-InvtsHomologyS3} and to the continued fraction $-[r_n,  \dots,  r_1] = [-r_n, \dots, -r_1]$ used in \cite{Baker-largevolumeBergeknots}.

If $L$ is a simple closed curve in an oriented surface $S$, then $\phi_L$ is a {\em positive Dehn twist} of $S$ along $L$.   If $S$ is embedded in an oriented $3$--manifold, let $L_+$ and $L_-$ denote the push-offs of $L$ to the positive and negative side of $S$, so that they cobound an annulus $\hatR$ intersecting $S$ in the curve $L$.  If $K$ is a simple closed curve in $S$, then $\phi_L^n(K)$ may be obtained as the image of $K$ after $(-1/n, 1/n)$--Dehn surgery on the link $L_+ \cup L_-$, as framed by $\hatR$.  

\begin{lemma}[{E.g.\ \cite[Lemma 2.1]{Baker-largevolumeBergeknots}}]
\label{lem:dehntwists}
	If $p/q = [r_n, \dots, r_1]$ with $n$ odd, then 
	\[ K(p/q) = \phi_\lambda^{r_n}  \circ \cdots \circ \phi_\mu^{r_{2}} \circ \phi_\lambda^{r_1}(\mu) \]
	where  $K(p/q)$ is the essential simple closed curve in $T$ homologous to $p\mu+q\lambda$ for some orientation.
\end{lemma}
\begin{proof}
With changes of notation and continued fraction convention, this is \cite[Lemma 2.1]{Baker-largevolumeBergeknots}.  However, to be explicit about our parametrizations, we sketch the proof here.

With respect to the oriented basis $\langle \mu, \lambda \rangle$ for $H_1(T)$, we have that $\phi_\mu = \begin{bmatrix}1 & 1 \\ 0 &1\end{bmatrix}$ and $\phi_\lambda = \begin{bmatrix}1 & 0 \\ -1 & 1\end{bmatrix}$.  Then $\phi_\lambda^{r_1}(\mu) = \mu  - r_1 \lambda$ and its slope is $-1/r_1 = [r_1]$.
Now say $a/b = [r_{n-2}, \dots, r_1]$ with $n$ odd.  Then  
\[ \phi_\lambda^{r_n} \circ \phi_\mu^{r_{n-1}} (a \mu + b \lambda) = (a+r_{n-1}b)\mu + (-r_{n}(a+r_{n-1}b)+b)\lambda\]
has slope 
\[ \frac{a+r_{n-1}b}{-r_{n}(a+r_{n-1}b)+b}=  - \frac{1}{r_{n}-\frac{1}{r_{n-1}+\tfrac{a}{b}}} = - \frac{1}{r_{n}-\frac{1}{r_{n-1}+\left(-\frac{1}{r_{n-2}  - \frac{1}{\dots-\frac{1}{r_2-\frac{1}{r_1}}}}\right)}} = [r_n, r_{n-1}, r_{n-2}, \dots, r_2, r_1].\]
Hence induction gives the desired result.
\end{proof}

\begin{figure}
	\centering
	\includegraphics[width=6in]{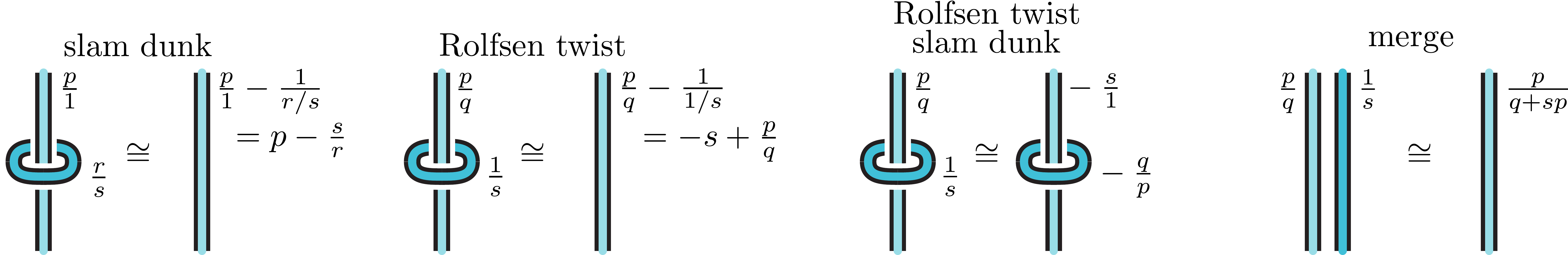}
	\caption{A few basic moves on surgery diagrams.  The merge move amalgamates two surgery curves cobounding an annulus along their $0$--framings.  }
	\label{fig:kirbycalcmoves}
\end{figure}

\begin{figure}
	\centering
	\includegraphics[width=6.5in]{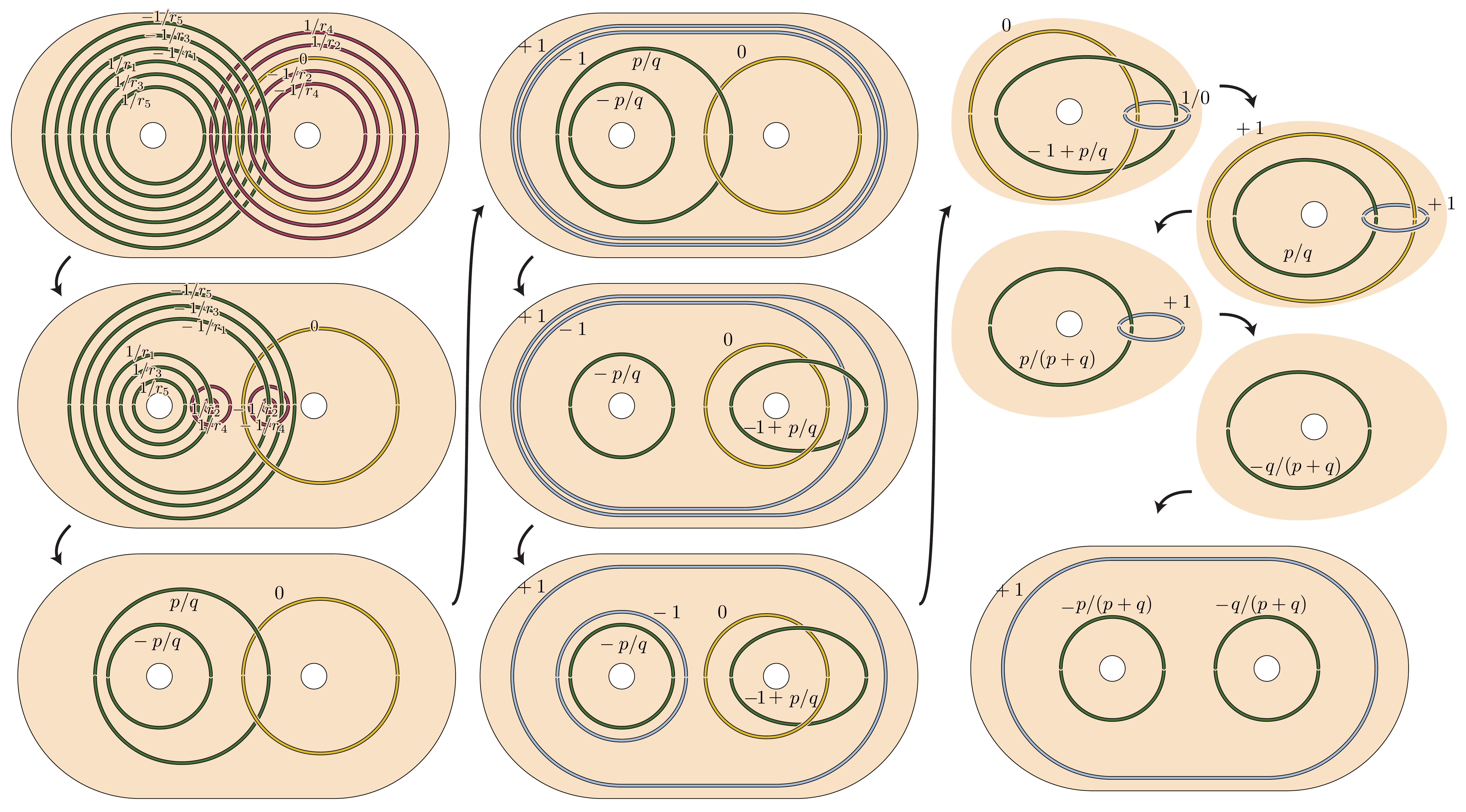}
	\caption{
		A sequence of equivalences of surgery diagrams.  In the top-left figure, the surgery coefficients on the two sets of concentric circles from outside to inside are $-1/r_5, -1/r_3, -1/r_1, 1/r_1, 1/r_3, 1/r_5$ and $1/r_4, 1/r_2, 0, -1/r_2, -1/r_4$. }
	\label{fig:Kgammasugery}
\end{figure}

\begin{proof}[Proof of Lemma~\ref{lem:lashingsurgery}]

	Take a continued fraction expansion $p/q = [r_n, \dots, r_1]$ with $n$ odd (as may always be done since $[r_n, \dots, r_1] = [r_n, \dots, r_1+1,1]$).
	Since the $p/q$--lashing $K=K(p/q)$ is an essential simple closed curve in $T$ with slope $p/q$, it may be expressed as the sequence of Dehn twists $\phi_\lambda^{r_n}  \circ \cdots \circ \phi_\mu^{r_{2}} \circ \phi_\lambda^{r_1}(\mu)$ of $\mu$ in $T$ along $\lambda$ and $\mu$ alternately  by Lemma~\ref{lem:dehntwists}. 
	 Then, as done in \cite[Section 3.1]{Baker-largevolumeBergeknots}, using the surgery realization of Dehn twists we may obtain a link in $P \times I$ formed around $\mu$ from nested pairs of push-offs of $\lambda$ and $\mu$ alternately on which the sequence of surgeries 
	\[-1/r_n, -1/r_{n-1}, \dots, -1/r_2, -1/r_1, \bigcdot, 1/r_1, 1/r_2, \dots, 1/r_{n-1}, 1/r_n\]
	in order from above $T$ down to below takes the central $\mu$ to $K$.    These surgeries are all framed with respect to $T$ which coincides with the blackboard framing of $\lambda$ and $\mu$ in Figure~\ref{fig:TinPxIprojection} (because $T$ is ``flat'' in this projection). The single $\bigcdot$ indicates that no surgery is being done on the central $\mu$.  This gives a surgery presentation of $K$.	For the $T$--framed surgery on $K$, we then replace the $\bigcdot$ with $0$ in the surgery sequence.  For $n=5$, a diagram in $P$ of the link with its surgery coefficients is shown in the upper left of Figure~\ref{fig:Kgammasugery}. 
	
	Now we perform a sequence of equivalences of Dehn surgery diagrams, transforming the surgery on the large nested link in $P\times I$ to a surgery on a link isotopic in $P \times I$ to $\bdry P$.  It begins with sliding the push-offs of $\mu$ across the $0$--framed central $\mu$ followed by slam-dunks and mergers (see Figure~\ref{fig:kirbycalcmoves} for terminology) to produce a $3$--component link that is close to, but not quite, what we want.  Further applications of blow-ups, mergers, and slides produce the desired link. This is explicitly illustrated in Figure~\ref{fig:Kgammasugery} for the situation with $n=5$.   The general case follows directly.
\end{proof}

\section{A construction of L-space knots via lashings and the Montesinos Trick}\label{sec:construction}
The center diagram of Figure~\ref{fig:almostalternating} gives a family of  diagrams of the unknot $J=J^{\alpha,m}$ in a $3$--bridge position indexed by the $3$--braid $\alpha$ and integer $m$.  On the right are the diagrammatic conventions for the braid $\alpha$
 and the numbered twist boxes.  The diagram on the left more clearly exhibits $J$ as an unknot since the braiding above and below the horizontal line are inverses except for an extra twist at the bottom and the offset capping of the six strands.  The diagram in the center is obtained from an isotopy that flips the $\alpha^{-1}$ braid box in the left figure along its vertical axis and rearranges the strands below the horizontal line.  Note that if $m\geq0$ and $a_i\geq 0$ for all $i$, then the center diagram is an {\em almost-alternating} unknot; that is, switching the circled crossing makes the diagram alternating.  

Figure~\ref{fig:almostalternating} further shows a triple of arcs $\calc = \calc^{\alpha,m}$ in the $3$--bridge sphere of $J$.
Naming the arcs of $\calc$ in Figure~\ref{fig:almostalternating} as $c_\nu, c_\mu, c_\lambda$ from left to right, Figure~\ref{fig:replacetoalternating} shows the result of  the $(+1, -q/(p+q), -p/(p+q))$--rational tangle replacements on them followed by 
an isotopy eliminating two pairs of crossings.  
We refer to the link resulting from this tangle replacement as $J^{\alpha,m}_{\calc}(+1, -q/(p+q), -p/(p+q))$. 
Observe that if $p/q = [r_n, \dots, r_2, r_1]$ 
then 
\[
\frac{-p}{p+q} = -\frac{1}{1+ q/p} =-\frac{1}{1-r_n - \frac{1}{-r_{n-1} -\frac{1}{\dots -\frac{1}{-r_2-\frac{1}{-r_1}}}}} = [1- r_n, -r_{n-1}, \dots, -r_2, -r_1] 
\]
and
\[
\frac{-q}{p+q} = -\frac{1}{1 + p/q} = -\frac{1}{1-\frac{1}{r_n - \frac{1}{\dots -\frac{1}{r_2-\frac{1}{r_1}}}}}= [1, r_n, \dots, r_2, r_1] =  [1,r_n, \dots, r_2,r_1+1,1]
\]
so that the corresponding rational tangles may indeed be expressed as shown in Figure~\ref{fig:replacetoalternating}.
 When $m\geq0$ and $a_i\geq 0$ for all $i$ and  $p/q>0$    so that the continued fraction expansion $p/q=[-b_n, b_{n-1}, \dots, b_2, -b_1]$ with $n$ odd has $b_j\geq 0$ for all $j$, this resulting reduced diagram is alternating. That is, $J^{\alpha,m}_{\calc}(+1, -q/(p+q), -p/(p+q))$ is an alternating link.

\begin{figure}
	\centering
	\includegraphics[width=6.5in]{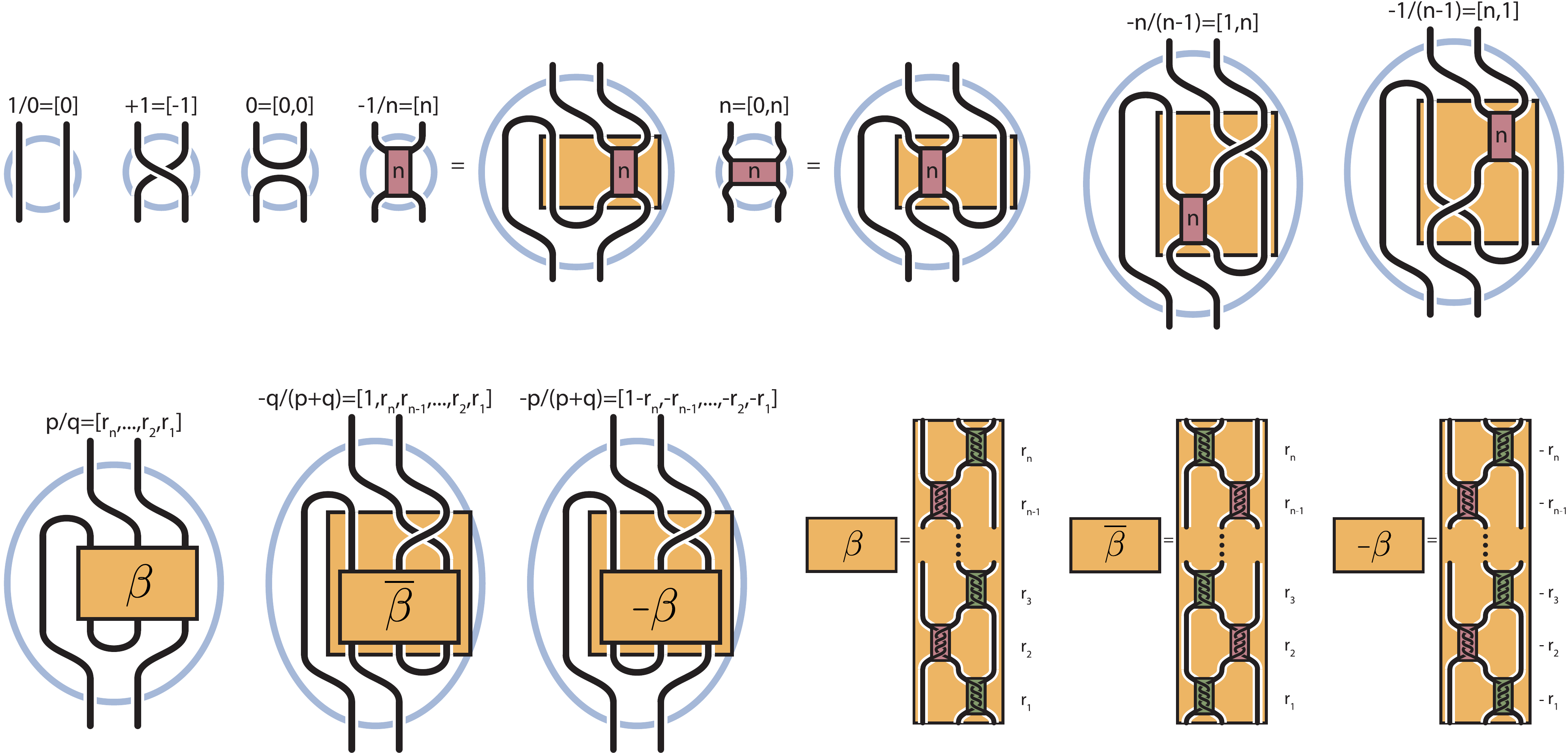}
	\caption{The top row shows how the convention for continued fractions and the corresponding rational tangles. (A twist box labelled $n$ means $n$ vertical right-handed twists when
$n>0$ and $|n|$ vertical left-handed twists when $n<0$.) The bottom row indicates the general forms for the related rational tangles of slopes $p/q$, $-p/(p+q)$, and $-q/(p+q)$ where $p/q = [r_n, \dots, r_2, r_1]$ for $n$ odd.  As illustrated, when $r_i \leq 0$ for $i$ odd and $r_i \geq 0$ for $i$ even, these tangles are alternating.}  
	\label{fig:tangleconventions}
\end{figure}

\begin{figure}
	\centering
	\includegraphics[width=5in]{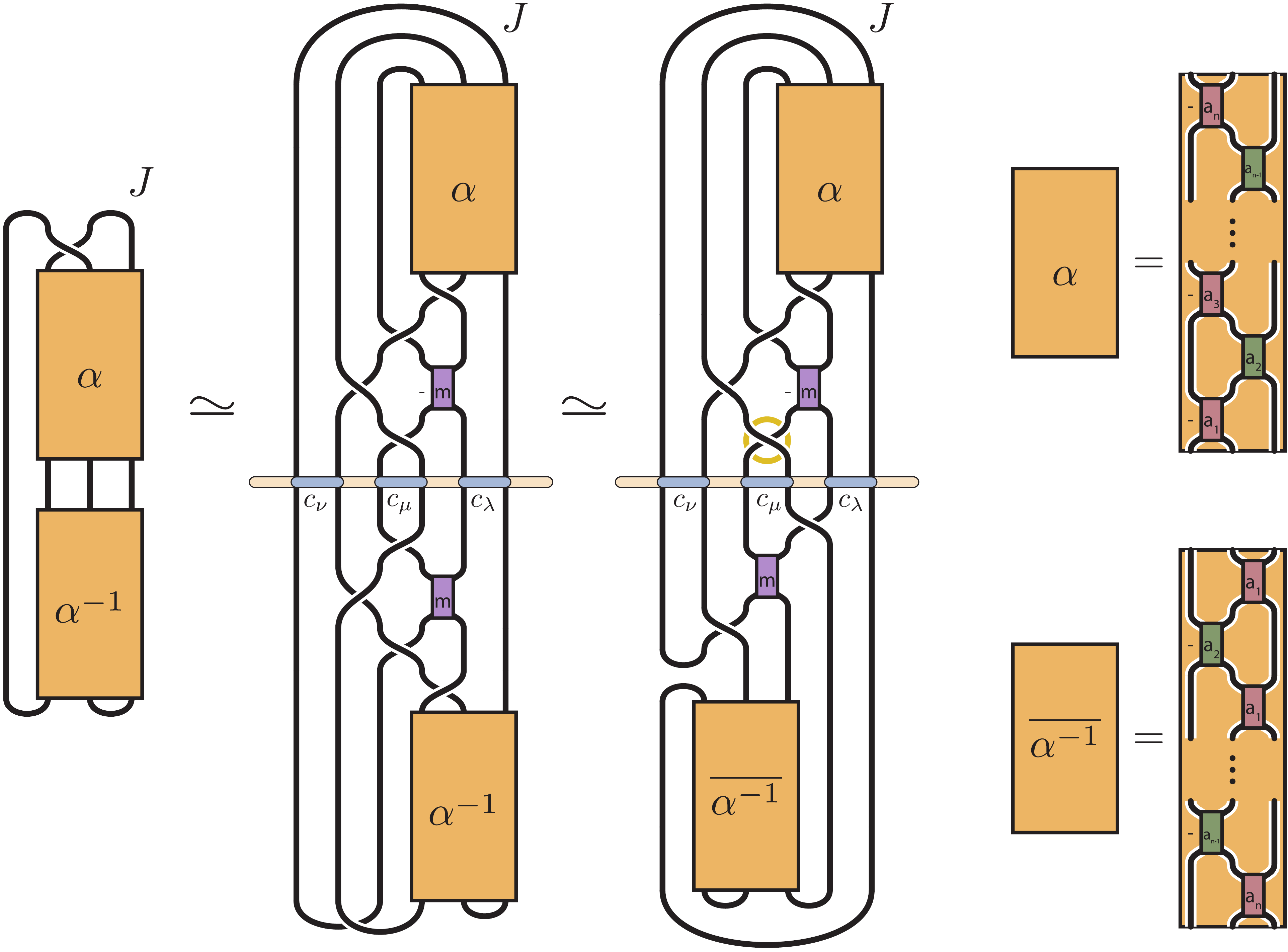}
	\caption{A diagram of the unknot $J$ in a $3$--bridge position so that it is almost-alternating when $m\geq0$ and $a_i\geq0$ for $i=1,2,\dots,n$ is shown in the middle. The dealternating crossing is circled.} 
	\label{fig:almostalternating}
\end{figure}

\begin{figure}
	\centering
	\includegraphics[width=6in]{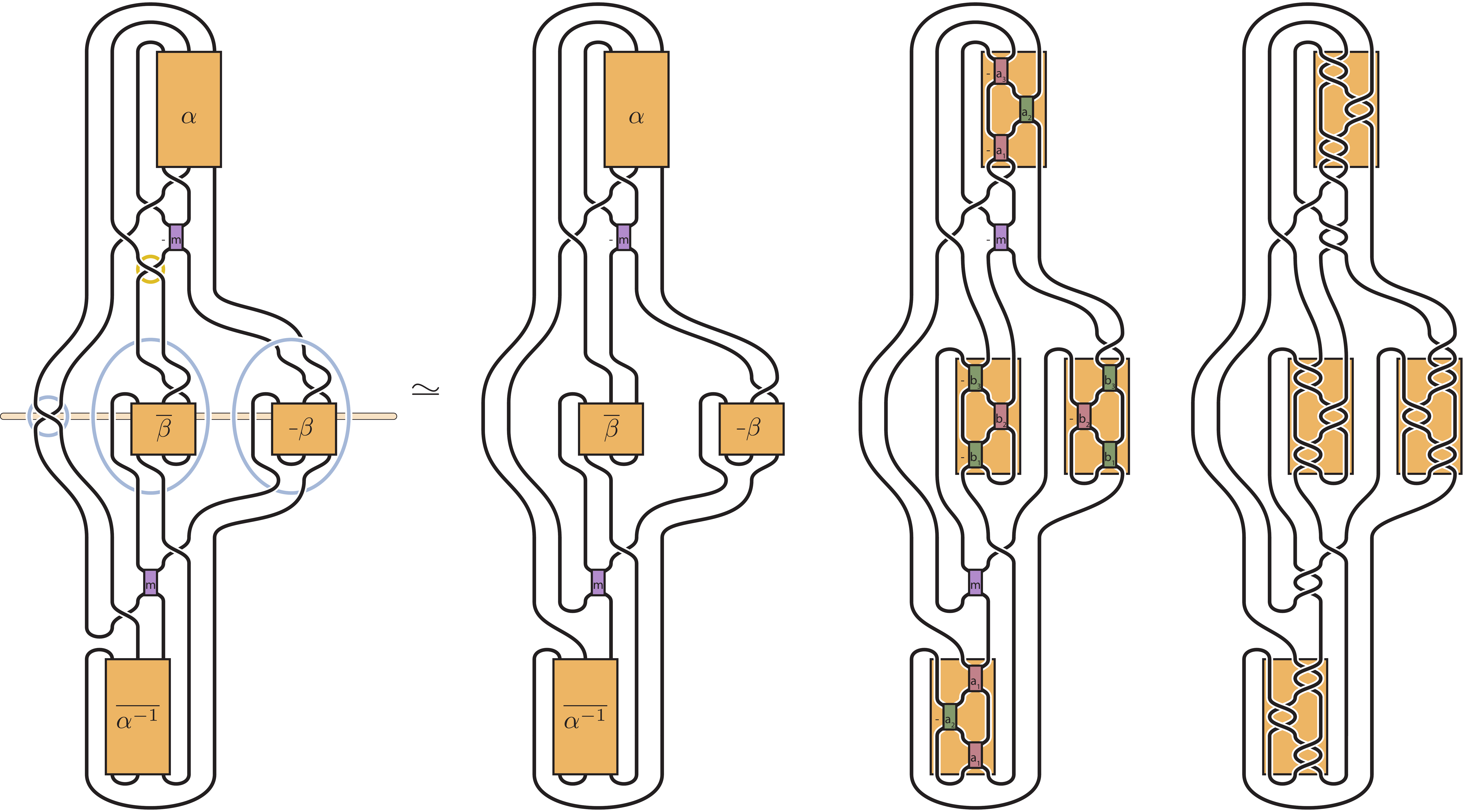}
	\caption{The rational tangle replacements are made along $\calc$ and then the diagram is simplified to be alternating.  Examples with $\alpha$ and $\beta$ of length $3$ and further with $m=a_1=a_2=a_3=b_1=b_2=b_3=2$ are shown. }
	\label{fig:replacetoalternating}
\end{figure}

\begin{remark} If $\delta$ is a $3$--braid, then we use $\overline{\delta}$ to denote the $3$--braid obtained
by rotating $\delta$ through the page along a vertical axis. We denote by $-\delta$ the $3$--braid 
obtained by changing all crossings of $\delta$, and by $\delta^{-1}$ the inverse in the braid group. 
\end{remark}

\begin{remark} The parameterization of the $3$--braid $\beta$ and the $3$--braid $\alpha$ are 
different because of the assignment of a rational tangle to $\beta$ and for efficiency in using $\alpha$
in the 
representation of the $3$--bridge knot of Figure~\ref{fig:almostalternating}. If $\beta$ corresponds
to $b_n,b_{n-2}, \dots, b_1$ and $\alpha$ corresponds to $a_n, a_{n-2}, \dots, a_1$ where $b_i=a_i$,
then $\alpha=\overline{\beta}$ as a $3$--braid.
\end{remark}

Since $J$ is the unknot in $S^3$, the double branched cover of $J$ is again $S^3$.   In this cover, the $3$--bridge sphere lifts to a genus $2$ Heegaard surface  $\Sigma$ for $S^3$ and the triple of arcs $\calc=\calc^{\alpha,m}$ lifts to a triple of curves $\calC = \calC^{\alpha,m}$ that divide $\Sigma$ into two pairs of pants.  Label the components of $\calC$ so that $C_*$ is the lift of $c_*$ for $*=\nu,\mu,\lambda$.
Let $P=P^{\alpha,m}$ be one of the pairs of pants in $\Sigma$ bounded by $\calC$ (identified with the pair of pants $P$ in Figure~\ref{fig:TinPxIprojection} by an orientation preserving homeomorphism so that the labeling of $\bdry P$ by $\calC$ agrees)  and set $Q = Q^{\alpha,m}$ to be the other one.  

\begin{theorem}\label{thm:lspaceknot}
For each $3$--braid $\alpha$ that is a product of positive powers of $\sigma_1^{-1}$ and $\sigma_2$, non-negative integer $m$, and non-negative  coprime integers $p$ and $q$, the $p/q$--lashing of $P^{\alpha,m}$ admits a longitudinal alternating surgery.  Consequently, the lashing is an L-space knot.
\end{theorem}

\begin{proof}
Let $K=K^{\alpha,m}(p/q)$ be the framed knot that is the $p/q$--lashing of $P^{\alpha,m}$ in $S^3$, and let $Z$ be the result of the framed surgery on $K$.
By Lemma~\ref{lem:lashingsurgery}, $Z$ is also obtained by $(+1, -q/(p+q), -p/(p+q))$--surgery on the link $C_\nu \cup C_\mu \cup C_\lambda = \calC =  \bdry P^{\alpha,m}$ with respect to the framing by $P^{\alpha,m}$.   By construction, the link $\calC$ arises as the double branched cover of the arcs $\calc$ and the framing of $\calC$ by $P^{\alpha,m}$ is the lift of the framing of $\calc$ by the bridge sphere.   Hence, by the Montesinos Trick, $(+1, -q/(p+q), -p/(p+q))$--surgery on $\calC$ is the double branched cover of $J^* = J^{\alpha,m}_\calc(+1, -q/(p+q), -p/(p+q))$, the result of the $(+1, -q/(p+q), -p/(p+q))$ rational tangle replacement on the arcs $\calc=\calc^{\alpha,m}$ on the unknot $J=J^{\alpha,m}$.  The result $J^*$ of this rational tangle replacement is shown in Figure~\ref{fig:replacetoalternating} and thus its double branched cover is the manifold $Z$. With the specified constraints on $p/q$, $\alpha$, and $m$, this resulting link $J^*$ has a connected alternating diagram, and therefore the link is non-split \cite{menasco}.  Thus $Z$ is an L-space \cite{OS-dbc}.  Since $K$ has a non-trivial surgery to the L-space $Z$, it is an L-space knot by definition.
\end{proof}

\section{A construction of asymmetric hyperbolic manifolds via lashings}
\subsection{Exteriors of lashings and basic twist families}\label{sec:twistfamily}
Consider a lashing $K$ of a pair of pants $P$ embedded in a $3$--manifold $Y$.  Let $H_P$ be the genus $2$ handlebody $P \times I$ and set $M = Y \cut H_P$ to be the exterior of this handlebody.  Since the lashing $K$ of $P$ is a core curve of $H_P$ (by \cite[Lemma~2.7]{BBL-handlebody} for example), 
there is a non-separating curve $\gamma$ in $\bdry H_P = \bdry M$  that bounds a non-separating disk $D_K$ in $H_P$ such that $H_P \cut D_K$ is a solid torus neighborhood of $K$.  
In particular, with the homeomorphism $H_P \cong T \times [-1, +1]$ and $K \subset T \times \{0\}$
(see Figure~\ref{fig:TinPxI}), the disk $D_K$ is a product disk $K' \times [-1,1]$ where $K'$ is a properly 
embedded, essential arc in $T$ that is disjoint from the curve $K$.
Hence the exterior of $K$,  $X_K = Y \cut \nbhd(K)$,  may be obtained as $M[\gamma]$, the result of attaching a $2$--handle to $M$ along $\gamma$:  $X_K = M[\gamma]$.

Let $T$ be the oriented once-punctured torus in $H_P$ that contains the lashings of $P$.  Let $L$ be another lashing of $P$ so that $K$ and $L$ transversally intersect once as simple closed curves in $T$.
 Let $L_+$ and $L_-$ be push-offs of $L$ to the positive and negative side of $T$ so that they cobound an annulus $\hatR$ in $H_P$ which intersects $T$ in the curve $L$ and which $K$ intersects once.  
 
 By construction, the link $\L = L_+ \cup L_- \cup K$ is contained in $H_P$, and $H_P$ may be identified with the closure of $\nbhd(K \cup \hatR)$. Observe that the link's exterior $X_\L = Y \cut \nbhd(\L)$  contains the pair of pants $R = \hatR \cap X_\L$  and we may further identify $X_\L \cut \nbhd(R)$ as the manifold $M$.
This gives a decomposition of $\bdry M$ as the union of two pairs of pants $R'$ and $R''$ that are push-offs  to each side of $R$ in $X_\L$ that are joined by the three annuli  of $\calA_{\L}=\bdry X_\L \cut \nbhd(R)$.  

Note that the core curves of the annuli of $\calA_{\L}$ are each non-separating in $\bdry M$ and they are collectively isotopic in $X_\L$ to $R \cap \bdry X_\L$, a longitude of each $L_+$ and $L_-$ and a meridian of $K$. In particular, the core curves of $\calA_{\L}$ are the curves $L \times \{1\}$, $L \times \{-1\}$, and $\bdry D_L$.  Here $D_L = L' \times [-1,1]$ where $L'$ is a properly embedded, essential arc in $T$ that is disjoint from the curve $L$.  Because $L$ intersects $K$ once, $L'$ may be taken to intersect $K$ once too; hence $\bdry D_L$ is isotopic in $X_\L$ to a meridian of $K$.

 \begin{defn}\label{twistfamily} [Basic twist family]
For each integer $n$, the knot $K^n = \phi_L^n(K)$ resulting from $n$ Dehn twists of $K$ along $L$ in $T$ is also a lashing of $P$.
We say  a collection of lashings such as $\{K^n\}$ is a {\em basic twist family} of lashings of $P$, cf.\ \cite[Definition 2.6]{BBL-handlebody}.
The lashing $K^n$ may be obtained as the image of $K$ upon $(-1/n, 1/n)$--Dehn surgery on the link $L_+ \cup L_-$ framed with respect to $\hatR$, and hence the exterior of $K^n$, $X^n=Y\cut \nbhd(K^n)$ may be obtained as $X_\L(-1/n, 1/n, \emptyset)$.  We let $L_+^n$ and $L_-^n$ be the core curves of the two fillings in $X^n$. Since the  link $L_+ \cup L_-$ is isotopic into $\bdry M = \bdry H_P$  with the framing given by $\hatR$,   the exterior $X^n$ 
may be regarded as $M[\gamma^n]$ where $\gamma^n$ is the result of Dehn twisting $\gamma$ along $L_+$ and $L_-$ in $\bdry M$.
\end{defn}

\begin{remark}
When $\bdry M=\bdry H_P$ is viewed from outside $H_P$,  $\gamma^n = \phi_{L_-}^{-n}(\phi_{L_+}^{n}(\gamma))$; when $\bdry M$ is viewed from outside $M$,  $\gamma^n= \phi_{L_-}^{n}(\phi_{L_+}^{-n}(\gamma))$.
\end{remark}

\begin{lemma}\label{lem:gammawithlargen}
For suitably large $n$, $\gamma^n$ minimally intersects each component of $\bdry P$  
 a distinct non-zero number of times. Furthermore, $\gamma^n$ intersects some 
component of $\bdry P$ an odd number of times.
\end{lemma}

\begin{proof}
We use the following criterion to determine minimal intersection number:
Let $\alpha,\beta$ be two non-trivial simple closed curves which are not isotopic in a surface $F$.
Then $\alpha$ and $\beta$ realize the minimal intersection number under isotopy if and only if 
there is no disk $D \subset F$ such that  $D \cap (\alpha \cup \beta)=a \cup b$ where $a$ is a
sub-arc of $\alpha$, $b$ is a subarc of $\beta$, and $a \cap b=\partial a=\partial b$ (i.e. $D$ is a
bigon guiding an isotopy of $\alpha$ that reduces its intersection with $\beta$). 

We denote the minimal intersection number between $\alpha$ and $\beta$ in $F$ by 
$\Delta_F(\alpha, \beta)$. Note that the above critierion implies that if $\alpha$ and $\beta$
intersect coherently (i.e. upon orienting, all intersections are the same sign), then $\alpha$
and $\beta$ intersect minimally. 

Recall that $H_P=P \times [-1,1]=T \times [-1,1]$.

Any pair of an essential simple closed curve $\tau$ and an essential properly embedded arc $\alpha'$ in $T$ may be isotoped to intersect coherently, so that $|\tau \cap \alpha'| = \Delta_T(\tau, \alpha)$ where $\alpha$ is the essential simple closed curve in $T$ that is disjoint from $\alpha'$.
 Then for the product disk  $D_{\alpha}=\alpha' \times [-1,1]$,  we have $\Delta_{\bdry H_P}(\tau \times \{1\}, \bdry D_\alpha) = \Delta_{\bdry H_P}(\tau \times \{-1\}, \bdry D_\alpha) =  \Delta_T(\tau, \alpha)$.   
Recall that $\bdry Q = \bdry P$ is the triple of curves $C_\mu, C_\lambda, C_\nu$.
Since $C_\mu = \mu \times \{1\}$, $C_\lambda = \lambda \times \{-1\}$ (Figure~\ref{fig:TinPxI}), 
and $\gamma^n$ is the boundary of $D_{K^n}=(K^n)' \times [-1,1]$, we therefore have that $\Delta_{\bdry H_P}(C_\mu ,\gamma^n) = \Delta_T(\mu, K^n)$ and  $\Delta_{\bdry H_P}(C_\lambda, \gamma^n) = \Delta_T(\lambda, K^n)$.

\begin{figure}
	\centering
	\includegraphics[height=1.5in]{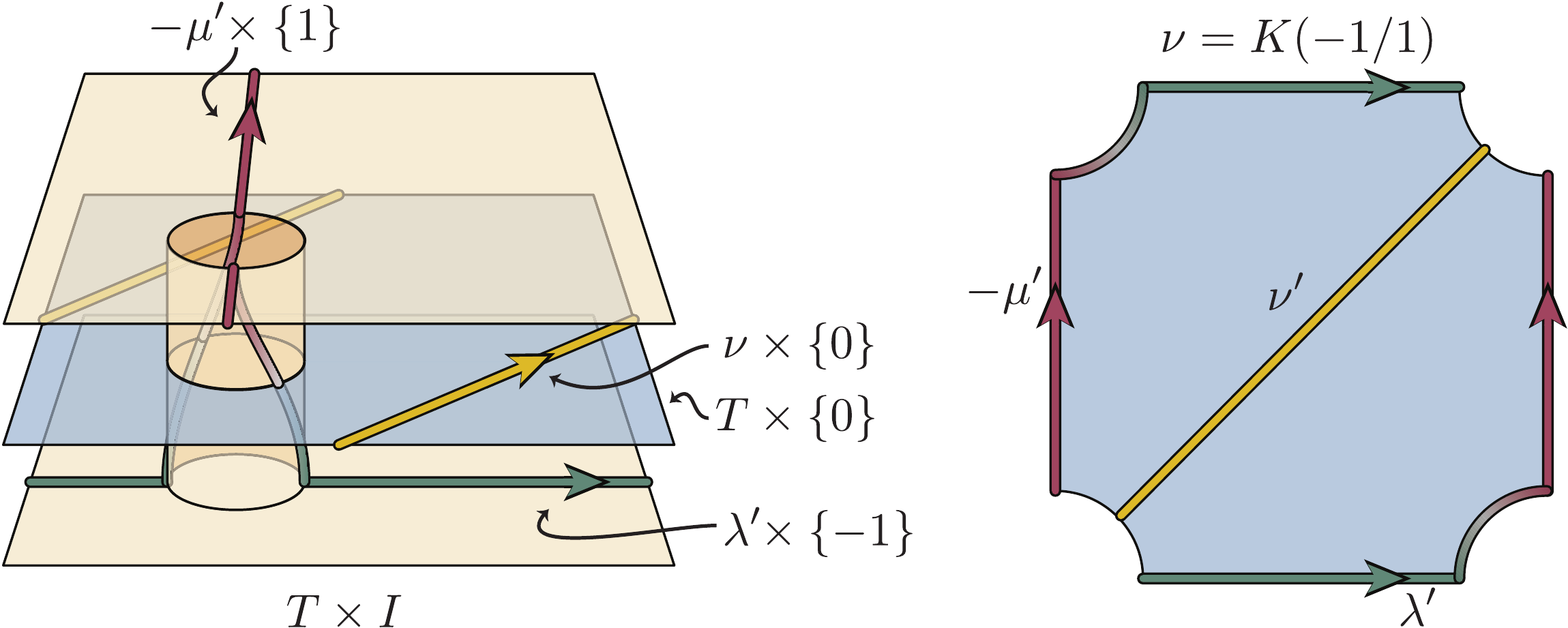}
	\caption{(Left) The oriented curve $C_\nu$ is a union of the arcs $-\mu' \times \{1\}$ and $\lambda' \times \{-1\}$ and two arcs in $\bdry T \times I$. It is isotopic to $\nu = \nu \times \{0\}$.  (Right) The projection of the curve $C_\nu$ to $T$ is isotopic to the curve $\nu = K(-1/1)$, and disjoint from the arc $\nu'$.}
	\label{fig:nuinTxI}
\end{figure}

\begin{figure}
	\centering
	\includegraphics[height=1.5in]{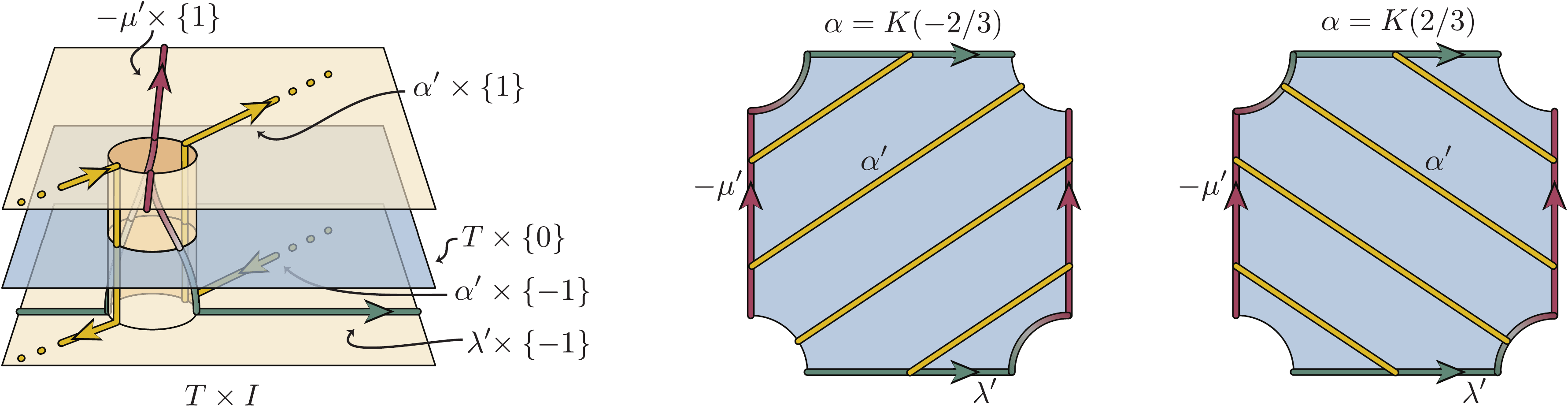}
	\caption{ (Left) The curve $\bdry D_\alpha$  is the union of the two arcs $\alpha' \times \{\pm 1\}$ and two vertical arcs in $\bdry T \times I$ since $D_\alpha = \alpha' \times I$. In the cylinder $\bdry T \times I$, the curves $C_\nu$ and $\bdry D_\alpha$ are either disjoint or intersect twice.  With $\alpha$ of negative slope as shown, they are disjoint.
		(Center, Right)	
	The projection of $C_\nu$ and $\bdry D_\alpha$ to $T$ where $\alpha = K(a_1/a_2)$ is illustrated for negative slopes with $a_1/a_2=-2/3$ and for positive slopes with $a_1/a_2 = 2/3$.}
	\label{fig:Cnu123}
\end{figure}

The curve $C_\nu$, however, is the union of arcs $\mu' \times \{1\}$, $\lambda' \times \{-1\}$ and a pair of arcs in $(\bdry T) \times [-1,1]$ so that $C_\nu$ projects to an embedded curve in $T = T \times \{0\}$ that is isotopic to $\nu$ as illustrated in Figure~\ref{fig:nuinTxI}. Figure~\ref{fig:Cnu123} indicates  minimal intersection
presentations between $C_\nu$ and $\bdry D_{\alpha}$ using this projection. 
If $[\alpha]=a_1 [\mu] + a_2 [\lambda]$
in $H_1(T)$  so that $\alpha = K(a_1/a_2)$, then Figure~\ref{fig:Cnu123} (Right) corresponds to the minimal presentation when $a_1,a_2$ are
of the same sign, and Figure~\ref{fig:Cnu123} (Center) corresponds to when $a_1,a_2$ are of opposite sign. 
Note that $\bdry D_\alpha$ contains two 
copies of $\alpha'$ oriented 
oppositely. One checks that there are no bigons of intersection between $C_\nu$ and 
$\partial D_{\alpha}$, thereby verifying the minimal intersection. In Figure~\ref{fig:Cnu123} (Right), this follows
because all intersections are of the same sign on $\bdry H_P$. In Figure~\ref{fig:Cnu123} (Center), one checks
directly that the criterion holds. Figure~\ref{fig:Cnu123} (Right) then shows that 
$\Delta_{\bdry H_P}(C_\nu, \bdry D_\alpha)=|a_1|+|a_2|$ when $a_1,a_2$  are of the same sign.
Figure~\ref{fig:Cnu123} (Center) shows that 
$\Delta_{\bdry H_P}(C_\nu, \bdry D_\alpha)=|a_1|+|a_2|-2$ when $a_1,a_2$  are of opposite sign.

As curves in $T$, we may choose orientations on $K$ and $L$ so that $L \cdot K = +1$ with $[K]=p[\mu]+q[\lambda]$ and $[L] = r[\mu] + s[\lambda]$  where $rq-ps=1$.  Then $[K^n] = [K]+n[L] = (p+rn)[\mu] + (q+sn)[\lambda]$.  Hence 
\begin{itemize}
\item $\Delta_{\bdry H_P}(C_\mu ,\gamma^n) = \Delta_T(\mu, K^n) = |q+sn|$,
\item $\Delta_{\bdry H_P}(C_\lambda, \gamma^n) = \Delta_T(\lambda, K^n) = |p+rn|$, 
\item $\Delta_{\bdry H_P}(C_\nu, \gamma^n) = |p+rn|+|q+sn|$ when $q+sn, p+rn$ are of the same sign, and
\item $\Delta_{\bdry H_P}(C_\nu, \gamma^n) = |p+rn|+|q+sn|-2$ when $q+sn, p+rn$ are of opposite sign.
\end{itemize}

The fact that $(p,q),(r,s)$ represent different slopes of $T$ guarantees that $|p +rn| \neq |q+sn|$
and $|p +rn|+|q+sn|-2 > \max(|p +rn|,|q+sn|)$ for all but finitely many $n$,
thereby verifying 
that $\gamma^n$ intersects each component of $\bdry P$ a different number of times.
Finally, since  $[K^n] =  (p+rn)[\mu] + (q+sn)[\lambda]$ is the homology class of a simple closed curve in $T$, the integers $|p+rn|$ and $|q+sn|$ must be relatively prime, and hence one must be odd.
\end{proof}

\begin{lemma}
	\label{lem:bdryQintersectcalA}
	Unless $L$ is isotopic to $\mu$ or $\lambda$  in $T$, each component of $\bdry Q$ is intersected by some core curve of the annuli $\calA_{\L}$.
\end{lemma}

\begin{proof}
	Recall that the core curves of $\calA_{\L}$ are the curves $L \times \{1\}$, $L \times \{-1\}$, and $\bdry D_L$ while $\bdry Q$ is the triple of curves $C_\nu, C_\mu, C_\lambda$  which are isotopic in $\bdry H_P$ to a union of the arcs $-\mu' \times \{1\}$ and $\lambda' \times \{-1\}$ and two arcs in $\bdry T \times I$ (as shown in Figure~\ref{fig:nuinTxI}), $\mu \times \{1\}$, and $\lambda \times \{+1\}$ respectively.  Continuing as in the proof of Lemma~\ref{lem:gammawithlargen}, if $[L]=r[\mu]+s[\lambda]$ in $T$, then we may calculate:
	\begin{itemize}
		\item $\Delta_{\bdry H_P}(C_\mu, L\times\{1\}) = \Delta_{\bdry H_P}(\mu\times\{1\}, L\times\{1\})  = \Delta_T(\mu,L) = |s|$,
		\item $\Delta_{\bdry H_P}(C_\lambda, L\times\{-1\}) = \Delta_{\bdry H_P}(\lambda\times\{-1\}, L\times\{-1\})  = \Delta_T(\lambda,L) = |r|$,
		\item $\Delta_{\bdry H_P}(C_\nu,L\times\{1\}) = \Delta_T(\mu', L) = \Delta_T(\mu,L) = |s|$, and
		\item $\Delta_{\bdry H_P}(C_\nu,L\times\{-1\}) = \Delta_T(\lambda', L) = \Delta_T(\lambda,L) = |r|$.
\end{itemize}
Hence we have the desired conclusion as long as both $r$ and $s$ are non-zero,  i.e.\ as long as $L$ is not $\mu$ or $\lambda$.
\end{proof}

\subsection{Hyperbolicity of $\L$ and uniqueness of $R$, given the simplicity of $M$}

For this section, we continue to assume that we have a basic twist family $\{K^n\}$ corresponding
to a link $\L$ in $H_P$.  

\begin{lemma}\label{lem:parallelinM}
If $F'$ is a $\bdry$--parallel disk or annulus in $M$ intersecting each component of $\calA_{\L}$ in a collection of spanning arcs, then each collection has an even number of arcs.   
\end{lemma}

\begin{proof}
Regarding each component of $\calA_{\L}$ as represented by a core curve, we may then regard $\bdry F'$ as transverse to $\calA_{\L}$.  Since $F'$ is $\bdry$--parallel, $\bdry F'$ bounds a surface in $\bdry M$.  Elementary mod $2$ intersection theory then gives the result.
\end{proof}

\begin{defn}[Simple $3$--manifold]
A {\em simple} $3$--manifold is irreducible, $\bdry$--irreducible, atoroidal, and acylindrical. 
In particular, a manifold is simple if and only if every $2$--sphere bounds a ball, every properly
embedded disk is boundary parallel,
and every properly embedded, incompressible annulus or torus is boundary parallel.
\end{defn}

A link in a 3-manifold is hyperbolic if its complement admits a complete hyperbolic metric of finite volume (and this complement is said to be hyperbolic). 

\begin{proposition}\label{prop:hypL}
If $M$ is simple, then $\L$ is hyperbolic.
\end{proposition}
\begin{proof}
By Geometrization for Haken manifolds \cite{Thurston-hakengeom}, to show $\L$ is hyperbolic it is sufficient to show that if $F$ is a properly embedded sphere, disk, annulus, or torus in $X_\L$ then $F$ is not essential --- i.e.\ $F$ is a sphere that bounds a ball or $F$ is a disk, annulus, or torus
which is either compressible or boundary parallel. 

So assume $F$ is an essential properly embedded sphere, disk, annulus, or torus in $X_\L$, isotoped to intersect $R$ minimally.   
If $F\cap R = \emptyset$, then $(F, \bdry F) \subset (M, \calA_{\L})$.  Since $M$ is simple (by hypothesis), $F$ must then be $\bdry$--parallel in $M$ and hence must be a disk or annulus.  However since any disk or annulus in $\bdry M$ bounded by curves in $\calA_{\L}$ must themselves be contained in $\calA_{\L}$, $F$ would have to be $\bdry$--parallel in $X_\L$, a contradiction.  
Therefore $F \cap R \neq \emptyset$.
Since $F$ is essential, any simple closed curve of $F \cap R$ must be $\bdry$--parallel in $R$ and any arc of $F \cap R$ must either separate two components of $\bdry R$ or connect two components of $\bdry R$.

If $F \cap R$ contains simple closed curves that bound disks in $F$, then an innermost such bounds a disk in $F$ with interior disjoint from $R$.  Since this curve must be $\bdry$--parallel in $R$ and the boundary components of $R$ are all essential curves in $\bdry M$, then $M$ is boundary reducible.  But then $M$ is not simple, a contradiction.

If $F \cap R$ contains arcs that are $\bdry$--parallel in $F$, then an outermost such bounds a disk in $F$ with interior disjoint from $R$.  Thus this disk is properly embedded in $M$ and its boundary is a curve in $\bdry M$ that intersects the three annuli $\calA_{\L}$ of $\bdry X_\L \cut R$  in a single spanning arc.  Since the disk must be $\bdry$--parallel in $M$ because $M$ is simple, this contradicts Lemma~\ref{lem:parallelinM}.

The previous two paragraphs show that $F$ cannot be a sphere or a disk.  Hence $F$ is either an annulus or a torus.

If $F$ is an annulus, then  $F \cap R$ is either a collection of spanning arcs in $F$ or a collection of essential simple closed curves.  In the former case, each component of $F \cut R$ is a properly embedded disk in $M$ that  crosses $\calA_{\L}$ twice.  Yet this implies that the arcs of $F \cap R$ must be $\bdry$--parallel in $R$, a contradiction.  In the latter case, an outermost component in $F$ cuts off a subannulus $F'$ in $F$ with interior disjoint from $R$ and a subannulus $R'$ of $R$.  Since together the annulus $F' \cup R'$  may be nudged off $R$ to be a properly embedded annulus in $M$, due to the simplicity of $M$ the annulus $F' \cup R'$ must be $\bdry$--parallel.  This parallelism then guides an isotopy of $F$ through $R'$ that reduces $|F \cap R|$, contradicting the assumed minimality.

If $F$ is a torus, then the components of $F \cut R$ are all annuli.  Since these annuli are all properly embedded in $M$ and $M$ is simple, they all must be $\bdry$--parallel.   Hence their boundary components are all isotopic in $R$ to the same component of $\bdry R$.  Due to the minimality of $|F \cap R|$ they cannot be $\bdry$--parallel into one side of $R$ but rather must be $\bdry$--parallel across a component of $\calA_{\L}$. Thus any component $F \cut R$ is an annulus that runs between opposite sides of $R$ and may be joined together by a subannulus of $R$ (or just a single curve of $F \cap R$) to form a torus that is $\bdry$--parallel in $X_\L$.  Since $F$ is an embedded closed compact surface,  $F \cap R$ must in fact be a single curve and hence $F$ itself is a $\bdry$--parallel torus.  Yet this means $F$ is not essential.
\end{proof}

\begin{proposition}\label{prop:uniquepop}
Assume $M$ is simple.  
If $R'$ is a properly embedded pair of pants in $X_\L$ with a component of $\bdry R'$ in each component of $\bdry X_\L$, then either 
\begin{enumerate}
\item there is a properly embedded, pair of pants in $M$ that is incompressible and not $\bdry$--parallel whose boundary is the set of core curves of $\calA_{\L}$, or
\item $R'$ is isotopic to $R$.
\end{enumerate}
\end{proposition}

\begin{proof}
The pair of pants $R'$ may be isotoped to intersect $R$ transversally and minimally. Then $R' \cut R$ is a collection of properly embedded surfaces in $M = X_\L\cut \nbhd(R)$.

Assume $R \cap R' =\emptyset$. Thus $R'$ is a properly embedded, pair of pants in $M$ whose boundary is the core curves of $\calA_{\L}$. 
If $R'$ is neither compressible nor $\bdry$--parallel, then we have our first conclusion.
If $R'$ is compressible in $M$, then due to the simplicity of $M$ some component of $\bdry R'$ bounds a disk in $\bdry M$; but this is contrary to the cores of $\calA_{\L}$ being non-separating curves in $\bdry M$.  If $R'$ is $\bdry$--parallel, then $R'$ is isotopic to $R$ giving our second conclusion.

So assume $R \cap R' \neq \emptyset$.
Following the arguments of Proposition~\ref{prop:hypL} with $R'$ in the stead of $F$, $R \cap R'$ is a non-empty set of arcs, essential in each $R$ and $R'$ and with no two parallel. (Use the argument for when $F$ is an annulus to show $R \cap R'$ contains no simple closed curves that are essential in both $R$ and $R'$.)  Consequently, $R \cap R'$ is  one, two, or three essential arcs in $R'$ where at most one is separating. The possible configurations are illustrated in Figure~\ref{fig:RprimecutR}.  In order to not violate Lemma~\ref{lem:parallelinM} each component of $R' \cut R$ must be incident to each component of $\bdry R'$ an even number of times.  (This is because each component of $\bdry R'$ is in its own component of $\bdry X_\L$ and each component of $(\bdry R')\cut R$ is a spanning arc of $\calA_{\L}$.)  An examination of Figure~\ref{fig:RprimecutR} shows this does not occur.
\end{proof}

\begin{figure}
	\centering
	\includegraphics[width=3in]{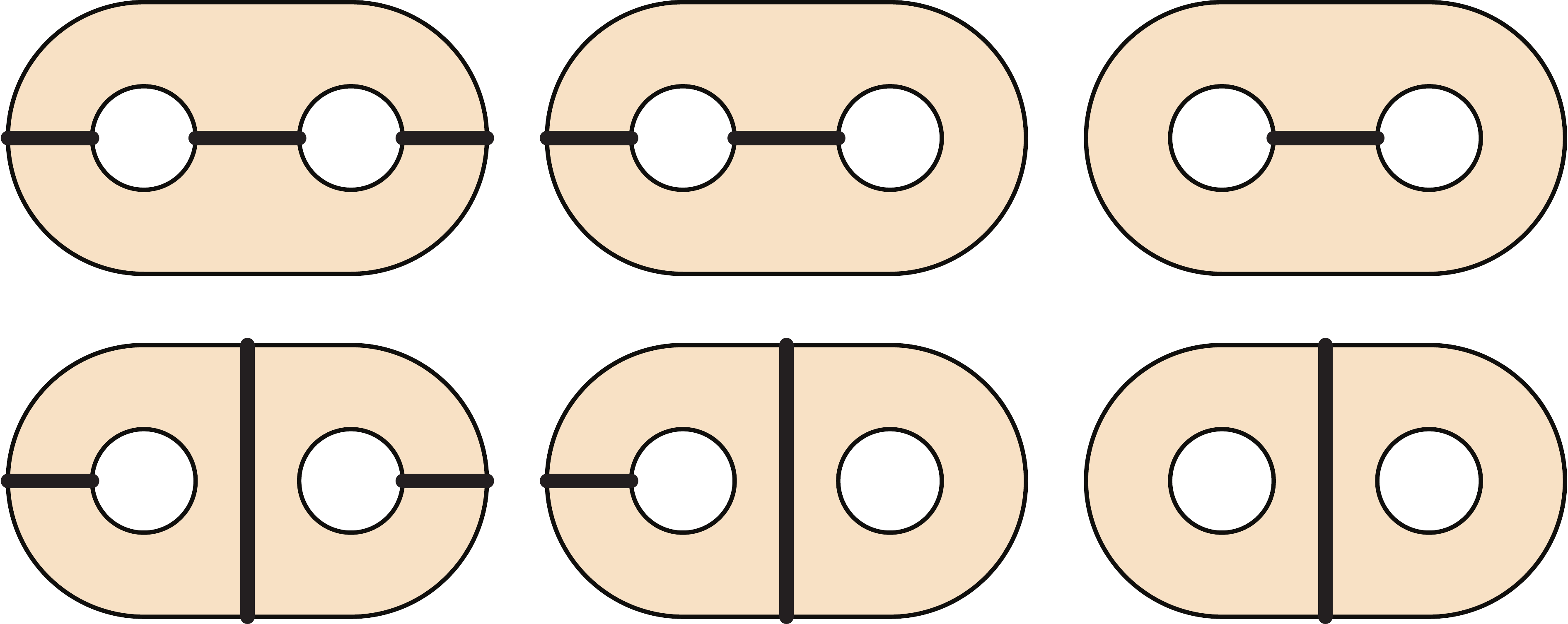}
	\caption{Enumerated are the six possible configurations, up to homeomorphism, of the non-empty set of arcs $R \cap R'$ in the pair of pants $R'$ so that the arcs are essential in $R'$ and no two are parallel.}
	\label{fig:RprimecutR}
\end{figure}

\subsection{Twisting $K$, hyperbolicity and short geodesics.}

Let $\{K^n\}$ be a {\em basic twist family} of lashings of $P$, where $K^n$ is the $n$--fold Dehn twist in
$T$ of $K$ along $L$, see Definition~\ref{twistfamily}.  Recall that $\L$ is the link $L_+ \cup L_- \cup K$, and let  $\mathring{X}_K,\mathring{X}_\L$ denote the complements of $K,\L$ in $Y$. The complement, $\mathring{X}^n$, of the lashing $K^n$ in $Y$ may then be obtained as the $(-1/n, 1/n)$-Dehn surgery on the link $(L_+,L_-)$ in $\mathring{X}_K$. Let $L_{+}^n,L_{-}^n$ be the closed curves in 
$\mathring{X}^n$ gotten
by taking the cores of the attached solid tori in that Dehn surgery. 

\begin{lemma} \label{lem:shortestgeodesics}
	Assume $\mathring{X}_\L$ is hyperbolic. For large $n$, $\mathring{X}^n$ is hyperbolic and $L_+^n$ and $L_-^n$ are the two shortest geodesics in $\mathring{X}^n$.
\end{lemma}

\begin{proof}
	This is a well-known strong version of Thurston's Hyperbolic Dehn Surgery Theorem and follows from 
the proof thereof by Benedetti and Petronio, see \cite[Theorem E.5.1]{BP}.  
We apply this to our sequence of manifolds $\{\mathring{X}^n\}$. Their proof shows that
	for large $n$, there is a sequence of structures $\{z_n\}$ in $\Def(\mathring{X}_\L)$ (the incomplete hyperbolic structures on $\mathring{X}_\L$ in a neighborhood of the complete structure, $z_0$) that complete to $\mathring{X}^n$, and 
	$z_n \to z_0$ in $\Def(\mathring{X}_\L)$. By Proposition E.6.29, the corresponding complete hyperbolic structures
	$\mathring{X}^n$ converge to the complete structure on $\mathring{X}_\L$ in the geometric topology. Furthermore,
	by  Neumann-Zagier \cite{NZ}, the cores of the attached solid tori, $L_+^n$ and $L_-^n$, are geodesics whose lengths go to $0$ in $\mathring{X}^n$ as $n \to \infty$.

Because $\{\mathring{X}^n\}$ converges in the geometric topology to the complete structure on $\mathring{X}_\L$, 
\cite[Theorem E.2.4]{BP} says that, for small enough $\epsilon>0$ and large enough $n$,  
the $\epsilon$--thin part of $\mathring{X}^n$ must  be the tubular neighborhoods of two simple geodesics
along with a cusp neighborhood. 
By \cite[Proposition D.3.11]{BP}, the core curves are the unique geodesics in these tubular neighborhoods.
For large enough $n$, the geodesics $L_+^n$ and $L_-^n$ have lengths less than $\epsilon$ and
consequently lie in the  $\epsilon$--thin part of $X^n$. Thus $L_+^n$ and $L_-^n$ must be 
the core curves of these
tubes and the shortest geodesics in $\mathring{X}^n$.
\end{proof}

\subsection{Asymmetric hyperbolic lashings}
\begin{theorem}\label{thm:mainasymmetry}
Consider a connected, closed, compact, oriented $3$--manifold $Y$ that contains an embedded  genus $2$ Heegaard surface decomposed into two pairs of pants $P$ and $Q$ such that $P \cap Q = \bdry P = \bdry Q$.  Assume the following:
\begin{enumerate}
\item \label{item:simpleM} The manifold $M = Y \cut H_P$ is simple.
\item \label{item:essentialQ} The pair of pants $Q$ is a properly embedded, incompressible, boundary incompressible, separating surface in $M$, dividing $M$ into handlebodies $H_+$ and $H_-$.
\item \label{item:nothomeo} The pairs $(H_+, Q)$ and $(H_-,Q)$ are not homeomorphic.
\item \label{item:uniqueQ} Any properly embedded pair of pants in $M$ is either compressible, $\bdry$--parallel, isotopic to $Q$, or non-separating and can be properly isotoped in $M$ to be disjoint from some component of $\bdry Q$. 
\end{enumerate}
Let $\{K^n\}$ be a basic twist family of lashings of $P$ in which the twisting curve $L$ is neither $\mu$ nor $\lambda$.  Then for sufficiently large $n$,  $K^n$ is a hyperbolic knot with asymmetric complement.
\end{theorem}

Observe that $\calC = \bdry Q = \bdry P$ also decomposes $\bdry M$ into two pairs of pants $P_+ \subset \bdry H_+$ and $P_- \subset \bdry H_-$ that are both isotopic to $P$ in $H_P$.
Schematically, the three pairs of pants $P_+,P_-,Q$ decompose $Y$ as in Figure~\ref{fig:manifolddecomposition}.  

\begin{remark}
	Our proof of Theorem~\ref{thm:mainasymmetry} uses the work of Oertel on homeomorphisms of handlebodies \cite{oertelhandlebodies}.  Nonetheless, we expect the theorem to continue to hold when  the submanifolds $H_+$ and $H_-$ are not necessarily handlebodies so that the genus $2$ surface $P \cup Q$ is not necessarily a Heegaard surface.
\end{remark}

\begin{figure}
\includegraphics[width=6in]{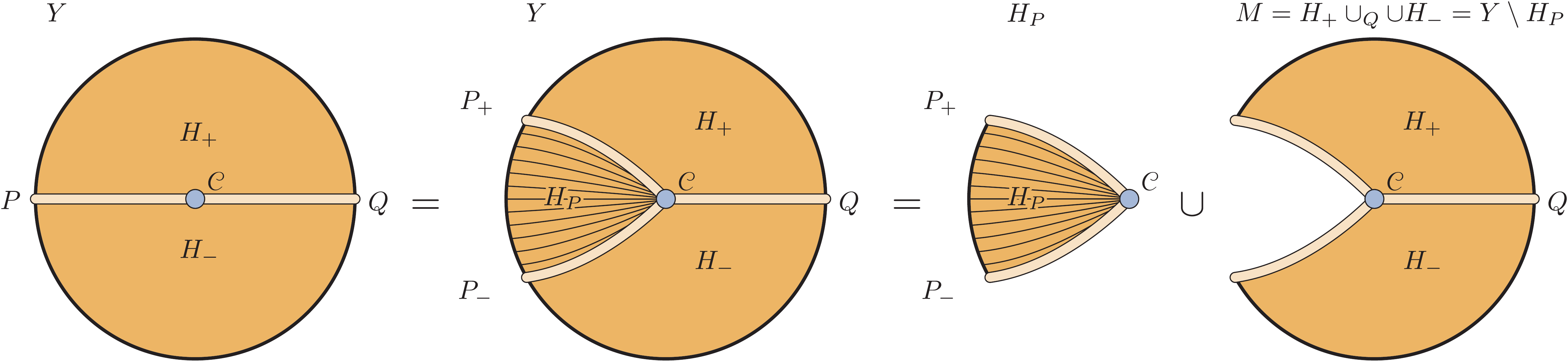}
\caption{The schematic decomposition of $Y$ along  the three pairs of pants $P_+,P_-,Q$.}
\label{fig:manifolddecomposition}
\end{figure}

\begin{proof}[Proof of \ref{thm:mainasymmetry}]
Let $K=K^0$ and $L$ be the pair of lashings such that the twist family $\{K^n\}$ is obtained by twisting $K$ along $L$, and use the notation of Section~\ref{sec:twistfamily}. 
In particular, $K$ with two push-offs of $L$ form the link $\L=K \cup L_+ \cup L_-$ in $Y$ which has exterior $X_\L = Y \cut \nbhd(\L)$ and complement $\mathring{X}_\L=Y-\L$.
We denote the exterior of $K^n$ in $Y$ as $X^n=Y \cut \nbhd(K^n)$, and the complement of $K^n$
in $Y$ as $\mathring{X}^n=Y-K^n$.

Given that $M$ is simple by (\ref{item:simpleM}), Proposition~\ref{prop:hypL} implies that $\mathring{X}_\L$ is hyperbolic.  Then for $n \gg 0$, $\mathring{X}^n$, the interior of $X^n = M[\gamma^n]$, is hyperbolic and the core curves of the fillings  $L_+^n$ and $L_-^n$ are the two shortest simple geodesics in $\mathring{X}^n$ by Lemma~\ref{lem:shortestgeodesics}.  Also, for suitably large $n$, $\gamma^n$ minimally intersects each component of $\bdry Q=\bdry P$ a distinct non-zero number of times and at least one component an odd number of times by Lemma~\ref{lem:gammawithlargen}.

For economy of notation, set $K=K^n$ where we choose $n$ large enough that 
\begin{enumerate}
\item[(a)] $\mathring{X}_K =\mathring{X}^n$ is hyperbolic, 
\item[(b)] the curves $L_+ = L_+^n$ and $L_- = L_-^n$ are the two shortest geodesics in $\mathring{X}_K$, and 
\item[(c)] $\gamma=\gamma^n$ minimally intersects each component of $\bdry Q$ a distinct non-zero number of times. It intersects at least one component of $\bdry Q$ an odd number of times. 
\end{enumerate}
We show that $K$ is asymmetric, verifying the Theorem.

Let $h$ be a diffeomorphism of $\mathring{X}_K$. Since $\mathring{X}_K$ is a (complete, finite volume) hyperbolic manifold by (a), any element of the mapping class group is uniquely represented by an isometry
(see the discussion preceding Theorem 6.2 of \cite{Bon}).   So we may begin by taking $h$ to be an isometry of $\mathring{X}_K$. Our goal is then to show that $h$ is isotopic to the identity diffeomorphism, implying that the isometry group of $\mathring{X}_K$ must be trivial and the hyperbolic knot $K$ is asymmetric.

Given that $h$ is an isometry, since $L_+$ and $L_-$ are the two shortest geodesics in $\mathring{X}_K$ by (c), $h(L_+ \cup L_-)=L_+ \cup L_-$. Therefore $h$ restricts to a diffeomorphism $h_\L$ on the link exterior $X_{\L}$ that preserves the original boundary torus $T_K$ and the pair of new boundary tori $T_+ \cup T_-$ obtained by drilling the geodesics $L_+ \cup L_-$ (for example, the
proof of Lemma~\ref{lem:shortestgeodesics} shows that the thick part of the complement, which is preserved by an isometry, is the exterior $X_\L$). We will now think of $h$ as a diffeomorphism on $X_K$ and show that it is isotopic to the identity. This implies that the original diffeomorphism
on $\mathring{X}_K$ is isotopic to the identity.

By Lemma~\ref{lem:bdryQintersectcalA}, each component of $\bdry Q = \bdry P$ is intersected by some core curve of $\calA_{\L}$.
Hence, with hypothesis (\ref{item:uniqueQ}) of the Theorem,
Proposition~\ref{prop:uniquepop} then implies that $h_\L(R)$ is isotopic to $R$, keeping $\bdry h_\L(R)$ in $\bdry X_\L$.  Hence the diffeomorphism $h_\L$ may be isotoped so that $R$ is invariant.

Since $\hatR \cap \nbhd(L_+ \cup L_-)$ is just a pair of annuli, one each in the solid tori $\nbhd(L_+)$ and $\nbhd(L_-)$ from the boundary to the core, the isotopy of $h_\L(R)$ to $R$ extends across $\nbhd(L_+ \cup L_-)$ to give an isotopy of $h_\L(\hatR \cap X_K)$ to $\hatR \cap X_K$.  Hence the diffeomorphism $h$ may be isotoped so that this punctured annulus $\hatR \cap X_K$ is invariant.
  A further isotopy of $h$ ensures that a regular neighborhood  $\nbhd(\hatR \cap X_K)$ is invariant under $h$.  
  Since the closure of $\nbhd(\hatR \cup K)$ is the handlebody $H_P$, $h$ now restricts to a diffeomorphism $h\vert_M$ of the submanifold $M = Y\cut \nbhd(\hatR \cup K) = Y \cut H_P$.

Because $Q$ is an incompressible, boundary incompressible pair of pants in $M$ by hypothesis (\ref{item:essentialQ}), its image $h\vert_M(Q)$ must be as well. Similarly, since $Q$ is separating in $M$, so too must be $h\vert_M(Q)$. So by hypothesis
(\ref{item:uniqueQ}), $h\vert_M(Q)$ must be isotopic to $Q$ in $M$.  Therefore $h$ may be isotoped to preserve the pair $(M,Q)$.   Then by hypothesis (\ref{item:nothomeo}), $h\vert_M$ must preserve the sides of $Q$.  That is, $h$ may be isotoped so that there is an $h$--invariant product neighborhood $Q \times I$ of $Q$ in $M$ (with $\bdry Q \times I \subset \bdry M$) so that $h$ restricted to $Q \times I$ acts as identity on the $I$ factor.

We may view $X_K = M[\gamma]$ as $M$ in union with a compression body $W= (\partial M \times I) \cup (2-\hbox
{handle})$ where the $2$-handle is attached along the non-separating curve $\gamma$ 
(since $W$ is the exterior of a core curve of the genus $2$ handlebody $H_P$).  Thus
$\gamma$ is the unique isotopy class in $\partial W$ of the boundary of a non-separating, boundary reducing disk of $W$ (cf.\ \cite[Lemma 2.8]{BBL-handlebody}). Thus $h(\gamma)$ is isotopic to $\gamma$
in $\partial W$ and hence in the component $\partial M$ of $\bdry W$.

Place a hyperbolic structure on $\partial M$. Isotope $\partial Q, \gamma$ to geodesics. 
By Lemma 2.6 of \cite{Casson-Bleiler} (applied to $C_1=h(\partial Q), C_2=h(\gamma)$), 
and the uniqueness of geodesic representatives, 
$h$ can be isotoped so that $h(\partial Q)=\partial Q$ and $h(\gamma)=\gamma$.
Thus $h\vert_{\partial M}$ is a graph homeomorphism of the graph
$\bdry Q \cup \gamma$ to itself.

\begin{claim}
$h\vert_{\partial M}$ fixes the vertices of the graph $\bdry Q \cup \gamma$.
\end{claim}

\begin{proof}
 Since geodesic representatives intersect minimally in their isotopy classes, $\gamma$ intersects each component of $\bdry Q$ a distinct non-zero number of times by (c). Thus $h$ takes each component of $\bdry Q$ to itself. As $h$ restricted to $Q \times I$ is the identity on the $I$--factor, $h\vert_{\partial M}$ preserves
the sides of each component of $\bdry Q$.

First assume that $h\vert_{\partial M}$ reverses orientation. Then as
$h\vert_{\partial M}$ preserves sides of $\bdry Q$, 
it must take each component of $\bdry Q$ to itself reversing orientation.  
Furthermore, as $\gamma$
is transverse to each component of $\bdry Q$ and is invariant under $h\vert_{\partial M}$, 
$h\vert_{\partial M}$ takes $\gamma$ to itself preserving orientation. 
Thus $h\vert_{\partial M}$ acts as a rotation along $\gamma$ on the vertices $\gamma \cup
\bdry Q$. In particular, if $h\vert_{\partial M}$ fixes one vertex, it fixes all. But by (c) above,
$\gamma$ intersects some component of $\bdry Q$ an odd number of times. Since 
$h\vert_{\partial M}$ takes this component to itself reversing orientation, it must fix some point of
$\gamma \cap \bdry Q$. This proves the claim.

So assume $h\vert_{\partial M}$ is 
orientation-preserving. Since it preserves the sides of each component of $\bdry Q$, 
$h\vert_{\partial M}$ is 
orientation-preserving on each component of $\partial Q$.  Thus $h\vert_{\partial M}$
rotates the vertices of $\gamma \cup \partial Q$ along each component of $\bdry Q$ (possibly trivially).  Using the facts that $h\vert_{\partial M}$ takes arcs of  
$\gamma - \bdry Q$ to arcs of  $\gamma - \bdry Q$ connecting the same components of 
$\bdry Q$ and that $\gamma$ intersects each component of $\partial Q$,
one can see that $h$ restricted to $\partial Q \cup \gamma$ fixes vertices along
each component of $\bdry Q$.
\end{proof}

Since $h\vert_{\partial M}$ fixes vertices, it can be isotoped to be the identity on the full graph. 
As the complementary regions of this graph in $\bdry M$ are all disks ($\gamma$ intersects each
component of $\bdry Q$),
the Alexander isotopy trick then allows
us to extend this to an isotopy of all of  $h\vert_{\partial M}$ to the identity.

We next claim that, after a further isotopy of $h$ with support in the interior of $M$, $h$ is the identity on $Q$. 
Recall that the pair of pants $Q$ divides the manifold $M$ into the two genus $2$ handlebodies $H_+$ and $H_-$.  Since  $h(Q)=Q$ and $h$ preserves the sides of $Q$, $h$ restricts to a diffeomorphism of each of these handlebodies.  Moreover, since $h\vert_{\bdry M}$ is the identity, $h\vert_{\bdry H_+}$ and $h\vert_{\bdry H_-}$ are each a composition of Dehn twists along a collection of disjoint curves in $Q$.  By \cite[Theorem~1.11]{oertelhandlebodies} and its proof, for each $H_+$ and $H_-$, this collection of curves is the boundary of a collection of disjoint  meridional disks and  incompressible, non-boundary parallel annuli (so that twists along the disks and annuli produce the diffeomorphism of the handlebody).  However, since $Q$ is a pair of pants, each of these curves are isotopic to a component of $\bdry Q$, and so these disks and annuli may be isotoped in their respective handlebodies to have boundary in $\bdry M$.  Yet unless these collections of disks and annuli are empty, this now contradicts that $M$ is simple.  Hence we may further isotop $h$ to also be the identity on $Q$.

Since $h$ is the identity on $\bdry M \cup Q$, it must be isotopic to the identity on each handlebody $H_+$ and $H_-$, and hence $M$.

Finally, since $\gamma$ bounds a non-separating disk $D$ in the compression-body $W=X_K \cut M$, 
we may further isotope $h$ in the interior of $W$ so that $D$ is invariant under $h$. 
Because $h$ is the identity on $\gamma = \bdry D$, it may be isotoped to be the identity on all of $D$.  Therefore, $h$ may be further isotoped in the compression-body to be the identity on a collar of $\bdry M \cup D$, and thus to be the identity on $X_K$.  Hence the diffeomorphism $h$ is isotopic to the identity.
\end{proof}

\section{The pair of pants $Q$ in the submanifold $M=H_+ \cup_Q H_-$}\label{sec:popQ}

We continue with the notation set in the statement of Theorem~\ref{thm:mainasymmetry} and proceed to develop conditions that ensure the hypotheses of the Theorem are met.  In particular, in this section we only need that $M$ is a $3$--manifold with genus $2$ boundary obtained as the union of two genus $2$ handlebodies $H_+$ and $H_-$ glued together along a pair of pants $Q$.
Then Lemmas~\ref{lem:Qess}, \ref{lem:Msimple}, 
\ref{lem:Qunique} respectively demonstrate that the requirements   (\ref{item:essentialQ}),  (\ref{item:simpleM}), 
 (\ref{item:uniqueQ}) on $M$ and $Q$ of Theorem~\ref{thm:mainasymmetry} are implied by conditions on the disk-busting and annulus-busting nature of $\bdry Q$ in the two handlebodies $H_+$ and $H_-$.  For our application in the special case that $M$ is the exterior of the pair of pants $P=P^{\alpha,m} \subset S^3$ (as set in Section~\ref{sec:construction}), these conditions are then checked in Section~\ref{sec:finalcheck}.

\begin{defn}[Disk-busting and annulus-busting]
Let $\calC$ be a collection of simple closed curves embedded in the boundary of an orientable $3$--manifold $H$. We say $\calC$ is {\em $k$--disk-busting (in $H$)} for a positive integer $k$ if any properly embedded disk in $H$ that $\calC$ intersects fewer than $k$ times is $\bdry$--parallel in $H$.   When $k=1$, we simply say $\calC$ is {\em disk-busting}.   
Similarly, we say $\calC$ is {\em annulus-busting} if any properly embedded annulus in $H$ that is disjoint from $\calC$ is either compressible or $\bdry$--parallel in $H$.
\end{defn}

\begin{lemma}\label{lem:Qess}
	Assume $\bdry Q$ is $3$--disk-busting in each $H_+$ and $H_-$.
	Then $Q$ is incompressible and boundary incompressible in $M$. 
\end{lemma}

\begin{proof}
	Since $Q$ separates $M$ into $H_+$ and $H_-$, any compressing disk or boundary compressing disk for $Q$ would lie in either $H_+$ or $H_-$.  However, since such disks must be meridional in the handlebodies and must meet $\bdry Q$ either $0$ or $2$ times at most, $\bdry Q$ could not be $3$--disk-busting.  Thus $Q$ is incompressible and $\bdry$--incompressible.  
\end{proof}

\begin{lemma}\label{lem:Msimple}
	Assume $\bdry Q$ is $6$--disk-busting and annulus-busting in each $H_+$ and $H_-$.
	Then $M$ is simple.    
\end{lemma}
Consequently, $M$ is a hyperbolic $3$--manifold with geodesic boundary.

\begin{proof}
Among $2$--spheres that do not bound $3$--balls and properly embedded incompressible, non--$\bdry$--parallel
disks, annuli, or tori that are transverse to $Q$, choose $F$ to be one that intersects $Q$ minimally.  Note that $Q$ is incompressible and boundary incompressible by Lemma~\ref{lem:Qess}. We may assume no simple closed 
curve component of $F \cap Q$ is trivial in $Q$, as otherwise surgery of $F$ along the disk bounded by
an innermost  such curve will
produce  a new essential surface intersecting $Q$ fewer times. Thus any simple closed curve of $F \cap Q$ must be isotopic in $Q$ to a component of $\bdry Q$. Similarly no arc component of $F \cap Q$ can
be boundary parallel in $Q$, else surgery would find a new essential surface intersecting $Q$ fewer
times. No simple closed curve of $F \cap Q$ can bound a disk in $F$, for an innermost such would
be a compressing disk for $Q$. No arc component of $F \cap Q$ can be boundary parallel in $F$
as such would give rise to a boundary compression of $Q$. 

\medskip
	{\bf $\boldsymbol{F}$ is a sphere.}
	If $F$ is a sphere, then since every simple closed curve in $F$  bounds a disk, $F$ must be disjoint from $Q$.  But then the sphere $F$ is contained in a handlebody.  Since handlebodies are irreducible, $F$ must bound a ball, a contradiction.

	\medskip
	{\bf $\boldsymbol{F}$ is a disk.}
	If $F$ is a disk, since every simple closed curve in $F$ bounds a disk and every arc in $F$ is $\bdry$--parallel, then $F$ must be disjoint from $Q$.  Thus $\bdry F$ is contained in one of the pairs of pants $P_+$ or $P_-$ and  is therefore isotopic to a component of $\bdry P_+$ or $\bdry P_-$.  Since $\bdry P_+ =\bdry P_- = \bdry Q$, $\bdry F$ must be isotopic to a component of $\bdry Q$.  Hence $F$ would be a compressing disk for $Q$, a contradiction.

	\medskip
	{\bf $\boldsymbol{F}$ is an annulus.}
	If $F$ is an annulus, then $F \cap Q$ consists of either only spanning arcs of $F$ or only curves isotopic into $\bdry F$. Because $Q$ is annulus-busting, $F \cap Q$ is non-empty.   
	
Since $Q$ is separating, if there is one arc of $F \cap Q$ then there must be another and  $Q$ chops $F$ into rectangles contained in either $H_+$ or $H_-$.  Since these rectangles are disks in $H_+$ or $H_-$ that cross $\bdry Q$ exactly $4$ times, they must be $\bdry$--parallel because $Q$ is $6$--disk-busting. 
Such a boundary parallelism guides an isotopy of $F$ that reduces the number of intersections of
$F$ with $Q$, a contradiction. 

	If $F \cap Q$ is a collection of simple closed curves, they are all $\bdry$--parallel in each $F$ and $Q$.  Let $\gamma$ be one that is outermost in $Q$, cutting off a subannulus $A \subset Q$ with a component of $\bdry Q$.  Let $F'$ and $F''$ be the two annuli formed by surgering $F$ along $A$. Let $A^*$ be the ``dual'' annulus in $\bdry M$ so that surgering $F' \cup F''$ along $A^*$ recovers $F$.  
	Since each $F'$ and $F''$ intersect $Q$ fewer times than $F$, they must be boundary parallel.  Any compression of $F'$ or $F''$ would give a compression of $F$.  Let $V'$ and $V''$ be the two solid tori in $X_P$ through which $F'$ and $F''$ are $\bdry$--parallel; they are either disjoint or nested.  If $V'$ and $V''$ are disjoint, then $A$ is also disjoint from them and $V' \cup \nbhd(A^*) \cup V''$ is a solid torus giving a $\bdry$--parallelism of $F$.  If, say, $V'$ is contained in $V''$, then $V'' \cut (\nbhd(A^*) \cup V')$ is a solid torus giving a $\bdry$--parallelism of $F$.

	\medskip
	{\bf $\boldsymbol{F}$ is a torus.}
	If $F$ is a torus, then because handlebodies contain no embedded closed incompressible surfaces, $F \cap Q$ is non-empty.  Since $Q\cap F$ is a collection of simple closed curves, they must all be essential in $F$ and hence parallel in $F$.  Thus $Q$ chops $F$ into annuli in $H_+$ and $H_-$ with boundary disjoint from $\bdry Q$. Since $Q$ is annulus-busting in each $H_+$ and $H_-$, these annuli must be boundary parallel or compressible. These annuli must be incompressible since otherwise they would induce a compression of $Q$, contrary to Lemma~\ref{lem:Qess}. Hence they are $\bdry$--parallel.  Since no two components of $\bdry Q$ are parallel in $\bdry H_+$ or $\bdry H_-$, these annuli must all be parallel into $Q$.  Hence $F$ is isotopic into $\nbhd(Q)$ and is therefore compressible, a contradiction.
\end{proof}

\begin{lemma}\label{lem:Qunique}
	Assume $\bdry Q$ is $6$--disk-busting and annulus-busting in each $H_+$ and $H_-$, and  $8$--disk-busting in either $H_+$ or $H_-$.

	Then any properly embedded pair of pants in $M$ is either  compressible, $\bdry$--parallel, isotopic to $Q$, or non-separating and can be isotoped to be disjoint from some component of $\bdry Q$.
\end{lemma}

\begin{proof}
	Assume there is a properly embedded, incompressible,  non--$\bdry$--parallel pair of pants in $M$ that is not isotopic to $Q$.
	Among such surfaces that are transverse to $Q$ but not isotopic to $Q$, choose $F$ to be one that intersects $Q$ minimally.   As at the beginning of the proof of Lemma~\ref{lem:Msimple},  any simple closed curve of $F \cap Q$ must be isotopic in $Q$ to a component of $\bdry Q$ and isotopic in $F$ to a component of $\bdry F$ because both $Q$ and $F$ are incompressible.   Similarly, no arc of $Q \cap F$ is $\bdry$--parallel in $F$ because $Q$ is $\bdry$--incompressible.   
	
	First we show that $F$ cannot be $\bdry$--compressible.
	Assume $F$ admits a $\bdry$--compression along a disk $\delta$.  Compressing $F$ along $\delta$ produces one or two annuli, properly embedded in $X_P$. Let $\delta^*$ be the arc in $\bdry M$ dual to $\delta$ so that surgering these annuli along $\delta^*$ recovers $F$. These annuli must be incompressible since they share a boundary component with the incompressible surface $F$.  Therefore by Lemma~\ref{lem:Msimple} these annuli must be $\bdry$--parallel. Using the arrangements of the $\bdry$--parallelisms and $\delta^*$, it follows that $F$ must either compress or be $\bdry$--parallel, contrary to assumption.   (If $F$ surgers along $\delta$ to give one annulus $F'$ with $\bdry$--parallelism through the solid torus $V'$, then either $\delta^*$ is  inside $V'$ and so $F$ is compressible or $\delta^*$ is outside $V'$ and $F$ is $\bdry$--parallel.  If $F$ surgers along $\delta$ to give two annuli $F'$ and $F''$ with $\bdry$--parallelism through the solid tori $V'$ and $V''$ joined by $\delta^*$, then either $V'$ and $V''$ are disjoint so that $F$ is $\bdry$--parallel or $V'$ and $V''$ are nested with $F'$ and $F''$ parallel so that $F$ compresses.) 
	
	Thus any component of $F \cap Q$ is either an arc that is essential in both $F$ and $Q$ or a simple closed curve that is $\bdry$--parallel in both $F$ and $Q$.
	\medskip

	{\bf $\boldsymbol{F\cap Q}$ is empty.}
	If $F \cap Q$ is empty, then $F$ is contained in either $H_+$ or $H_-$, say $H_+$.  Then $\bdry F$ is contained in $P_+$. Since $F$ is incompressible, each component of $\bdry F$ is isotopic to some component of $\bdry P_+ = \bdry Q$. 	
	
	If $F$ is non-separating, then we have the final conclusion. (Though note that because handlebodies do not contain closed non-separating surfaces, at least two components of $\bdry F$ are isotopic in $P_+$.)
	
	If $F$ is separating in $M$, it is separating in $H_+$.  Then $\bdry F$ is separating in $\bdry H_+$ and so $\bdry F$ is isotopic to $\bdry Q$. If not, then two of the isotopy classes of $\bdry Q$ contain an even number of components of $\bdry F$ and the remaining contains an odd number of components of $\bdry F$.  Since each each component of $\bdry Q$ is non-separating in $\bdry H_+$, there is a loop in $\bdry H_+$ intersecting this third component of $\bdry Q$ an odd number of times.  Therefore this loop must intersect $\bdry F$ an odd number of times, a contradiction.
	
	So now isotop $F$ in $H_+$ so that $\bdry F = \bdry Q$. Then let $D$ be a meridional disk of $H_+$ intersecting $F$ minimally.  Because $F$ is not $\bdry$--compressible in $M$, an arc of $F \cap D$ that is outermost in $D$ must cut off a $\bdry$--compressing disk $\delta$ for $F$ in $H_+$ so that $\bdry \delta$ is the union of an essential arc in $F$ and an essential arc in $Q$.  The $\bdry$--compression of $F$ along $\delta$ produces one or two annuli  properly embedded in $H_+$ whose boundaries cobound annuli in $Q$.  These annuli must be incompressible since $F$ is incompressible in $H_+$.  Then, since $\bdry Q$ is annulus-busting in $H_+$, these annuli must be $\bdry$--parallel.   Thus they are $\bdry$--parallel into $Q$.  Such parallelisms taken together with $\delta$ give an isotopy of $F$ to $Q$.

	\medskip
	
	{\bf $\boldsymbol{F\cap Q}$ contains a simple closed curve.}
	If there is a simple closed curve of $F \cap Q$, let $\gamma$ be one that is outermost in $Q$, cutting off a subannulus $A \subset Q$ with a component of $\bdry Q$.  Any arc of $F \cap Q$ in $A$ would bound a disk in $A$ giving a $\bdry$--compression of $F$, so the interior of $A$ is disjoint from $F$. Let $F'$ and $F''$ be the annulus and pair of pants respectively formed by surgering $F$ along $A$. Let $A^*$ be the ``dual'' annulus in $\bdry M$ so that surgering $F' \cup F''$ along $A^*$ recovers $F$.  Both $F'$ and $F''$ must be incompressible since their boundary components are all isotopic to boundary components of incompressible surfaces.  By Lemma~\ref{lem:Msimple} the annulus $F'$ must then be $\bdry$--parallel.   Since $F''$ intersects $Q$ fewer times than $F$, by the assumed minimality of $|F \cap Q|$ the pair of pants $F''$ must be $\bdry$--parallel (if $F''$ were isotopic to $Q$ then $F$ would be also). Let $V'$ and $V''$ be the solid torus and genus $2$ handlebody of parallelisms for $F'$ and $F''$; since $F'$ and $F''$ are disjoint, they are either disjoint or nested.   If $V'$ and $V''$ are disjoint, then $A$ is also disjoint from them and $V' \cup \nbhd(A^*) \cup V''$ is a handlebody giving a $\bdry$--parallelism of $F$.  If $V'$ is contained in $V''$, then $V'' \cut (\nbhd(A^*) \cup V')$ is a handlebody giving a $\bdry$--parallelism of $F$.  Since $F''$ is incompressible, $V''$ cannot be contained in $V'$.  Thus $F \cap Q$ contains no simple closed curves.

	\medskip

	{\bf $\boldsymbol{F\cap Q}$ contains only arcs.}
	Since $F \cap Q$ is non-empty but contains no simple closed curves, $F \cap Q$ may contain only arcs that are essential in each $F$ and $Q$.  In particular, this means that no disk component of $F \cut Q$ can be $\bdry$--parallel in either $H_+$ or $H_-$.  Since $Q$ is separating,  either $F \cap Q$  contains a pair of arcs that are parallel in $F$ cutting off a rectangle component of $F \cut Q$, $F \cap Q$ is a set of three non-isotopic arcs chopping $F$ into two hexagonal disks, or $F \cap Q$ is a single arc chopping $F$ into two annuli.  In the first case, since $\bdry Q$ is $6$--disk busting in each $H_+$ and $H_-$, any such rectangle must be $\bdry$--parallel, a contradiction.  In the second case, since neither of the two hexagonal disks $F \cap H_+$ and $F \cap H_-$ is $\bdry$--parallel, $\bdry Q$ cannot be $8$--disk-busting in either $H_+$ or $H_-$, a contradiction.  Therefore $Q$ cuts $F$ along a single arc into two annuli, one in each $H_+$ and $H_-$. 
	
	Consider the annulus $F_+ = F \cap H_+$.  One component of $\bdry F_+$ is disjoint from $Q$ and thus isotopic in $P_+$ to a component of $\bdry Q$.  The other component of $\bdry F_+$ intersects $Q$ in an essential arc. Therefore if $F_+$ were $\bdry$--parallel in $H_+$, this second component of $\bdry F_+$ must intersect $P_+$ in a $\bdry$--parallel arc so that $F_+$ is isotopic to a component of $Q \cut F$ in $Q$.  Yet then $F$ could be isotoped to be disjoint from $Q$, contrary to assumption. Hence $F_+$ is not $\bdry$--parallel in $H_+$.

Assume $F_+$ is separating in $H_+$.  Because $F_+$ is a properly embedded,  incompressible, separating annulus in a genus $2$ handlebody $H_+$, the two components of $\bdry F_+$ must be isotopic in $\bdry H_+$ (e.g.\ see \cite[Section~6.3]{BGL-obtaininggenus2}).  Thus, since one component of $\bdry F_+$ is disjoint from $Q$, the other may be isotoped in $\bdry H_+$ to also be disjoint from $Q$.  Since $\bdry Q$ is annulus-busting, $F_+$ must be $\bdry$--parallel.  However we have already concluded that this cannot be the case. 

So $F_+$ must be non-separating in $H_+$. Since a component of $\bdry F_+$ is in $P_+$ and isotopic to a component of $\bdry Q$, $F_+$ is disjoint from that component of $\bdry Q$. Hence $F$ is non-separating and disjoint from that component of $\bdry Q$, giving the final conclusion.
\end{proof}

\section{Busting disks and annuli in genus $2$ handlebodies}

\subsection{Basics of disk-busting and primitive curves}

\begin{lemma}[{E.g.\ \cite[Lemma~5.3]{BBL-handlebody}}]
	\label{lem:2diskbusting}
	Let $\calC$ be a collection of curves in a genus $g$ boundary component of a $3$--manifold $H$.  If $\calC$ is disk-busting, then either $g=1$ or $\calC$ is $2$--disk-busting.
\end{lemma}
\begin{proof}
	Assume $D$ is a properly embedded disk in $H$ intersecting $\calC$ just once, then two copies of $D$ can be banded together along the component of $\calC$ that intersects $D$ to form a new disk $D'$ that is disjoint from $\calC$.  If $g\geq 2$, then $\bdry D'$ is essential in $\bdry H$ and so $D'$ cannot be $\bdry$--parallel.  But this contradicts that $\calC$ is disk-busting.   
\end{proof}

A single simple closed curve $C$ in the boundary of a handlebody $H$ is called {\em primitive} if there is some meridional disk of $H$ that it intersects only once.

 Let $\calC$ be a collection of simple closed curves embedded in the boundary of an orientable $3$--manifold $H$. Let $H[\calC]$ be the $3$--manifold obtained by attaching $2$--handles to $H$ along the components of $\calC$ and then attaching $3$--handles to any sphere components of the resulting boundary.

\begin{lemma}\label{lem:handleaddition}
	Let $C$ be a simple closed curve in the boundary of a  handlebody $H$ of genus $g\geq2$.  The curve $C$ is primitive if and only if $H[C]$ is a genus $g-1$ handlebody.  The curve $C$ is disk-busting if and only if $H[C]$ is $\bdry$--irreducible.
\end{lemma}

\begin{proof}
	If $C$ is primitive in $H$, then in $H[C]$ the core disk of the attached $2$--handle cancels a meridional disk of $H$ intersecting $C$ once.  Hence $H[C]$ is a handlebody.  Conversely, if $H[C]$ is a handlebody, then $C$ is primitive by \cite{gordon-primitiveloops}.

	If $C$ is disk-busting, then $\bdry H - C$ is incompressible in $H$.  Since $H$ is irreducible and $\bdry H$ is compressible, the Handle Addition Lemma \cite[Lemma~2.1.1]{CGLS} implies that $H[C]$ is irreducible and $\bdry$--irreducible.
	If  $C$ is not disk-busting, then there is a compressing disk $D$ for $\bdry H - C$ in $H$.  We may take $D$ so that it is non-separating in $H$ so that $D$ is non-separating in $H[C]$.    (If $D$ were separating, then the component of $H \cut D$ disjoint from $C$ is a handlebody of genus at least one that contains a non-separating compressing disk.) Hence $D$ is a compressing disk of $H[C]$.
\end{proof}

\subsection{The busting nature of curves in a genus $2$ handlebody}
Throughout this subsection, assume $\calC$ is a triple of simple closed curves in the boundary of a genus $2$ handlebody $H$ bounding two pairs of pants in $\bdry H$.

\begin{lemma}\label{lem:disktoannulusbusting}
	If each curve in $\calC$ is disk-busting, then $\calC$ is both $6$--disk-busting and annulus-busting.
\end{lemma}

\begin{proof}
	By Lemma~\ref{lem:2diskbusting}, each curve in $\calC$ is actually $2$--disk-busting.
	Since each curve of $\calC$ intersects the boundary of an essential disk of $H$ at least twice, together they intersect this disk at least $6$ times.  Hence $\calC$ is $6$--disk-busting.
	
	Now let $A$ be a properly embedded, incompressible annulus in $H$ that is not $\bdry$--parallel.   Assume $A$ is disjoint from $\calC$.  Because $\bdry H \cut \calC$ is two pairs of pants and no component of $\bdry A$ bounds a disk in $\bdry H$, the components of $\bdry A$ are each isotopic to some component of $\calC$.   Since $A$ is an incompressible annulus in a handlebody, it must admit a $\bdry$--compression.  The result of such a $\bdry$--compression is a properly embedded disk $D$ whose boundary is disjoint from $\bdry A$. Furthermore, because $A$ is incompressible and not $\bdry$--parallel, this disk $D$ cannot be $\bdry$--parallel; it must be a meridional disk of $H$.  Each component of $\calC$ intersects $\bdry D$ by hypothesis, and yet $\bdry A$ is disjoint from $\bdry D$.  This contradicts that each component of $\bdry A$ is isotopic to some component of $\calC$.
\end{proof}

\begin{lemma}\label{lem:primbustbustis5diskbusting}
	If two of the curves in $\calC$ are disk-busting and the third is primitive, then $\calC$ is $6$--disk-busting.
\end{lemma}

\begin{proof}
	Assume $\calC$ is not $6$--disk-busting.  Since $\calC$ is separating in $\bdry H$, any curve in $\bdry H$ transversely intersects $\calC$ an even number of times.  Hence if $\calC$ were $5$--disk-busting, it would be $6$--disk-busting.  So we may assume $\calC$ is not $5$--disk-busting.
	
	Let $B$ and $B'$ be the disk-busting curves of $\calC$ and let $C$ be the primitive curve.  Let $P$ and $Q$ be the two pairs of pants in $\bdry H$ bounded by $\calC$.  Since any disk-busting curve is $2$--disk-busting (Lemma~\ref{lem:2diskbusting}), there is a meridional disk $D$ of $H$ that is disjoint from $C$ and intersects each $B$ and $B'$ exactly twice.  Therefore $\bdry D \cap Q$ is a pair of arcs, essential and parallel in $Q$, that join $B$ and $B'$.  Similarly, $\bdry D \cap P$ is a pair of arcs, essential and parallel in $P$, that join $B$ and $B'$. 
	
	If $D$ is separating, then $\bdry H \cut \bdry D$ is two once-punctured tori.  Then the component containing $C$ is the union of the two annulus components of $P \cut \bdry D$ and $Q \cut \bdry D$ along the curve $C$ and an arc in each $B$ and $B'$.  This union however must have $\chi=-2$ and thus is not a once-punctured torus, a contradiction.

If $D$ is non-separating, then it becomes a meridional disk in the solid torus $H[C]$ which is disjoint
from the dual arc $C^*$ (the co-core of the attached $2$--handle).  Because $C$ is primitive in $H$, there is a disk $D'$ in $H$ intersecting $C$ in a single point; furthermore, this disk $D'$ can
be chosen to be disjoint from $D$. Thus $C^*$ is a trivial arc in the ball  $H[C]-\nbhd(D)$.  The pair of pants $Q$ closes off to an annulus $A=Q[C]$ in $\bdry H[C]$ that crosses $D$ twice  and contains only one endpoint of the arc 
$C^*$.  The core of the annulus $A$ winds twice around the solid torus $H[C]$.  Since $A$ contains only one endpoint of $C^*$, $C^*$ is isotopic to an arc in the boundary of $H[C]$ intersecting some component
of $\bdry A$ exactly once.  The trace of this isotopy can be taken to be a disk properly embedded in $H$ and
intersecting $B$ or $B'$ exactly once, a contradiction.
\end{proof}

\begin{lemma}\label{lem:primitivetoannulusbusting}
	If two of the curves in $\calC$ are disk-busting and the third is primitive, then $\calC$ is annulus-busting.
\end{lemma}

\begin{proof}
	As in the proof of Lemma~\ref{lem:disktoannulusbusting}, let $A$ be a properly embedded, incompressible annulus in $H$ that is not $\bdry$--parallel.   Assume $A$ is disjoint from $\calC$.  Because $\bdry H \cut \calC$ is two pairs of pants and no component of $\bdry A$ bounds a disk in $\bdry H$, the components of $\bdry A$ are each isotopic to some component of $\calC$.   Since $A$ is an incompressible annulus in a handlebody, it must admit a $\bdry$--compression.  The result of such a $\bdry$--compression is a properly embedded disk $D$ whose boundary is disjoint from $\bdry A$. Furthermore, because $A$ is incompressible and not $\bdry$--parallel, this disk $D$ cannot be $\bdry$--parallel; it must be a meridional disk of $H$.  
	
	If $A$ is non-separating, then the two components of $\bdry A$ are non-isotopic.  Hence they are isotopic to distinct components of $\calC$.  However this contradicts that only one component of $\calC$ is not disk-busting.   Therefore $A$ is separating.   Hence the two components of $\bdry A$ are isotopic and must therefore be isotopic to the primitive component of $\calC$.  Since both components of $\bdry A$ are isotopic to a primitive curve, there is a meridional disk $D'$ of $H$ that intersects $A$ in a single spanning arc.  Therefore $D' \cut A$ is a pair of $\bdry$--compression disks for $A$ in $H$, one to each side of $A$.  But this implies that $A$ is $\bdry$--parallel in $H$, e.g.\ see \cite[Section~6.3]{BGL-obtaininggenus2}.
\end{proof}

\begin{lemma}\label{lem:sepdiskbust} 
If each of the curves in $\calC$ is disk-busting and $H[\calC]$ is not a lens space (including   $S^1 \times S^2$), $\R P^3 \# \R P^3$, or a prism manifold, then  $\calC$ is $8$--disk-busting.
\end{lemma}

\begin{proof}
Assume each of the curves in $\calC$ is disk-busting and that $\calC$ is not $8$--disk-busting.
Since $\calC$ is separating, its  intersection number with any curve in $\bdry H$ is even; hence $\calC$ cannot be $7$--disk-busting.   
Since each component of  $\calC$ is disk-busting, $\calC$ is necessarily $6$--disk-busting (Lemma~\ref{lem:disktoannulusbusting}). 
So assume $D$ is a meridional disk  of $H$ that $\calC$ intersects $6$ times.  By Lemma~\ref{lem:2diskbusting}
each component intersects $D$ twice.

Observe that $\bdry H$ is a Heegaard surface of $H[\calC]$:  the attached $2$--handles together with the $3$--handles filling up the sphere boundary components define a genus $2$ handlebody $H_{\calC}$ with $\bdry H_{\calC} = \bdry H$ in which the curves $\calC$ are meridians.  Since the curves of $\calC$ each intersect $\bdry D$ twice, there is a homeomorphism from $H_{\calC}$ to an interval bundle over $F$, where $F$ is a surface with boudary, in which an annular collar of $\bdry D$ in $\bdry H$ maps to the corresponding interval bundle over $\bdry F$.  Here $F$ must be a connected surface with one boundary component and $\chi(F) =-1$.  Thus $F$ is either a once-punctured torus or a once-punctured Klein bottle depending on whether $\bdry D$ is separating or not in $\bdry H_{\calC}$.

If $\bdry D$ is separating, then  $H_{\calC} = F \times [-1,1]$, and $D \cup F \times\{0\}$ is a torus
whose exterior in $H[\calC]$ is $H \cut \nbhd(D)$.  Since $\bdry D$ is separating in $\bdry H$, $D$ must be separating in $H$.  Hence $H \cut \nbhd(D)$ must be two solid tori.  Therefore $H[\calC]$ is a lens space.

If $\bdry D$ is non-separating, then capping the $0$-section of the interval bundle $H_{\calC}$ gives a
Klein bottle whose exterior in $H[\calC]$ is $H \cut \nbhd(D)$. Since $\bdry D$ is non-separating in $\bdry H$, $D$ must be non-separating in $H$.  Hence $H \cut \nbhd(D)$ must be one solid torus.  Therefore $H[\calC]$ is a Dehn filling of the twisted $I$--bundle over the Klein bottle.  In other words, $H[\calC]$ is a (generalized) Seifert fibered space $S^2(0; \alpha, 1/2, 1/2)$ for some $\alpha$, a prism manifold (including  $\R P^3 \# \R P^3$ and $S^1 \times S^2$). 
\end{proof}

\section{Tangles} 
For $n\geq 1$, an {\em $n$--strand tangle (in a ball)} is a pair $\tau = (B, \calt)$ of a $3$--ball $B$ with a  properly embedded $1$--manifold $\calt$ that is homeomorphic to $n$ arcs and some number of circles. The $n$--strand tangle $\tau = (B, \calt)$ is called {\em rational} (or also {\em trivial}) if $\calt$ has no closed components and is isotopic rel--$\bdry$ into $\bdry B$.  Here we take a $0$--strand tangle to be (the pair of) a link $\calt$ in $B=S^3$.

\medskip
Let $\tau = (B,\calt)$ be a $n$--strand tangle.
If any properly embedded disk  or sphere  in $B$ that is disjoint from $\calt$ cuts off a ball disjoint from $\calt$, then $\tau$ is {\em non-split}.
If any properly embedded disk in $B$ that $\calt$ transversally intersects just once  cuts off a trivial $1$--strand tangle, then $\tau$ is {\em indivisible}.  
If $\tau$ is  non-split, indivisible, and not the trivial $1$--strand tangle, then $\tau$ is {\em essential}.

If any embedded sphere in $B$ that intersects $\calt$ in two points cuts out a trivial $1$--strand tangle from $B$, then $\tau$ is {\em locally trivial}.  
If $\tau$ is non-split, indivisible, and locally trivial, then $\tau$ is {\em prime}.
So if $\tau$ is prime and not the trivial $1$--strand tangle, then it is essential.  
(Tangles with closed components can be essential without being prime.)

The tangle $\tau$ is {\em toroidal} if some embedded torus in $B - \nbhd(\calt)$  is neither compressible in $B - \nbhd(\calt)$ nor isotopic to a boundary component of $B - \nbhd(\calt)$.   A {\em Conway sphere} for $\tau$ is a sphere $S$ embedded in $B$ that meets $\calt$ transversally in four points so that $S-\nbhd(\calt)$ is incompressible in $B - \nbhd(\calt)$ and $S-\nbhd(\calt)$ is not isotopic rel--$\bdry$ into $\bdry(B-\nbhd(\calt))$.  A {\em Conway disk} for $\tau$ is a disk $D$ properly embedded in $B$ that meets $\calt$ transversally in two points so that $D-\nbhd(\calt)$ is incompressible in $B - \nbhd(\calt)$ and $D-\nbhd(\calt)$ is not isotopic rel--$\bdry$ into $\bdry(B-\nbhd(\calt))$.

A diagram of the tangle $\tau$ is a projection of $B$ to a disk $D$ so that $\bdry \calt$ projects to distinct points in $\bdry D$ while the rest of $\calt$ projects to the interior of $D$ with over-under information at the crossings as in an ordinary link diagram.  A diagram of $\tau$ is {\em alternating} if when following along any arc of $\tau$, the crossings alternate between over and under.   It is {\em locally trivial} if any loop in $D$ meeting the arcs of the diagram bounds a disk in which the diagram consists of a single arc.   It is {\em reduced} if it contains no nugatory crossings.

\subsection{Disk-busting and primitive curves in genus $2$ handlebodies  via tangle quotients}

Since a genus $2$ handlebody $H$ is the double branched cover of a  $3$--strand rational tangle $\tau$, a collection of disjoint, non-separating, pairwise non-parallel, simple closed curves $\calC$ in $\bdry H$ can be described as the preimage of a set of arcs  $\calc$ in the boundary $\tau$.  When presented in this manner, it can be convenient to determine the disk-busting and primitive nature of the components of $\calC$ in $H$ in terms of the arc components of $\calc$ in $\tau$.

Let $\calc$ be collection of simple arcs in the boundary sphere of a tangle $\tau = (B, \calt)$ such that $\bdry \calc = \calc \cap \calt$.   Let $\tau_1 =(B_1,\calt_1)$ be the trivial $1$--strand tangle.   The {\em arc-closure}  of $\tau$ along $\calc$ (or simply the {\em $\calc$--closure} of $\tau$) is the tangle $\tau[\calc] = (B', \calt')$ obtained by attaching a copy of $\tau_1$ to $\tau$ along each disk component of a regular neighborhood of $\calc$ in $\bdry B$ so that the endpoints of the arcs $\calt_1$ correspond to the endpoints  $\bdry \calc \subset \calt$,  $\calt'$ is the union of $\calt$ with the arcs $\calt_1$, and $B'$ is the union of $B$ with the balls $B_1$ with its boundary further filled with a ball (i.e.\ $S^3$) if $\bdry \calt' = \emptyset$.  Equivalently, $\tau[\calc] = (B', \calt')$ where the properly embedded $1$--manifold $\calt'$ is obtained from $\calc \cup \calt$ by pushing the interior of $\calc \cup \calt$ into the interior of $B$ and $B'$ is either $B$ or $B$ with its boundary filled with a ball (i.e.\ $S^3$)  if $\bdry \calt' = \emptyset$.

If $H$ is the double branched cover of $\tau$, then $c$ lifts to a curve $C$ in $\bdry H$ so that $H[C]$ is the double branched cover of $\tau[c]$. When $\tau$ is a rational tangle, so that $H$ is a handlebody, Lemma~\ref{lem:handleaddition} allows us to determine if the curve $C$ in $\bdry H$ is primitive or disk busting in terms of the tangle $\tau[c]$.

\begin{lemma}\label{lem:tanglehandleaddition}
Let $c$ be an arc in the boundary of a rational $3$--strand tangle $\tau$.   Let $H$ be the genus $2$ handlebody that is the double branched cover of $\tau$, and let the curve $C \subset \bdry H$ be the lift of $c$. Then
\begin{itemize}
\item $C$ is primitive in $H$ if and only if $\tau[c]$ is a rational $2$--strand tangle, and
\item $C$ is disk-busting in $H$ if and only if $\tau[c]$ is an essential tangle.
\end{itemize}
\end{lemma}

\begin{proof}
Since the involution of $H$ given by the double cover is the hyperelliptic involution,  every simple closed curve in $\bdry H$ and properly embedded disk in $H$ can be made invariant under the involution. Thus this lemma follows directly from Lemma~\ref{lem:handleaddition}. (A disk indicating the failure of $\tau[c]$  to be non-split or indivisible will lift to a boundary reducing disk of $H[C]$.  
If $H[C]$ is boundary reducible, there is a boundary reducing disk that lies in $H$ by the Handle Addition Lemma \cite[Lemma~2.1.1]{CGLS} which may then be made equivariant under the involution by an isotopy.
 Either this disk transversally intersects the fixed set once or it contains an arc of the fixed set.  In the former case, the disk itself quotients to disk indicating the failure of $\tau[c]$ to be indivisible.  In the latter case an equivariant pair of push-offs of the disk descends to show that $\tau[c]$ fails to be non-split.)
\end{proof}

\subsection{Non-simple manifolds  via tangle quotients}

In this section, we collect some well-known results that will allow us to recognize the 
non-simplicity of a  double branched cover from its tangle quotient.

\begin{lemma}\label{lem:primetangles}
Let $X$ be the double branched cover of an $n$--strand tangle $\chi$ in a ball $B$ (or the 
$3$-sphere $B$ when $n=0$).
If $X$ is reducible then  $\chi$ is not prime.
\end{lemma}

\begin{proof}
Let $\iota$ be the involution on $X$ that is the deck transformation of the branched covering $p \colon X \to \chi$, and let $\fix(\iota)$ be its fixed set. 

Assume $X$ contains a reducing sphere.  By the Equivariant Sphere Theorem \cite{Dunwoody, MSY}, there exists a reducing sphere $S$ such that either $S \cap \iota(S) = \emptyset$ or $S$ is transverse to $\fix(\iota)$ and $S = \iota(S)$.  So first assume $S$ is disjoint from $\fix(\iota)$.  Then $S$ projects to a sphere that doesn't bound a ball in the tangle complement, though it must bound a ball in $B$.  Hence $\chi$ must be split and therefore not prime.  So now if instead $S$ is not disjoint from $\fix(\iota)$, then an Euler characteristic argument shows that they intersect twice.  Hence $S$ projects to a sphere $p(S)$ intersected twice by the tangle.  Since $p(S)$ bounds a ball in $B$, it bounds a $1$--strand tangle.  However since $S$ does not bound a ball in $X$, this $1$--strand tangle cannot be trivial.   Hence $\chi$ is not locally trivial and therefore not prime.
\end{proof}

\begin{lemma}\label{lem:nonsimpletangles}
Let $X$ be the double branched cover of a $2$--strand tangle $\chi$ in a ball $B$.  If $X$ is toroidal, then either $\chi$ contains an essential Conway sphere or $\chi$ is toroidal.  If $X$ is annular but not toroidal, then $\chi$ is the tangle sum of two rational tangles.
\end{lemma}

\begin{proof}
Let $\iota$ be the involution on $X$ that is the deck transformation of the branched covering $X \to \chi$, and let $\fix(\iota)$ be its fixed set.   

Assume $X$ contains an essential torus.  Then \cite[Corollary 4.6]{Holz} shows that there is an essential torus $T$ such that either  $T \cap \iota(T) = \emptyset$ or $T$ is transverse to $\fix(\iota)$ and $T = \iota(T)$. So first assume that $T$ is disjoint from $\fix(\iota)$.  Then $T$ projects to an incompressible, non--$\bdry$--parallel torus in the tangle complement; i.e.\ $\chi$ is toroidal.   Now if instead $T$ is not disjoint from $\fix(\iota)$, then an Euler characteristic argument shows its the quotient must be a Conway sphere.  Since $T$ is essential, this Conway sphere must be essential.

Assume $X$ contains an essential annulus but no essential torus.
Therefore $X$ is a Seifert fibered space over the disk with two
exceptional fibers.
Taking $X$ with such a Seifert fibration, the main theorem of \cite{Toll}
shows that $\iota$ may be taken to be fiber preserving. 
The involution on $X$ then induces an involution on the orbit surface $D$ that leaves the set of singular points invariant.  Thus it restricts to an involution $\iota'$ on the $3$--punctured sphere $D'$ that is the exterior of these singular points.   Moreover, we know that the involution takes the ``outside'' boundary $\bdry D \subset \bdry D'$ of this $3$--punctured sphere to itself.  By taking an essential arc with endpoints on this outside boundary that intersects itself minimally under the involution on the surface, we find one that is either invariant or taken off itself.  That is, there is such an arc $\alpha \subset D'$ with either $\alpha = \iota'(\alpha)$ or $\alpha \cap \iota'(\alpha) = \emptyset$.

Assume there is an essential arc $\alpha \subset D'$ with $\bdry \alpha \subset \bdry D$ such that $\alpha \cap \iota'(\alpha) = \emptyset$.  Then there is product region $\Delta$ in $D'$ between the arc
$\alpha$ and its image $\iota'(\alpha)$ that is invariant under the involution. Since the involution $\iota'$ switches the arcs,
it is either conjugate to a rotation on the disk or to a reflection across the axis in $\Delta$ between the two arcs.
The first can't happen because then $\fix(\iota)$ would be disjoint from $\bdry X$, which we know is not the case (since $\chi$ is a $2$--strand tangle in a ball). So the latter must occur.  Let $X_\Delta$ be the $S^1$--bundle over $\Delta$ as a Seifert fibered submanifold of $X$.  Then $\iota$ restricted to $X_\Delta$ must be reflection across the annulus $A$ over this axis followed by reflection in the $S^1$ factor since $\iota$ is orientation preserving.  Therefore $\fix(\iota) \cap X_\Delta$ consists of two spanning arcs of $A$.  Since $\fix(\iota)$ is a properly embedded $1$--manifold in $X$ meeting $\bdry X$ in four points, any remaining components of $\fix(\iota)$ must be closed components in the two solid tori of $X - X_\Delta$. 
 However because $\iota$ exchanges these two solid tori, 
the fixed set cannot meet their interior.  Hence $\fix{\iota}$ consists of only these two arcs of $A$. Thus, the annulus $A$ projects to a rectangle giving a parallelism between the two strands of $\chi$. Furthermore the exterior of this rectangle in the tangle ball is a solid torus.  Since $\chi$ is a tangle in a ball, $\chi$ is a trivial tangle. Hence $X$ is a solid torus and therefore contains no essential annulus, a contradiction.

So we assume there is an invariant essential arc $\alpha$. The involution on $\alpha$ is either the
identity or conjugate to reflection in the midpoint of $\alpha$. Look at the essential annulus $A$ sitting above
this arc. Assume $\fix(\iota)$ intersects $A$. If the involution is the identity on $\alpha$ then
it must be conjugate to reflection in the circle fibers. The involution must switch sides of the annulus
and we arrive at the contradiction above. 
So the involution on $\alpha$ must be conjugate to reflection along 
the midpoint.
Since $\fix(\iota)$ intersects
$\bdry X$, $\iota$
must preserve the sides of $A$. Since $\iota$ is
orientation-preserving it must reverse the
orientation of the circle fibers in $A$. Hence $\iota$ must be
conjugate to reflection in the circle fiber over
the midpoint of $\alpha$.  
So the quotient of $A$ is a Conway disk in $\chi$. This Conway disk must be essential since $A$ is.   Since each side of $A$ in $X$ is a solid torus, the Conway disk must split $\chi$ into two rational tangles.

Assume $\fix(\iota)$ is disjoint from $A$. Since $\fix(\iota)$ is not disjoint from $\bdry X$,  $\iota$ must fix the sides of $A$.  
Thus the involution $\iota$ restricts to an involution on each of the solid tori $X'  \cut  A = X_1 \cup X_2$ on the two sides of the annulus. Since $\fix(\iota)$ intersects $\bdry X$ in four points but is disjoint from $A$, the pair of its intersection numbers with the boundaries of these two solid tori is either $\{2,2\}$ or $\{0,4\}$.   Since $\iota$ further restricts to an involution on each $\bdry X_1$ and $\bdry X_2$, the pair $\{2,2\}$ cannot occur; an involution on a torus cannot have a fixed set consisting of just two points.  Therefore $\fix(\iota)$ intersects $\bdry X_1$, say, in four points, and it does so in the annular complement $A_1$ of $A$.  Since $A$ is invariant under $\iota$, so is $A_1$.  Thus $\iota$ restricts to an involution on the annulus $A_1$ with exactly four fixed points which cannot occur. 
\end{proof}

\section{Construction of asymmetric L-space knots}
\label{sec:finalcheck}

Throughout this section we continue with the notation used for the construction of L-space knots in Section~\ref{sec:construction} and its further development in the statement of Theorem~\ref{thm:mainasymmetry}.  In particular, we consider the presentation $J^{\alpha,m}$ of the unknot as shown in the center of Figure~\ref{fig:almostalternating} and its decomposition into tangles $\tau_+ = (B^3, t_+)$ (above) and $\tau_- = (B^3, t_-)$ (below) by the bridge sphere containing arcs $ \calc= \calc^{\alpha,m} = \{c_\nu, c_\mu, c_\lambda\}$. 

The double branched cover of this unknot $J=J^{\alpha,m}$ is the manifold $Y=S^3$. Let $\Sigma$ be the genus $2$ Heegaard surface that is the lift of the bridge sphere, let the curves $\calC=\calC^{\alpha,m} = \{C_\nu, C_\mu, C_\lambda\} \subset \Sigma$ be the lift of the arcs $\calc$, and let $P=P^{\alpha,m}$ and $Q=Q^{\alpha,m}$ be the two pairs of pants in $\Sigma$ bounded by $\calC$.  
The exterior of $P$ is the manifold $M = H_+ \cup_Q H_- = Y\cut H_P$ and $\bdry M \cap H_\pm = P_\pm$.

Since $P_+$ and $P_-$ are isotopic in $H_P$, $H_+$ and $H_-$ may be expanded through $H_P$ to $P$ to be the handlebodies of the Heegaard splitting by $\Sigma$. Hence the tangles $\tau_+$ and $\tau_-$ are also the quotients of handlebodies $H_+$ and $H_-$ by the hyperelliptic involution, and the arcs $\calc$ in each of their boundary spheres are the quotients of the curves $\bdry Q = \calC$ in the boundaries of these handlebodies.

To be clear, let us take  $\alpha$ to be the alternating $3$--braid 
\[\alpha = \prod_{i=n}^1 \sigma_{\bar{i}}^{\epsilon_{i} a_{i}}  =\sigma_{\bar{n}}^{\epsilon_n a_n}  \dots \sigma_1^{-a_3} \sigma_2^{a_2} \sigma_1^{-a_1} \]
 where $n\geq 1$ is an integer, $\epsilon_n=(-1)^n$, $\bar{n}=1$ or $2$ according to the parity of $n$, and $a_i \geq 1$ for $i=1,\dots,n$.

\begin{lemma}\label{lem:Qbusting}

Take $\alpha$ as above with $n\geq 3$ and also  $m\geq 3$.
In $H_+$, each component of $\calC$ is disk-busting, and together $\calC$ is $8$--disk-busting.
In $H_-$, two components of $\calC$ are disk-busting, and one component of $\calC$ is primitive. 
\end{lemma}

\begin{remark}\label{rem:lowparameters}
As one may observe from the proof of Lemma~\ref{lem:Qbusting}, regardless of choice of  $\alpha$, the component $C_\nu$ of $\calC$ will always be primitive in $H_-$.
One may further check that $C_\lambda$ will be primitive in $H_-$ when $m=1$ and primitive in $H_+$ when $n=2$.  
\end{remark}

\begin{proof}
We use Lemma~\ref{lem:tanglehandleaddition}, the tangle form of Lemma~\ref{lem:handleaddition}, to determine when the curves of $\calC$ are disk-busting or primitive in $H_+$ and $H_-$. Let $c_\nu, c_\mu, c_\lambda$ be the three arcs of $\calc$.

\begin{figure}
\centering
\includegraphics[width=6in]{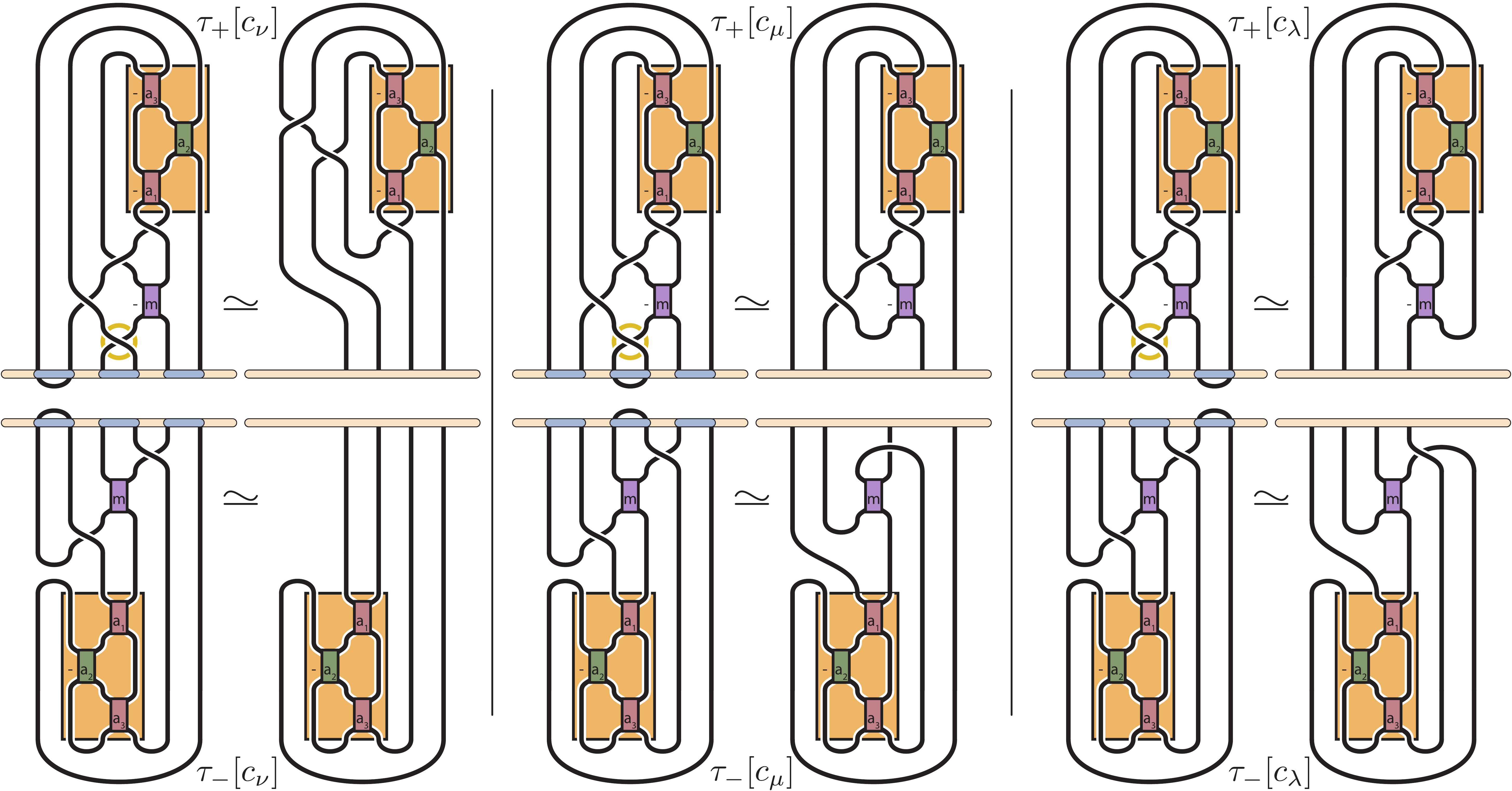}
\caption{Shown in the case that $n=3$, the tangles resulting from the arc-closures of the tangles $\tau_+$ and $\tau_-$ along the components of $\calc$ individually are isotoped to simpler, alternating diagrams. }
\label{fig:arcadditions}
\end{figure}

The bottom left of Figure~\ref{fig:arcadditions} shows that $\tau_-[c_\nu]$ is a rational tangle. (The diagrams shown for $\tau_-[c_\nu]$ are not reduced. Reductions would yield a crossingless diagram since the tangle is rational.)  Hence $C_\nu$ is primitive in $H_-$.   The rest of Figure~\ref{fig:arcadditions} shows that the other five tangles $\tau_+[c_\nu], \tau_+[c_\mu], \tau_+[c_\lambda], \tau_-[c_\mu], \tau_-[c_\lambda]$ have reduced (because $m\geq 2$), connected, alternating diagrams.  Direct inspection further shows that these five diagrams are each locally trivial and indivisible as well.  (While Figure~\ref{fig:arcadditions} is drawn with $n=3$, one may observe that the result continues to hold for larger integers $n$.) Then according to \cite[Theorem~1.2]{shimokawa-alternatingtangles} the tangles are prime.  Hence they are essential.  By Lemma~\ref{lem:tanglehandleaddition}, the corresponding components of $\calC$ in $H_+$ and $H_-$ are disk-busting.

We next show that $H_+[\calC]$ is neither a lens space (including $S^1 \times S^2$),  $\R P^3 \# \R P^3$, nor a prism manifold.

Figure~\ref{fig:closedtauplus} (Left) shows the link $\tau_+[\calc]$ 
with an isotopy to a simpler
configuration in which the diagram is reduced, alternating, and prime.
Since an alternating diagram of a non-prime link is non-prime
\cite{menasco}, $\tau_+[\calc]$ is a prime link. By Lemma 7.2  (for $0$--tangles), $H_+[\calC]$ is neither
$\R P^3 \# \R P^3$ nor $S^1 \times S^2$.

Figure~\ref{fig:closedtauplus} (Center) shows a rational tangle replacement (by setting $m=0$) of distance $\Delta=m$ that produces a reduced, non-split alternating diagram of a two-bridge link. Because the double branched cover of a two-bridge link is a lens space, the Montesinos Trick shows that $H_+[\calC]$ contains a knot with a Dehn surgery to a lens space (other than $S^1 \times S^2$).  In particular the exterior $X$ of this knot is the double branched cover of the  exterior tangle $\chi$ of the $-m$--tangle in $\tau_+[\calc]$ shown in  Figure~\ref{fig:closedtauplus} (Right).

Observe that this diagram of the tangle $\chi$ in  Figure~\ref{fig:closedtauplus} (Right) is reduced, connected, locally trivial, indivisible,  and alternating.  By \cite[Theorem~1.2]{shimokawa-alternatingtangles} again, $\chi$ is a prime tangle.
Since it is not crossingless, $\chi$ is not a rational tangle \cite[Corollary 3.2]{Thistlethwaite}.  Inspection shows there is no ``visible'' essential Conway disk in this diagram of $\chi$, so $\chi$ does not represent the sum of two rational tangles; see the last paragraph of \cite[\S3]{Thistlethwaite}.  
\cite[Theorem 4.2]{shimokawa-parallelstrings} shows that the tangle $\chi$ is atoroidal.  
Further inspection also reveals that there are no ``visible'' or ``hidden'' essential Conway spheres in this diagram of $\chi$, so the tangle has no essential Conway sphere; see \cite[\S3]{Thistlethwaite} or consider alternating rational tangle fillings of $\chi$ and apply \cite{menasco}.  
Therefore with Lemmas~\ref{lem:primetangles} and \ref{lem:nonsimpletangles} this implies that $X$, the double branched cover of $\chi$, is neither a solid torus, a Seifert fibered space over the disk,  a toroidal manifold, nor a reducible manifold.  (Recall that a Seifert fibered space over the disk is either toroidal, annular, or a solid torus.)

On the other hand, together with the Cyclic Surgery Theorem \cite{CGLS} and the Finite Surgery Theorems \cite{BZ}, if $H_+[\calC]$ were a lens space or a prism manifold, $X$ should either be a solid torus, a Seifert fibered space over the disk with two exceptional fibers,  or a union of a cable space and a Seifert fibered space over the disk with at most two exceptional fibers.  In this latter case, $X$ is either toroidal, one of the former two, or a connected sum of a lens space and a solid torus.
This however contradicts our previous determination about $X$.

\begin{figure}
\centering
\includegraphics[width=4in]{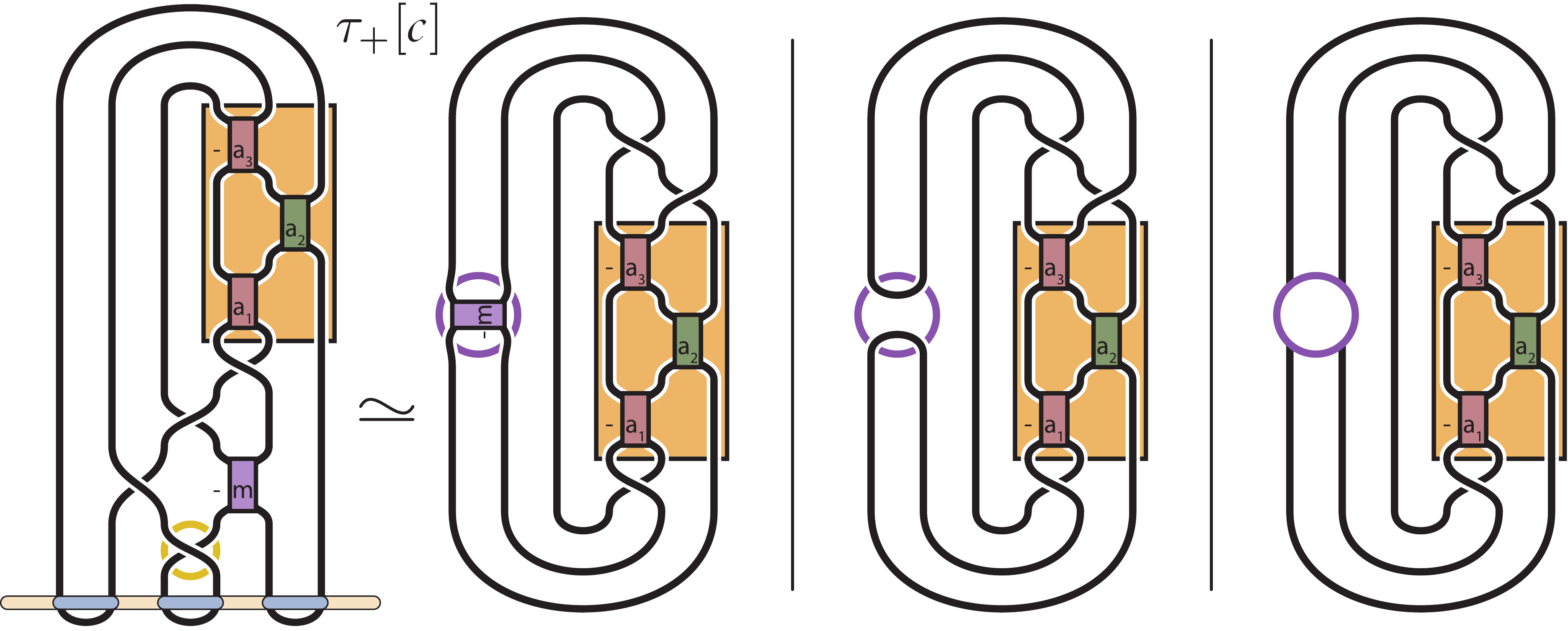}
\caption{(Left) The $\calc$--closure of $\tau_+$ is simplified to a reduced, prime, alternating link. (Center) A rational tangle replacement of distance $\Delta = m$ produces a two-bridge link with a non-split reduced alternating diagram. (Right) The exterior of the $-m$--twist tangle is the tangle $\chi$, shown with a reduced alternating diagram.}
\label{fig:closedtauplus}
\end{figure}

The above goes to imply that $H_+[\calC]$, the double branched cover of $\tau_+[\calc]$, is not a lens space or a prism manifold.   Since we have already shown that the  components of $\calC$ are each disk-busting in $H_+$, Lemma~\ref{lem:sepdiskbust} implies that $\calC$ is then $8$--disk-busting in $H_+$. 
\end{proof}

\begin{lemma}\label{lem:hypotheses}
The manifold $M= H_+ \cup_Q H_-$ is
\begin{enumerate}
	\item  a simple $3$--manifold
\end{enumerate} 
in which 
\begin{enumerate}[resume]
\item $(H_+,Q)$ and $(H_-,Q)$ are not homeomorphic,
\item $Q$ is incompressible and boundary incompressible, and
\item any properly embedded pair of pants in $M$ is either compressible, $\bdry$--parallel, isotopic to $Q$, or non-separating and can be isotoped to be disjoint from some component of $\bdry Q$.
\end{enumerate}
\end{lemma}

\begin{proof}
By Lemma~\ref{lem:Qbusting}, one component of $\bdry Q$ is primitive in $H_-$ while no components of $\bdry Q$ are primitive in $H_+$.  Hence $(H_+,Q)$ and $(H_-,Q)$ are not homeomorphic.

Due to Lemma~\ref{lem:Qbusting}, Lemma~\ref{lem:disktoannulusbusting} shows that $\bdry Q$ is $6$--disk-busting and annulus busting in $H_+$ while Lemmas~\ref{lem:primbustbustis5diskbusting} and \ref{lem:primitivetoannulusbusting} show that $\bdry Q$ is  $6$--disk-busting and annulus-busting in $H_-$.  Therefore Lemma~\ref{lem:Qess} shows that $Q$ is incompressible and
boundary incompressible in $M$, and Lemma~\ref{lem:Msimple} shows that $M$ is simple.

Finally because Lemma~\ref{lem:Qbusting} also shows that $\bdry Q$ is $8$--disk-busting in $H_+$, Lemma~\ref{lem:Qunique} gives the final desired property.
\end{proof}

\begin{theorem}\label{thm:asymlspaceknots}
  Take $\alpha$ as above with $n\geq 3$ and also $m\geq 3$.
 Take integers $p,q \geq 0$ and $p',q' \geq 1$ such that $|pq'-p'q|=1$.
  Then for suitably large integers $N$, the $\tfrac{p+Np'}{q+Nq'}$--lashing
  of $P=P^{\alpha,m}$ (with respect to $\bdry P = \calC = \{C_\nu, C_\mu, C_\lambda\}$) is  an asymmetric hyperbolic knot with an longitudinal surgery to the double branched cover of a non-split alternating link.  In particular, such a lashing is an asymmetric L-space knot. 
\end{theorem}

\begin{proof}
	Observe that for fixed integers $p,q,p',q'$ satisfying $|pq'-p'q|=1$ and $N \in \Z$, the knot $K$ of slope $p/q$ and the knot $L$ of slope $p'/q'$ intersect once in the once-punctured torus $T$.  Therefore the knots $K^N$ of slope $\tfrac{p+Np'}{q+Nq'}$ in $T$ are obtained by twisting $K$ along $L$ and hence form a basic twist family as in Section~\ref{sec:twistfamily}. 
	
	For $N\geq 0$, we have that $\tfrac{p+Np'}{q+Nq'}>0$ since $p,p',q,q'\geq0$.  Hence the lashing $K^N$ has a longitudinal surgery to  the double branched cover of a non-split alternating link and is thus an L-space knot by Theorem~\ref{thm:lspaceknot}.

	Lemma~\ref{lem:hypotheses} ensures the four numbered hypotheses of Theorem~\ref{thm:mainasymmetry} are satisfied.   Since $p',q' \geq 1$, $L$ is not isotopic to $\mu$ or $\lambda$ so that the final hypothesis of Theorem~\ref{thm:mainasymmetry} is satisfied. Thus the lashing $K^N$ is an asymmetric hyperbolic knot for suitably large $N$.	
\end{proof}

\section{Asymmetric L-space knots in lens spaces and $S^1 \times S^2$}\label{sec:generalization}
Here we explain how to adapt the above construction of asymmetric L-space knots in $S^3$ to produce asymmetric L-space knots in any lens space, including $S^1 \times S^2$.   Since this ends up being a mild modification, we will only discuss the necessary changes and impacts on relevant lemmas and theorems above and present the result in Theorem~\ref{thm:extension}.

We may generalize Figure~\ref{fig:almostalternating}, by using the $3$--braid $\omega$ in the stead of $\alpha^{-1}$ to form a link diagram $J^{\alpha,\omega,m}$ as depicted in the center of Figure~\ref{fig:Jalphaomega}. Here we take $3$--braids of the form:

\[ 
\alpha = \prod_{i=n}^1 \sigma_{\bar{i}}^{\epsilon_{i} a_{i}}  =\sigma_{\bar{n}}^{\epsilon_n a_n}  \dots \sigma_1^{-a_3} \sigma_2^{a_2} \sigma_1^{-a_1}
\quad  \mbox{ and } \quad
\omega = \prod_{i=1}^r \sigma_{\bar{i}}^{-\epsilon_{i} z_{i}}  =\sigma_1^{z_1} \sigma_2^{-z_{2}}   \sigma_1^{z_3} \dots \sigma_{\bar{r}}^{-\epsilon_r z_{r}} 
\tag{$\ast$}
\]
where $n,r \geq 1$ are integers, $\epsilon_j=(-1)^j$, and $\bar{j}$ is $1$ or $2$ according to the parity of $j$.
 The braid $\bar{\omega}$ is obtained from $\omega$ by swapping $\sigma_1$ and $\sigma_2$.  
When the integers $a_i, z_j$ are non-negative then $\alpha, \bar{\omega}$ are each 
alternating $3$-braids with negative twist boxes on the left and positive on the right. Examples of braids $\alpha$ and $\bar{\omega}$ for odd $n$ and $r$ are shown on the right of Figure~\ref{fig:Jalphaomega}.

Then for any integer $m$ the diagram  $J^{\alpha,\omega,m}$ depicts a two-bridge link that is a plat closure of the $3$--braid $\sigma_1^{-1} \alpha \omega$  as illustrated by the isotopies in Figure~\ref{fig:Jalphaomega} (Center) from its right to left.  (Here we take the closure of any $3$--braid $\eta$ as in Figure~\ref{fig:Jalphaomega} (Left).) Observe that if we further take all the integers $a_i$ and $z_j$ and the integer $m$ to be non-negative, then  $J^{\alpha,\omega,m}$ is an almost alternating diagram of this two-bridge link.

\begin{lemma}\label{lem:platclosureof2bridge}
	Any two bridge link may be expressed as the plat closure of  the $3$--braid $\sigma_1^{-1} \alpha \omega$ as shown in Figure~\ref{fig:Jalphaomega} (Left) with $\alpha$ and $\omega$ of the form ($*$), both $n,r \geq 3$, and all the integers $a_i$ and $z_j$ positive. In particular, $J^{\alpha,\omega,m}$ is an almost alternating diagram of this two-bridge link.
\end{lemma}

\begin{proof}
	First, take a $3$--braid $\alpha' =\displaystyle \prod_{i=n'}^{1} \sigma_{\bar{i}}^{\epsilon_{i} a'_{i}}$ for some integer $n' \geq 3$, positive coefficients $a'_{i}$, and where $\epsilon_i=(-1)^i$.

	Now, any two bridge link $L$ may be expressed as the plat closure of some alternating $3$--braid as shown in Figure~\ref{fig:Jalphaomega} (Left) using its Conway normal form \cite{conway}.  In particular $L$ is the plat closure of an alternating $3$--braid $\xi$.  
	Because the $3$--braids $\eta$ and $\sigma_2^N \eta$  have plat closures (as in Figure~\ref{fig:Jalphaomega} (Left)) giving isotopic links for any integer $N$, the  $3$--braid $\xi$ may be taken to begin with a non-zero power of $\sigma_1$ unless $\xi$ is the trivial braid. 
	 
	If $\xi$ begins with a negative power of $\sigma_1$, let $\xi'$ be the alternating $3$--braid such that $\xi = \sigma_1^{-1} \xi'$.  Then set $\alpha = \xi' \alpha'$ and $\omega = (\alpha')^{-1}$.  
	
	If $\xi$ begins with a positive power of $\sigma_1$, set $\alpha = \alpha'$ and $\omega = (\alpha')^{-1} \sigma_1 \xi$. 
	
	If $\xi$ is the trivial braid, then set $\alpha = \alpha'$ and $\omega = (\alpha')^{-1} \sigma_1$.

	In each of these cases $\sigma_1^{-1} \alpha \omega = \xi$ as needed.
\end{proof}

\begin{figure}
\centering
\includegraphics[width=4.5in]{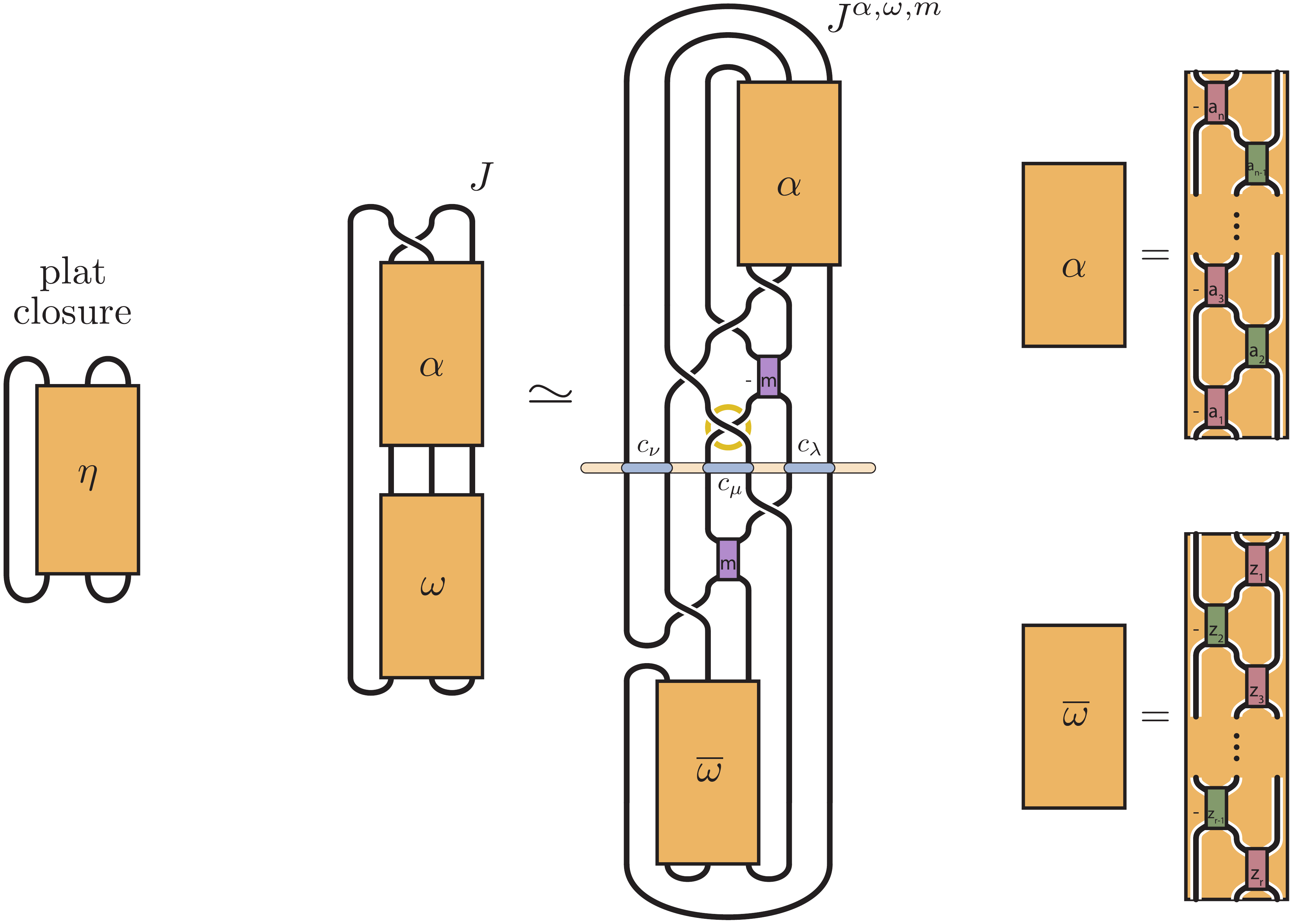}
\caption{(Left) The plat closure of a $3$--braid $\eta$. (Center) Replacing $\alpha^{-1}$ in Figure~\ref{fig:almostalternating} with another braid $\omega$ enables the production of almost alternating diagrams $J^{\alpha,\omega,m}$  of other two-bridge links, analogous to our almost alternating unknot diagrams. (Right) The braids $\alpha$ and $\bar{\omega}$ are shown for $\alpha, \omega$ in form ($*$) with $n,r$ odd.}
\label{fig:Jalphaomega}
\end{figure}

\begin{theorem} \label{thm:extension}
	Let $Y$ be a lens space, including $S^1 \times S^2$.  Then $Y$ contains infinitely many asymmetric hyperbolic L-space knots with a non-trivial alternating surgery.
\end{theorem}

\begin{proof}
	A lens space $Y$ is the double branched cover of some two-bridge link.   By Lemma~\ref{lem:platclosureof2bridge}, this two-bridge link may be taken to be the plat closure (as in Figure~\ref{fig:Jalphaomega} (Left)) of an alternating  $3$--braid $\sigma_1^{-1} \alpha \omega$ where $\alpha$ and $\omega$ each have at least $3$ twist regions.  Figure~\ref{fig:Jalphaomega} (Center) shows this link then has an almost alternating $3$--bridge presentation $J^{\alpha,\omega,m}$  for any choice of integer $m$.  Choose $m \geq 2$.  The three arcs $\calc = \{c_\nu, c_\mu, c_\lambda\}$ lift to the triple of curves $\calC = \{C_\nu, C_\mu, C_\lambda\}$ in the genus $2$ Heegaard surface $\Sigma$ that is the lift of the bridge sphere.
	Let $P=P^{\alpha,\omega,m}$ be one of the pairs of pants in $\Sigma$ bounded by $\calC$, and let $Q$ be the complementary pair of pants.
	Theorem~\ref{thm:lspaceknot} now extends directly to show that for any non-negative slope $p/q$, framed surgery on the $p/q$--lashing of $P$ (with respect to $\bdry P = \calC = \{C_\nu, C_\mu, C_\lambda\}$ as shown in Figure~\ref{fig:TinPxIprojection}) produces the double branched cover of an alternating link.  Hence a non-negative sloped $p/q$--lashing is an L-space knot.
	
	As above take the exterior of $P$, $Y \cut H_P$, to be the manifold $M=H_+ \cup_Q H_-$  where  $Q$ separates $M$ into the handlebodies $H_+$ and $H_-$.  Above and below the bridge sphere of the two-bridge link $J^{\alpha,\omega,m}$ as shown in Figure~\ref{fig:Jalphaomega} are two three-strand rational tangles $\tau_+=(B^3,t_+)$ and $\tau_-=(B^3,t_-)$. The handlebodies $H_+$ and $H_-$ are the double branched covers of these tangles and the curves $\calC = \bdry Q$ in their boundaries are the lifts of the arcs $\calc$ in the bridge sphere.  Observe that the proof of Lemma~\ref{lem:Qbusting} applies equally well with $\omega$ in place of $\alpha^{-1}$ as long as $r \geq 3$ (so that $\omega$ also has at least three twist regions) to give the same results about $\calC$ in $H_-$.  Therefore Lemma~\ref{lem:hypotheses} also continues to hold.

	Now we follow Theorem~\ref{thm:asymlspaceknots}.  Let $p,q,p',q'$ be  non-negative integers with $p',q'\geq 1$ such that $|pq'-p'q|=1$, and let $K^N$ be the $\tfrac{p+Np'}{q+Nq'}$--lashing of $P=P^{\alpha,\omega,m}$ (with respect to $\bdry P = \calC = \{C_\nu, C_\mu, C_\lambda\}$).
	 Then  for any $N\geq 0$ then knot $K^N$ is an L-space knot. Since Lemma~\ref{lem:hypotheses} ensures the four numbered hypotheses of Theorem~\ref{thm:mainasymmetry} are satisfied, then for each suitably large integer $N$ the lashing $K^N$ is also an asymmetric hyperbolic knot with a surgery to the double branched cover of an alternating link.	
\end{proof}

\section{Tunnel number}

\begin{proposition}\label{prop:tnofgenus2lashing}
	Let $K$ be a $p/q$--lashing of a genus $2$ Heegaard splitting.  Then the tunnel number of $K$ is at most $3$.
\end{proposition}

\begin{proof}
	Let $\Sigma = P \cup Q$ be the genus $2$ Heegaard surface in which $K$ is a lashing of $P$.  Since
	$K$ is a core of the handlebody $P\times I$, it follows that $K$ is a core of the genus $4$ handlebody $\Sigma' \times I$ where $\Sigma'$ is the Heegaard surface $\Sigma$ punctured once.   The complement of $\Sigma' \times I$ is then the boundary connect sum of the two genus $2$ handlebodies $H_+$ and $H_-$.  This gives a genus $4$ splitting of the exterior of $K$.  Hence the tunnel number of $K$ is at most $3$.
\end{proof}

\begin{proposition}\label{prop:tn2lashing}
	Let $K$ be an asymmetric hyperbolic $p/q$--lashing as constructed here.  Then the tunnel number of $K$ is $2$.
\end{proposition}

\begin{proof}
	Recall the construction of our family of $p/q$--lashings in $S^3$ from Section \ref{sec:construction}. The proof of Lemma~\ref{lem:Qbusting} shows that the curve $C_\nu$ of $\bdry Q$ is primitive in $H_-$.  This implies the manifold $M= H_+ \cup_{\nbhd(C_\nu)} H_-$ has Heegaard genus $3$.	As above, since	$K$ is a core of the handlebody $P\times I$, it follows that $K$ is a core of the genus $3$ handlebody $(\Sigma - \nbhd(C_\nu)) \times I$.  Thus we have a genus $3$ splitting of the exterior of $K$.  Hence its tunnel number is at most $2$.  
	Since any knot with tunnel number 1 is strongly invertible, the asymmetry of $K$ implies its tunnel number is exactly $2$.
\end{proof}

\begin{proposition}\label{prop:dbcgenus3}
	The double branched cover of $J^{\alpha,m}_{\calc}(+1, -q/(p+q), -p/(p+q))$  has Heegaard genus at most $3$.  When $p/q = +1$, the double branched cover of $J^{\alpha,m}_{\calc}(+1, -1/2, -1/2)$ has Heegaard genus at most $2$.
\end{proposition}

\begin{proof}
	The link $J^{\alpha,m}_{\calc}(+1, -q/(p+q), -p/(p+q))$ decomposes as the union of two trivial $4$--strand  tangles as indicated in Figure~\ref{fig:4strandrationaltangles}.  Hence its bridge number is at most $4$.  Thus its double branched cover has Heegaard genus at most $3$.
	
	Because the $+1$--tangle and the $-1/2$--tangle are only vertical twists, the link $J^{\alpha,m}_{\calc}(+1, -1/2, -1/2)$ is easily seen to be $3$--bridge.  Hence its double branched cover has Heegaard genus at most $2$.
\end{proof}

\begin{figure}
	\centering
	\includegraphics[width=5in]{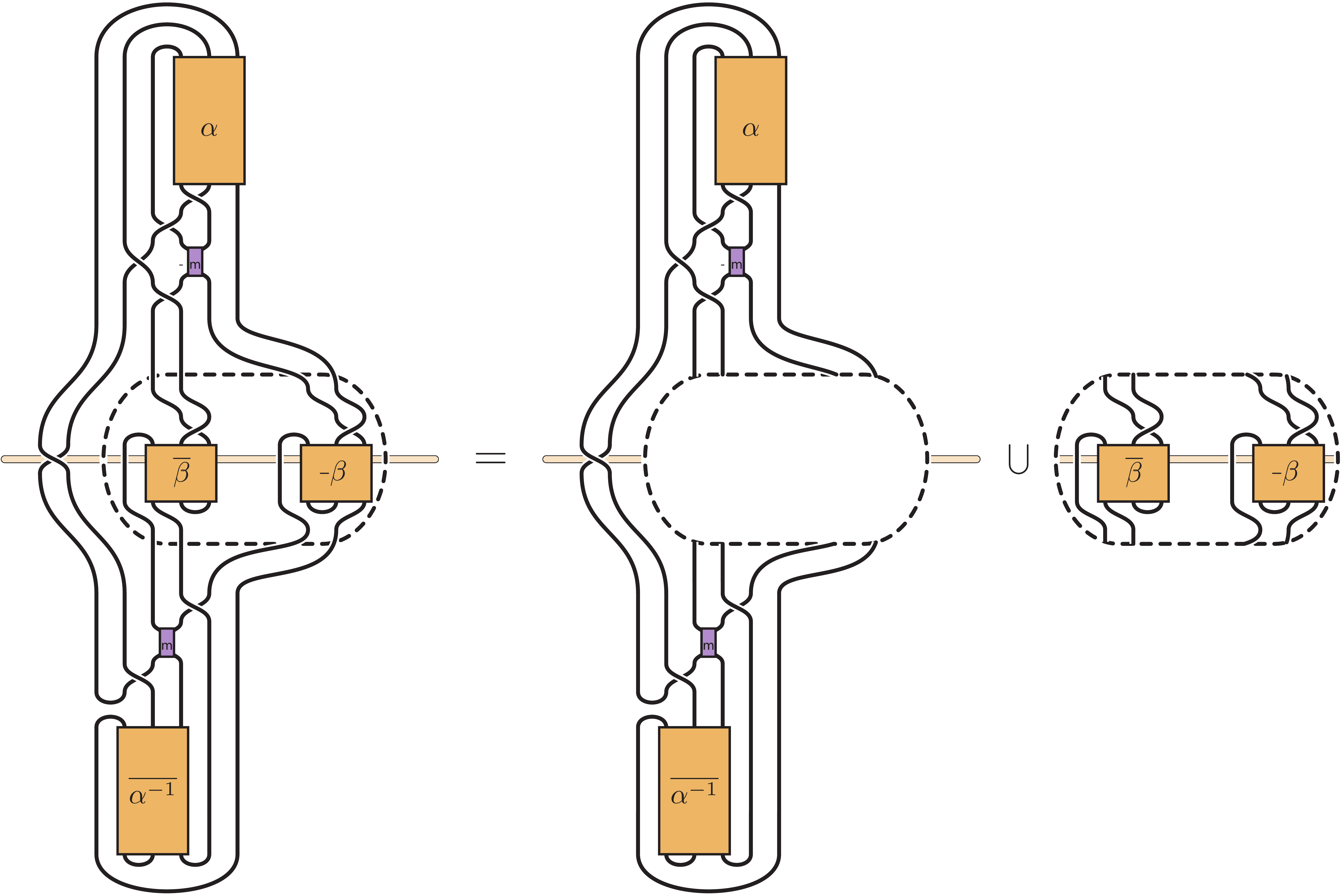}
	\caption{The link $J^{\alpha,m}_{\calc}(+1, -q/(p+q), -p/(p+q))$ splits into two four strand tangles as shown.  One fairly easily sees that each of these two tangles is freely isotopic to the trivial $4$--strand tangle.}
	\label{fig:4strandrationaltangles}
\end{figure}

\section{Surgery descriptions and examples}

Here we first obtain a surgery description of the lashings constructed in Section~\ref{sec:construction},
specifically for when $\alpha$ is of length at most three.  (One may use this as a model for obtaining surgery descriptions when $\alpha$ has length greater than $3$.) These lashings contain a family of asymmetric L-space knots that each have a surgery to the double branched cover of an alternating link.  The surgery description allows us to input these lashings into SnapPy.

  Then, in a different manner, we  obtain a Dehn twist presentation of the pair of pants $P^{\alpha,m}$ in the genus $2$ Heegaard surface for $S^3$. This allows us to obtain a framed oriented train track that carries our L-space knot lashings for any $\alpha$ as in Section~\ref{sec:construction} and $m\geq 0$.  Then we isotop this train track into a form that carries closed positive braids. Using this we are able to calculate the genus and alternating surgery slope of any positive $p/q$--lashing of $P^{\alpha,m}$.  The explicit presentation as a closed braid also gives another method for inputting these lashings into SnapPy.

  Using SnapPy along with Regina and Sage, both the surgery and closed braid descriptions enable us to confirm that many small examples not covered by Theorem~\ref{thm:asymlspaceknots} are also asymmetric hyperbolic L-space knots with the expected surgery to the double branched cover of the correct alternating link.   SnapPy allows us to also confirm that our two surgery descriptions and positive braid descriptions agree for these small examples.   
  Furthermore we explicitly present two small asymmetric examples.

\subsection{Surgery description of lashings}\label{sec:surgdesc}
\begin{figure}
	\centering
	\includegraphics[width=4in]{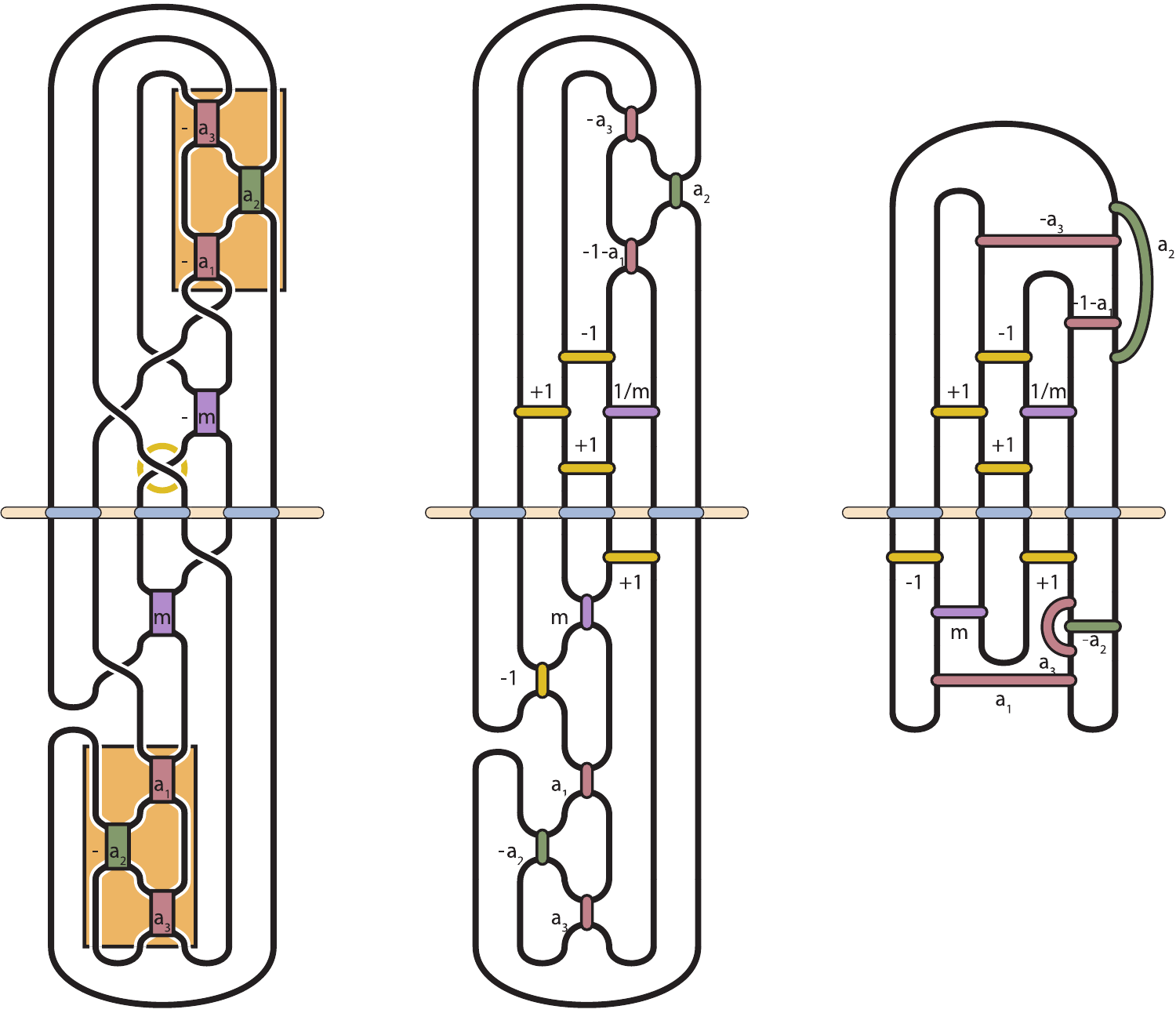}
	\caption{(Left) Twist regions of Figure~\ref{fig:almostalternating} with the choice $\alpha = \sigma_1^{-a_3}\sigma_2^{a_2} \sigma_1^{-a_1}$ are (Center) traded for blackboard framed arcs with rational tangle replacement instructions.   (Right) The diagram is then simplified by a planar isotopy that more clearly exhibits the resulting link as a planar unknot while retaining a sense of the original structure.}
	\label{fig:almostaltRTRdescription}
\end{figure}

\begin{figure}
	\centering
	\includegraphics[width=6in]{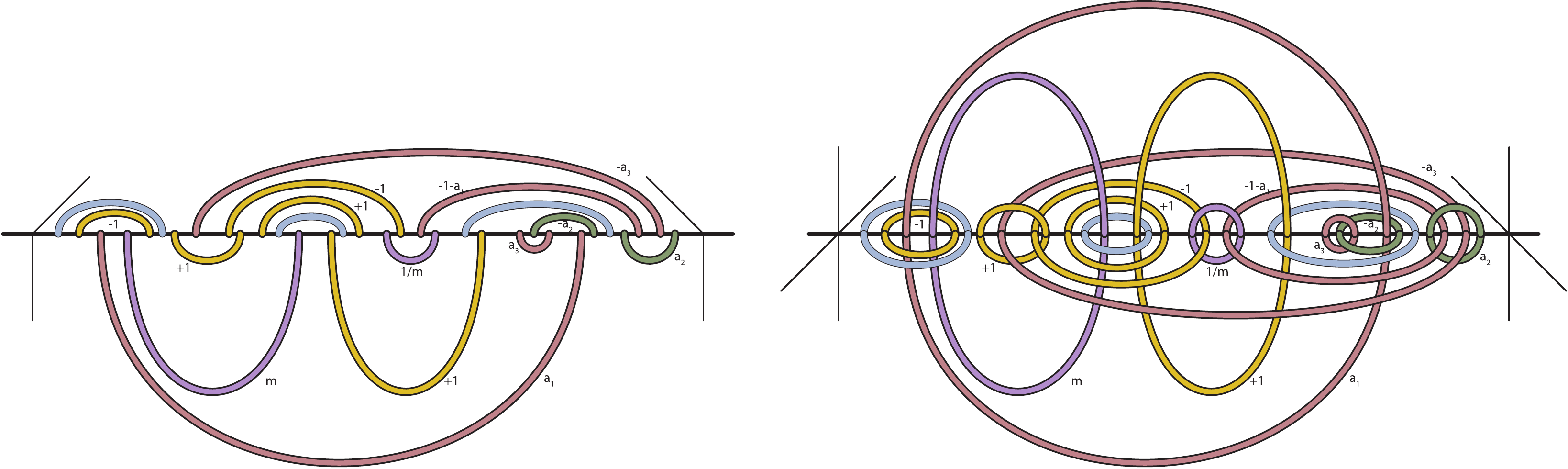}
	\caption{(Left) The planar unknot at the end of Figure~\ref{fig:almostaltRTRdescription} is mapped to a horizonal axis (with a point at infinity) so that its inside becomes a horizontal half-plane into the page while its outside becomes a downward vertical half-plane.  (Right) This facilitates the visualization of the lifts of the arcs in the double cover of $S^3$ branched over the unknot.  The rational tangle replacement instructions lift to Dehn surgery instructions and the arcs $c_\nu, c_\mu, c_\lambda$ lift to the knots $C_\nu, C_\mu, C_\lambda$.  }
	\label{fig:RTRsurgdescDBC-horiz}
\end{figure}

\begin{figure}
	\centering
	\includegraphics[width=6in]{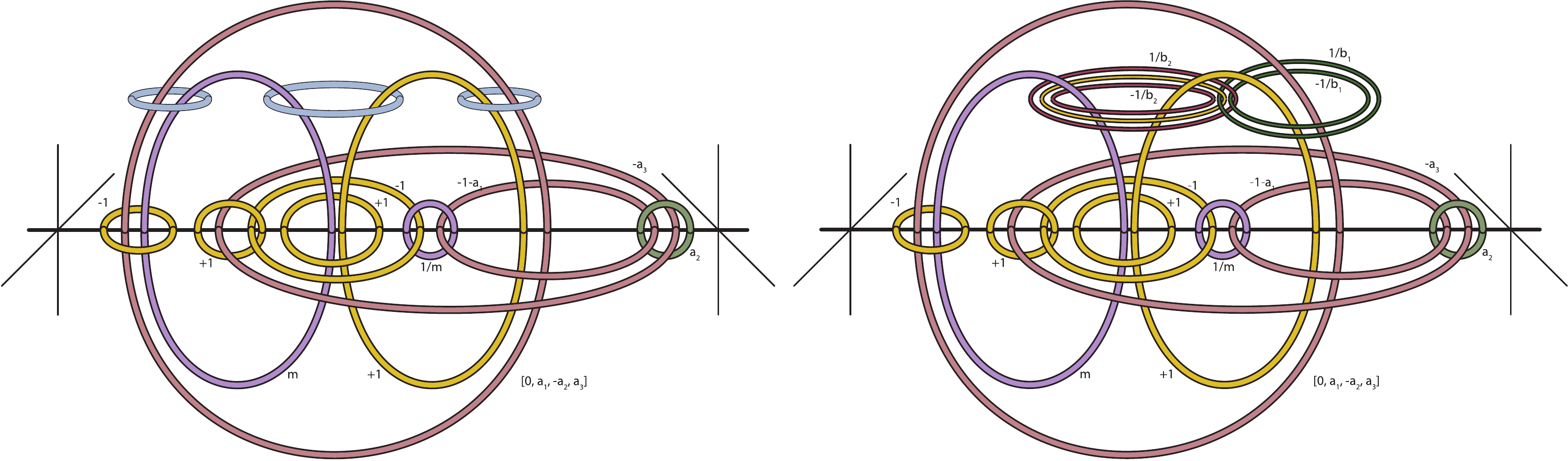}
	\caption{(Left) The genus $2$ Heegaard surface the knots $C_\nu, C_\mu, C_\lambda$ can be traced through its lift of the bridge sphere at the beginning of Figure~\ref{fig:almostaltRTRdescription}.  Here we lift these curves upwards to recognize the punctured horizontal plane  they bound (with the point at infinity) as the pair of pants $P$ in which our lashings occur. (Right) This triple of knots is then traded for the interlaced stack of knots as in Figure~\ref{fig:Kgammasugery} to give surgery instructions for a $p/q$--lashing.  Recall the convention $p/q=[-b_n, b_{n-1}, \dots, b_2, -b_1]$  with $b_j \geq 0$. }
	\label{fig:fullsurgdesc-horiz}
\end{figure}

Consider the case where $\alpha$ is of length at most three, $\alpha = \sigma_1^{-a_3}\sigma_2^{a_2} \sigma_1^{-a_1}$.
Figures~\ref{fig:almostaltRTRdescription}, \ref{fig:RTRsurgdescDBC-horiz}, \ref{fig:fullsurgdesc-horiz} illustrate the passage from the tangle presentation of $\calc$ of Figure~\ref{fig:almostalternating} to a surgery presentation of their lift $\calC$ and finally a surgery presentation of a family of lashings of the pair of pants $P$ that they bound.

This surgery presentation allows us to easily calculate the homology of the result of the framed surgery on the lashings.  Observe that the link on the right of Figure~\ref{fig:fullsurgdesc-horiz} is a union of unknots.  If we orient them all clockwise in the diagram, shown with $b_i=0$ for $i\geq 3$, then we obtain the linking matrix for the $r$--surgery on the knot:

\[
\left(
\begin{array}{CCCCCCCCCCC|CCCCC}
 A & 0 & 0 & -1 & 0 & 0 & 1 & 1 & 0 & 0 & 0 & 0 & 1 & 0 & 1 & 0 \\
 0 & m & 0 & -1 & 1 & 1 & 0 & 1 & 0 & 0 & 0 & 1 & 0 & 1 & 0 & 1 \\
 0 & 0 & 1 & 0 & -1 & -1 & 1 & 0 & 0 & 0 & 0 & -1 & 1 & -1 & 1 & -1 \\
 -1 & -1 & 0 & -1 & 0 & 0 & 0 & 0 & 0 & 0 & 0 & 0 & 0 & 0 & 0 & 0 \\
 0 & 1 & -1 & 0 & 1 & 0 & 0 & 0 & 0 & 0 & 0 & 0 & 0 & 0 & 0 & 0 \\
 0 & 1 & -1 & 0 & 0 & -1 & 0 & 0 & 1 & -1 & 0 & 0 & 0 & 0 & 0 & 0 \\
 1 & 0 & 1 & 0 & 0 & 0 & -a_1-1 & 0 & 0 & 1 & -1 & 0 & 0 & 0 & 0 & 0 \\
 1 & 1 & 0 & 0 & 0 & 0 & 0 & -a_3 & 1 & 0 & -1 & 0 & 0 & 0 & 0 & 0 \\
 0 & 0 & 0 & 0 & 0 & 1 & 0 & 1 & 1 & 0 & 0 & 0 & 0 & 0 & 0 & 0 \\
 0 & 0 & 0 & 0 & 0 & -1 & 1 & 0 & 0 & \frac{1}{m} & 0 & 0 & 0 & 0 & 0 & 0 \\
 0 & 0 & 0 & 0 & 0 & 0 & -1 & -1 & 0 & 0 & a_2 & 0 & 0 & 0 & 0 & 0 \\ \hline
 0 & 1 & -1 & 0 & 0 & 0 & 0 & 0 & 0 & 0 & 0 & -\frac{1}{b_2} & 0 & 0 & 0 & 0 \\
 1 & 0 & 1 & 0 & 0 & 0 & 0 & 0 & 0 & 0 & 0 & 0 & \frac{1}{b_1} & 1 & 0 & 1 \\
 0 & 1 & -1 & 0 & 0 & 0 & 0 & 0 & 0 & 0 & 0 & 0 & 1 & r & 0 & 0 \\
 1 & 0 & 1 & 0 & 0 & 0 & 0 & 0 & 0 & 0 & 0 & 0 & 0 & 0 & -\frac{1}{b_1} & 1 \\
 0 & 1 & -1 & 0 & 0 & 0 & 0 & 0 & 0 & 0 & 0 & 0 & 1 & 0 & 1 & \frac{1}{b_2} \\
\end{array}
\right)
\]
where $A=[0,a_1,-a_2,a_3]=\frac{a_1 a_2 a_3+a_1+a_3}{a_2 a_3+1}$.

Setting $a_1=a_2=a_3=m=1$ and taking $r \in \Z$, the surgered manifold has first homology of order
\[|-389-r - b_1 (563 + 778 b2) - b_1^2 (204 + 563 b_2 + 389 b_2^2)|\]
(which may be calculated from the associated framing matrix, see e.g.\ \cite[Section 1.1.6]{Saveliev-InvtsHomologyS3}).

With $b_2=0$, $b_1= n>0$, and $r \in \Z_+$ this has homology of order $389+r + 563 n + 204 n^2$.  Taking $r=0$ corresponds to the framed surgery of the lashing.

\subsection{Framed train tracks}
Beginning again from the presentation of $J=J^{\alpha,m}$ of Figure~\ref{fig:almostalternating}, we isotop $J$ so that the bridge sphere containing the arcs $\calc$ nearly mirrors the tangle above to the tangle below.  We coalesce the twistings of $\alpha$ and $m$ into rational tangle replacents along horizontal arcs, red and green for the left and right twisting of $\alpha$ and purple for $m$.  Then we flatten $J$ by a height preserving isotopy at the expense of twisting the arcs of $\calc$ in the bridge sphere.  This is done in Figure~\ref{fig:unravelJtounknot}.

\begin{figure}
	\centering
	\includegraphics[width=6in]{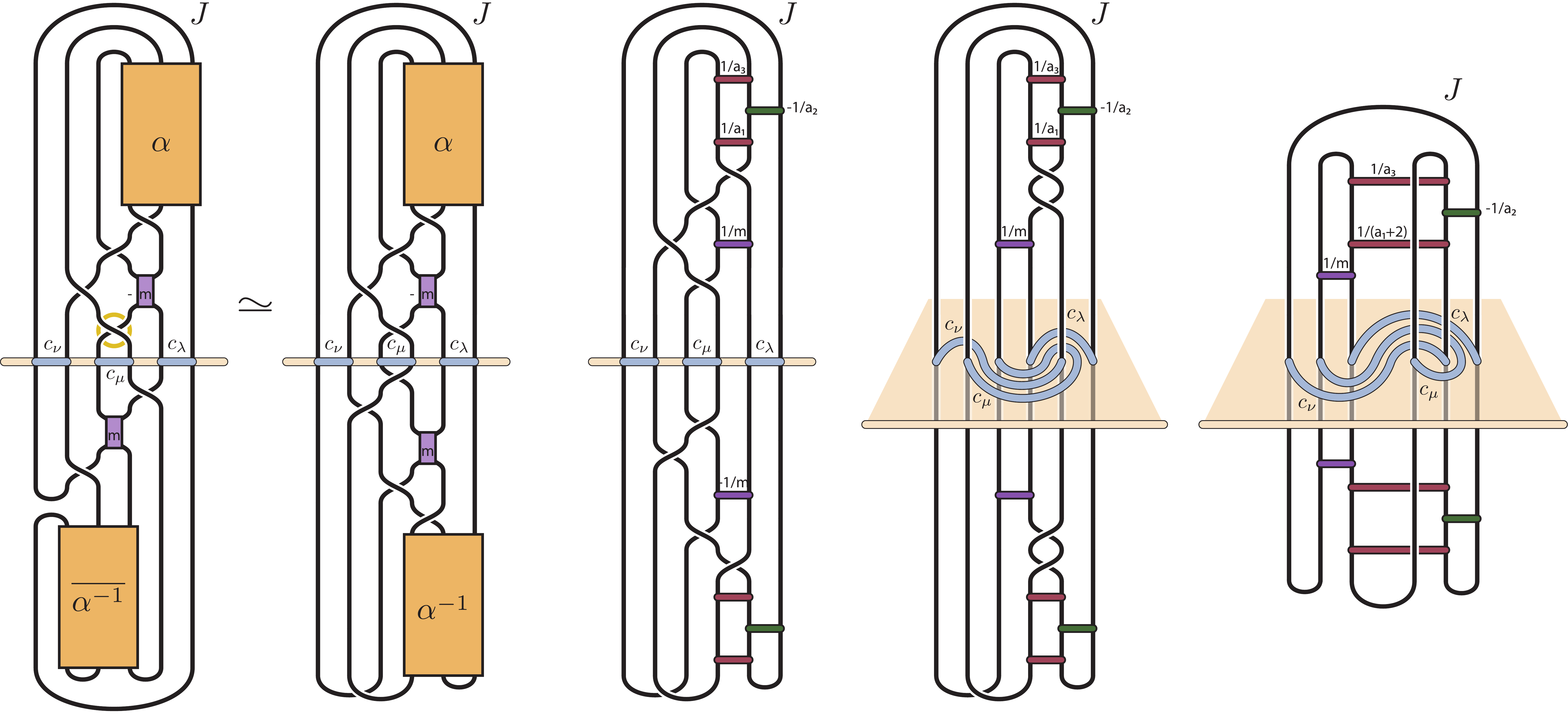}
	\caption{Isotopy, pairs of tangle replacements, and twisting transform the almost alternating unknot $J^{\alpha,m}$ into a planar $3$--bridge unknot where the arcs $c_\nu$, $c_\mu$, and $c_\lambda$ are twisted along the bridge sphere.}
	\label{fig:unravelJtounknot}
\end{figure}

With $J$ flattened and the twisting arcs for $\alpha$ and $m$ slid into the bridge sphere along with $\calc$, we take the double branched cover of $J$ and lift all these arcs to the simple closed curves in the genus $2$ Heegaard surface that is the lift of the bridge sphere.  This is shown on the left side of Figure~\ref{fig:branchcovertoPQ}.  On the right side, we show the result of an isotopy of $\calC$ (the lift of $\calc$) that more obviously bounds a pair of pants $P_0$, along with a purple curve in an annulus (to emphasize its framing) and darker once-punctured torus containing the red and green lifts of the red and green arcs as a basis.  The basis of the once-puncture torus for the lashings from Figure~\ref{fig:TinPxIprojection} is shown above and in the pair of pants $P_0$.

\begin{figure}
	\centering
	\includegraphics[width=6in]{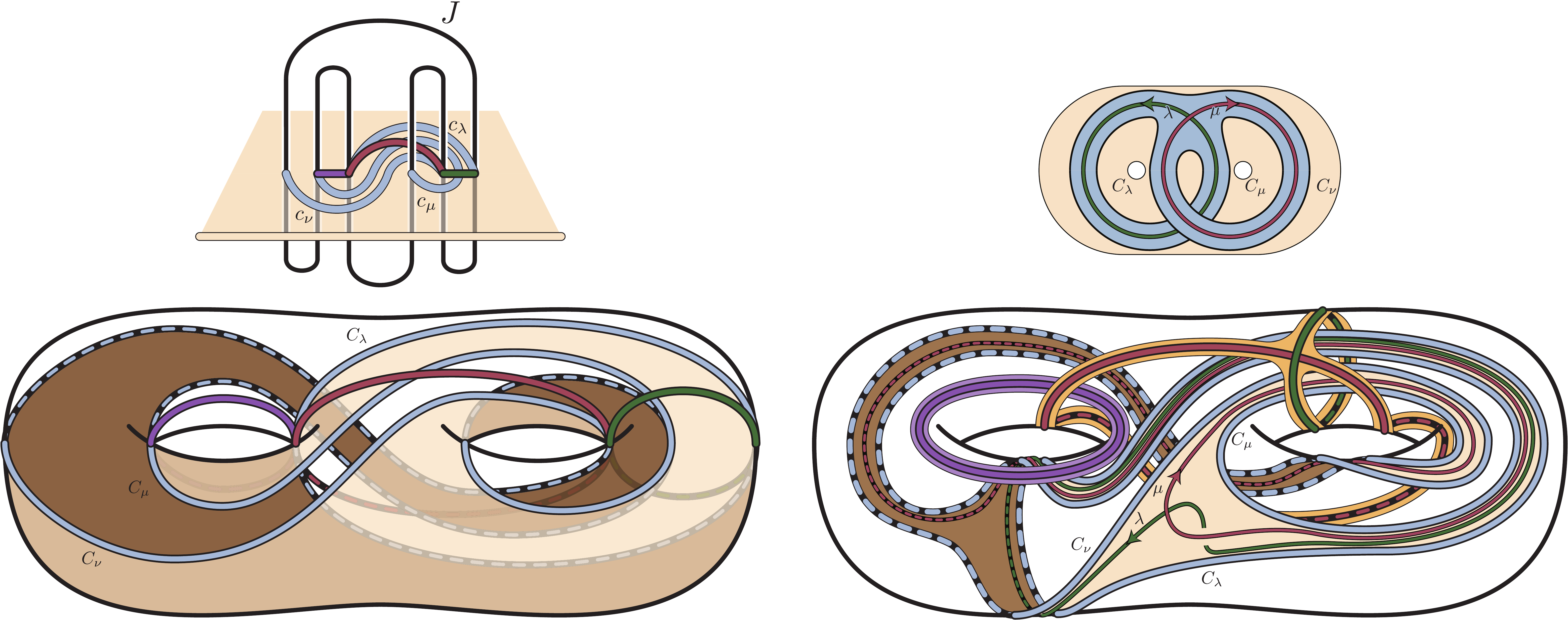}
	\caption{(Left, top and bottom) The double branched cover of the planar $3$--bridge unknot gives a genus $2$ Heegaard surface containing the lifts of the arcs in the bridge sphere.   (Right, bottom) After an isotopy, the curves $C_\nu$, $C_\mu$ and $C_\lambda$ more clearly bound the pair-of-pants $P$.   The basis of curves $\mu, \lambda$ in the once-punctured torus $T$ are shown in  $P$.  (Right, top) Figure~\ref{fig:TinPxIprojection} is included for reference. }
	\label{fig:branchcovertoPQ}
\end{figure}

The pair of pants $P=P^{\alpha, m}$ is then obtained by first performing $m$ left-handed Dehn twists of $P_0$ along the purple curve and then left-handed and right-handed Dehn twist along the red and green curves according to $\alpha$.  In Figure~\ref{fig:framedtraintrack} we show the result of these Dehn twists as an oriented train track for a $p/q$--lashing of $P$ with $p,q\geq 0$.  On the left, we retain the sense of the pair of pants $P_0$, the purple annulus and the once-punctured torus.  On the right, we retain just the framing of this train track and its embedding in the Heegaard surface as well as the weights of the branches in the case $m\geq0$ and $\alpha = \sigma_1^{-a_3} \sigma_2^{a_2} \sigma_1^{-a_1}$ for $a_1, a_2, a_3 \geq 0$.

\begin{figure}
	\centering
	\includegraphics[width=6in]{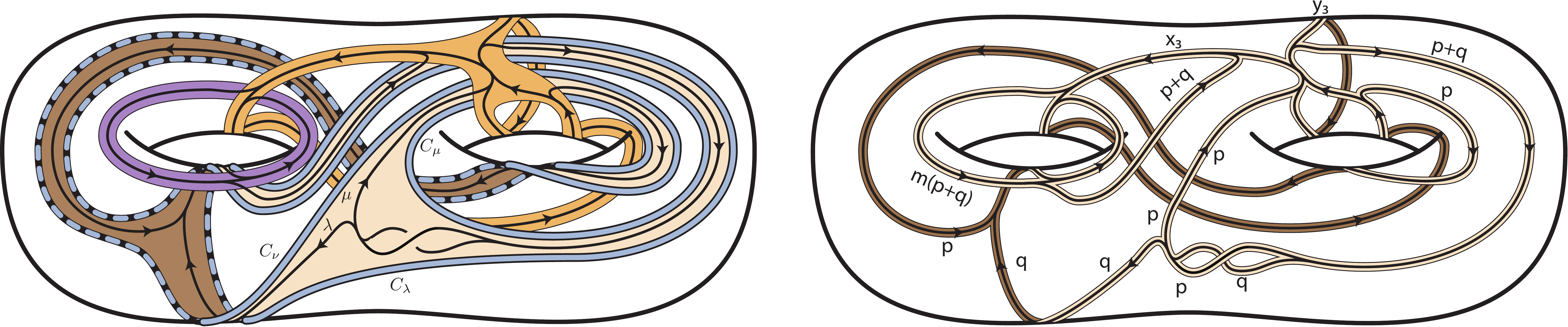}
	\caption{The $p/q$ lashing of $P$ (for $p,q>0$) is shown as a framed train track in the Heegaard surface.  Here $m \geq 0$ and $\alpha = \sigma_1^{-a_3} \sigma_2^{a_2} \sigma_1^{-a_1}$ where $a_1, a_2, a_3 \geq 0$ so that 
	$x_3 = ((a_3 a_2(a_1+2)+a_3+a_1 + 2)(m+2)-2 a_3 a_2)p+((a_3 a_2 (a_1+2)+a_3+a_1+2)(m+1)-a_3 a_2)q$ 
	and 
	$y_3 = (a_2(a_1+2)(m+2)-2a_2)p+(a_2(a_1+2)(m+1)-a_2)q$.
	} 
	\label{fig:framedtraintrack}
\end{figure}

More generally, consider using $\alpha=\sigma_{\bar{n}}^{\epsilon_n a_n}  \dots \sigma_1^{-a_3} \sigma_2^{a_2} \sigma_1^{-a_1}$ for positive integers $n$ and $a_1, \dots, a_n$ (i.e.\ having the form of ($\ast$) from the beginning of Section~\ref{sec:generalization}).  
Then, as one may check, the weights on the train track on the right side of Figure~\ref{fig:framedtraintrack} labeled $x_3,y_3$ become $x_n,y_n$ according to the following recursive formula for $n=2k$ or $2k+1$ with $k\geq 1$:
\[\tag{$\ast\ast$}
\begin{array}{lll}
\begin{cases}x_0= 0\\ y_0=0\end{cases} &
\quad &
\begin{cases}x_1= (a_1+2)(p+(m+1)(p+q))\\ y_1=0\end{cases} \\
\quad \\
  \begin{cases}x_{2k} = x_{2k-1}\\ y_{2k}=  y_{2k-1} + a_{2k}(x_{2k-1}-2p-q) \end{cases}  &
\quad &
  \begin{cases}x_{2k+1} = x_{2k} + a_{2k+1}(y_{2k}+p+(m+1)(p+q))\\ y_{2k+1}= y_{2k} \end{cases}
\end{array}
\]
Note that the switching in the train track of Figure~\ref{fig:framedtraintrack} 
requires that $x_n \geq m(p+q)$ which one easily checks, see the proof of Theorem~\ref{thm:positivebraid}.

\subsection{Positive Braids}

\begin{theorem}\label{thm:positivebraid}
Assume that $p, q, m, a_1, \dots, a_n$ are non-negative integers for some positive integer $n$ and that $p$ and $q$ are coprime.
Then the $p/q$--lashings of $P^{\alpha,m}$ are positive braids.  Furthermore, the $p/q$--lashing has genus 
\begin{align*}
g&= 
1 + (5 + 2 m) p^2 + (1 + 2 m) q^2 - 2 x_n + x_n^2 + x_n y_n \\
 \quad\quad &+ (4 + m - 4 x_n - m x_n - 2 y_n) p + 
  (1 + m - x_n - m x_n - y_n) q + (4 + 4 m) p q .
  \end{align*}
and an alternating surgery of integral slope
\begin{align*}
\lambda_{alt} 
& = (3 + m) p^2 + m q^2 + x_n^2 + x_n y_n \\
& \quad \quad + ( - 4 x_n - m x_n - 2 y_n) p +  (-x_n - m x_n - y_n) q + (1 + 2 m) p q.
\end{align*}
\end{theorem}

\begin{proof}

Assuming that $x_n \geq (m+2)(p+q)+2p$, Figure~\ref{fig:TrainTrackBraid} shows a sequence of isotopies that transforms the framed train track of Figure~\ref{fig:framedtraintrack} into one that carries closed positive braids. 

The recursive formula ($\ast\ast$) implies that $\{x_i\}$ is an increasing sequence where $x_1 \geq (m+2)(p+q)+2p$, with equality only when $m=a_1=0$. Thus the isotopies of Figure~\ref{fig:TrainTrackBraid} apply to show that the lashing is a closed positive braid.  The resulting braided train track is shown again at the top of Figure~\ref{fig:BraidedAndFlattened}.

Using this presentation as a closed positive braid, we can easily count the Euler characteristic of the fiber of the knot as the braid index minus the length of the positive braidword.   From this one calculates that the lashing has the genus stated.

We can further flatten this presentation of the train track, as shown in the bottom of Figure~\ref{fig:BraidedAndFlattened}, to have the blackboard framing while all crossings are still positive crossings.  The slope of the framing then equals the crossing number of this diagram. One also computes it to be as stated.
\end{proof}

\begin{figure}
	\centering
	\includegraphics[width=6in]{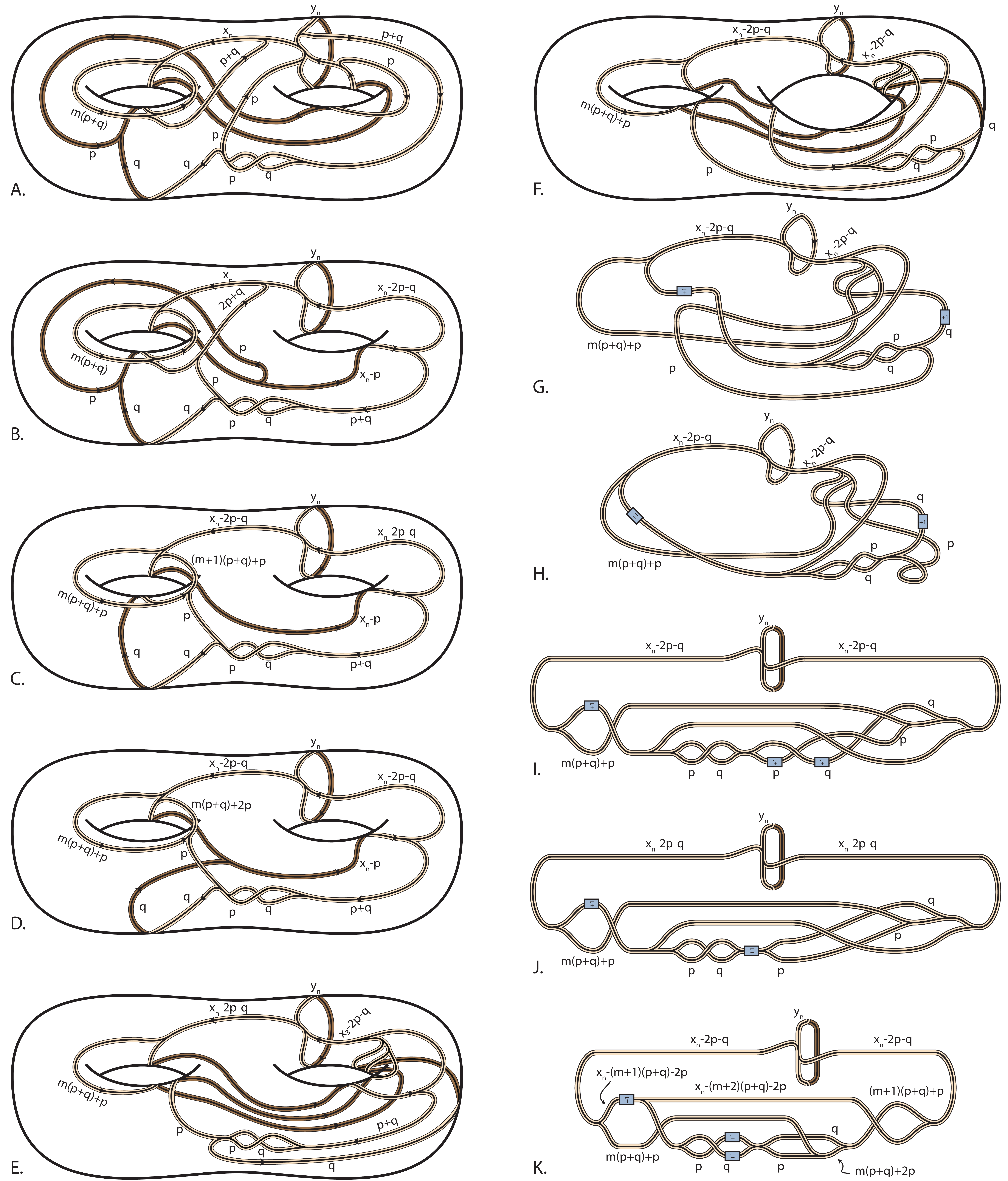}
	\caption{A sequence of isotopies shows the framed train track of Figure~\ref{fig:framedtraintrack} carries closed positive braids. 
	}
	\label{fig:TrainTrackBraid}
\end{figure}

\begin{figure}
	\centering
	\includegraphics[width=4in]{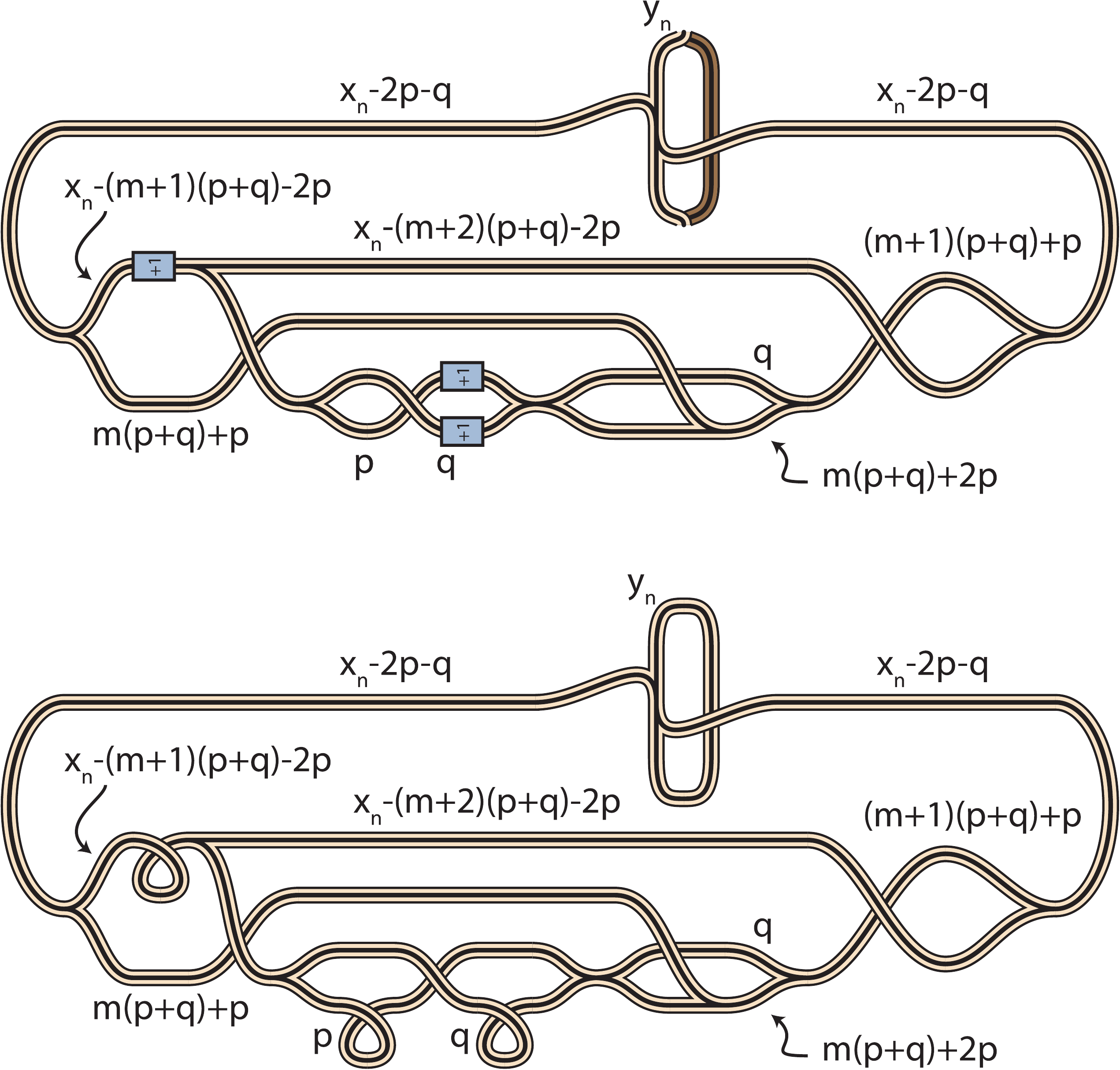}
	\caption{The final braided train track of Figure~\ref{fig:TrainTrackBraid} is shown along with a flattening of its twists.}
	\label{fig:BraidedAndFlattened}
\end{figure}

\subsection{Small examples in $S^3$}
 Let $K(a_3,a_2,a_1,m,b_1)$ denote the $1/b_1$--lashings of $P^{\alpha,m}$ constructed here  where $\alpha = \sigma_1^{-a_3} \sigma_2^{a_2} \sigma_1^{-a_1}$ for  non-negative integers $a_3, a_2, a_1, m$ and $b_1$.   For Theorem~\ref{thm:asymlspaceknots} to apply, we need each $a_1,a_2,a_3$ to be positive, $m\geq 3$, and 
to instead take, for instance, the $1+1/b_1$--lashings where $b_1$ is ``sufficiently large''. (The $b_1$--lashings correspond to taking $(p',q')=(0,1)$, hence Theorem~\ref{thm:asymlspaceknots} does not apply.)
 In practice, we find by computer verification that these conditions may be relaxed to still produce many asymmetric hyperbolic L-space knots.  Indeed, in general it appears considering the $1/b_1$--lashings with simply $b_1>0$ is suitable to find asymmetric hyperbolic L-space knots among the knots $K(a_3,a_2,a_1,m,b_1)$.  Furthermore we may also use $a_3=0$ and $m=1$ to produce examples.  

 For small parameter values, we use the surgery description above to input the knot into SnapPy, find its hyperbolic structure, and calculate its hyperbolic invariants.\footnote{ In the table, $\#$ Tetr. is the number of tetrahedra reported after using the {\tt canonize()} function. We have observed occurrences where this has increased the number of tetrahedra.}  These calculations are presented in Table~\ref{table:smallexamples}. We calculate the knot genus and slope of the alternating surgery using the formulae in Theorem~\ref{thm:positivebraid}.  
(Using SnapPy within SageMath, we can also determine the genus of the lashing $K$ from its Alexander polynomial -- being an L-space knot, $K$ is fibered.  The  calculation of homology of the alternating surgery  given at the end of Section~\ref{sec:surgdesc} also determines the surgery slope of the alternating surgery.)

Letting $J=J^{\alpha,m}(1,-b_1/(1+b_1),-1/(1+b_1))$ be the associated alternating link, by doing $(2,0)$ orbifold filling on all components of $J$ followed by taking the $2$--fold cover, we should obtain the result of the alternating surgery on $K$.  Indeed, in each of these examples SnapPy (with some help from Regina \cite{regina}) finds a hyperbolic structure on the resulting $2$--fold cover and identifies it as isometric to the alternating surgery on $K$.

\begin{table*}\centering
\caption{Examples of hyperbolic L-space knots $K(a_3,a_2,a_1,m,b_1)$ in $S^3$ with alternating surgeries}
\label{table:smallexamples}
\begin{tabular}{@{}rrrrrcrrcrrrrr@{}}
	\toprule
	\multicolumn{5}{c}{Parameters} & \phantom{abc} &\multicolumn{2}{c}{Top data} & \phantom{abc} & \multicolumn{4}{c}{Hyperbolic data}\\
	\cmidrule{1-5} \cmidrule{7-8} \cmidrule{10-13}
	$a_3$ & $a_2$ & $a_1$ & $m$ & $b_1$ && Genus & Alt.\ Surg.  && Sym & Volume & \# Tetr. & Cusp Shape \\ 
	\midrule
	$0$ & $1$ & $1$ & $1$ & $1$ && $119$ & $\Z/272$ && $\mathds{1}$ & $10.20098$ & $13$ &  $0.41433 + 1.19820 i$ \\
	$1$ & $1$ & $0$ & $1$ & $1$ && $214$ & $\Z/471$ && $\mathds{1}$ & $14.76163$ & $21$ &  $0.34127 + 1.50327 i$ \\		
	$0$ & $1$ & $1$ & $1$ & $2$ && $253$ & $\Z/555$ && $\mathds{1}$ & $12.39382$ & $16$ &  $0.47070 + 1.06723 i$ \\
	$0$ & $1$ & $1$ & $2$ & $1$ && $269$ & $\Z/588$ && $\mathds{1}$ & $11.26105$ & $15$ &  $0.28229 + 1.21171 i$ \\
	$1$ & $1$ & $0$ & $2$ & $1$ && $501$ & $\Z/1067$ && $\mathds{1}$ & $17.24796$ & $22$ & $0.28244 + 1.43145 i$ \\
	$1$ & $1$ & $1$ & $1$ & $1$ && $544$ & $\Z/1156$ && $\mathds{1}$ & $16.05972$ & $22$ &  $0.36009 + 1.53081 i$ \\	
	$0$ & $1$ & $1$ & $2$ & $2$ && $583$ & $\Z/1239$ && $\mathds{1}$ & $13.59287$ & $20$ &  $0.38818 + 1.09450 i$ \\		
	$1$ & $1$ & $1$ & $1$ & $2$ && $1117$ & $\Z/2331$ && $\mathds{1}$ & $18.24257$ & $22$ & $0.47741 + 1.38316 i$ \\

	\cmidrule{10-10}
	$0$ & $0$ & $1$ & $2$ & $2$ && $258$ & $\Z/563$ && $\Z/2$ & $9.12009$ & $17$ &  $0.07382 + 1.16144 i$ \\
	$1$ & $1$ & $1$ & $0$ & $2$ && $274$ & $\Z/597$ && $\Z/2$ & $7.47528$ & $13$ &  $0.31371 + 0.90262 i$ \\	
	\bottomrule
\end{tabular}

\end{table*}

\medskip

We now explicitly present the examples $K(0,1,1,1,1)$ and $K(1,1,1,1,1)$ as closures of positive braids using the train track of Figure~\ref{fig:TrainTrackBraid}.

The knot $K(0,1,1,1,1)$ (where $p=q=1$, $m=1$, and $a_1=a_2=1$) is the closure of the positive $12$--strand braid of length $249$ shown in Figure~\ref{fig:smallestAsymLSK-posbraid}.

The knot $K(1,1,1,1,1)$ (where $p=q=1$, $m=1$, and $a_1=a_2=a_3=1$) is the closure of the positive $29$--strand braid of length $1116$ shown in Figure~\ref{fig:smallAsymLSK-posbraid}.

\begin{figure}
	\centering
	\includegraphics[width=6in]{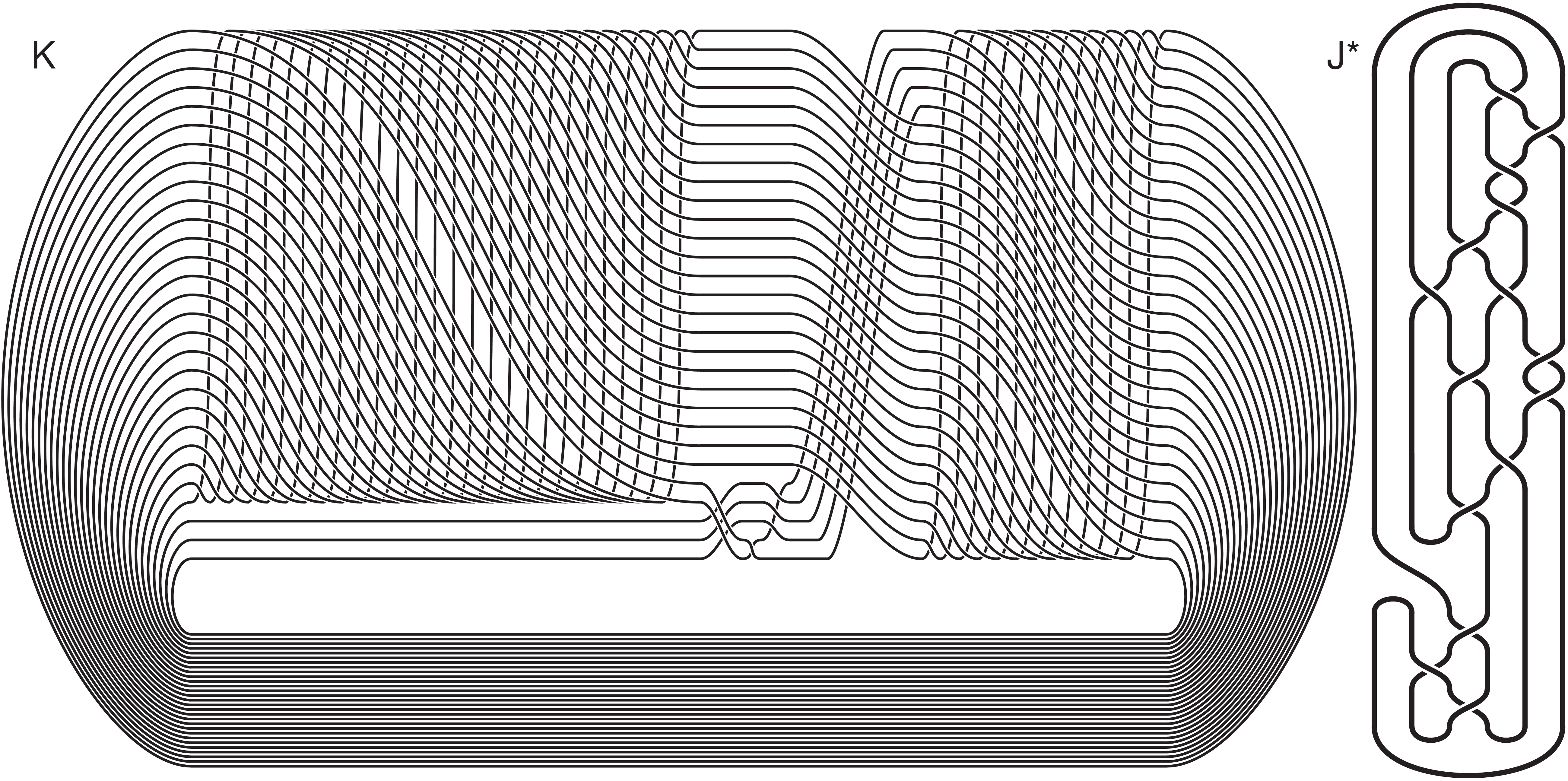}
	\caption{The knot $K=K(1,1,1,1,1)$ (Left) is the closure of the $29$--strand positive braid shown. 
		It is an asymmetric hyperbolic L-space knot.  The result of $1156/1$--surgery on $K$ knot produces the double branched cover of the $15$--crossing alternating link $J^*$ (Right).  }
	\label{fig:smallAsymLSK-posbraid}
\end{figure}

\subsection{Strongly invertible examples and bandings of unknots.} 
Table~\ref{table:smallexamples} shows that the two small examples $K(1,1,1,0,2)$ and $K(0,0,1,2,2)$ of our lashings are strongly-invertible hyperbolic L-space knots in $S^3$ with alternating surgeries. 
Using the surgery description of these knots, and simplifications given by SnapPy, one can
write the quotient of these surgeries as a band surgery to the unknot of the corresponding alternating
knots. In Figures~\ref{fig:K11102} and \ref{fig:K00122} we show the alternating knots $J^*(1,1,1,0,2)$ and $J^*(0,0,1,2,2)$ whose double branched covers are obtained by the framed surgeries on $K(1,1,1,0,2)$ and $K(0,0,1,2,2)$.  In each of these alternating diagrams a blackboard framed arc is also shown followed by a banding along the arc that produces an almost alternating diagram of the unknot.  Observe that a flype of the alternating diagram $J^*(0,0,1,2,2)$ to another alternating diagram of the knot is needed so that the banding produces an almost alternating diagram.  As shown, a further flype in fact allows the arc to be isotopic to a proper arc in a region of the alternating diagram, and the banding is dual to a smoothing of the dealternation crossing in an almost alternating unknot diagram. We propose the following
modification of Conjecture~\ref{conj:alternatingmccoy}.

\begin{figure}
	\centering
	\includegraphics[height=3in]{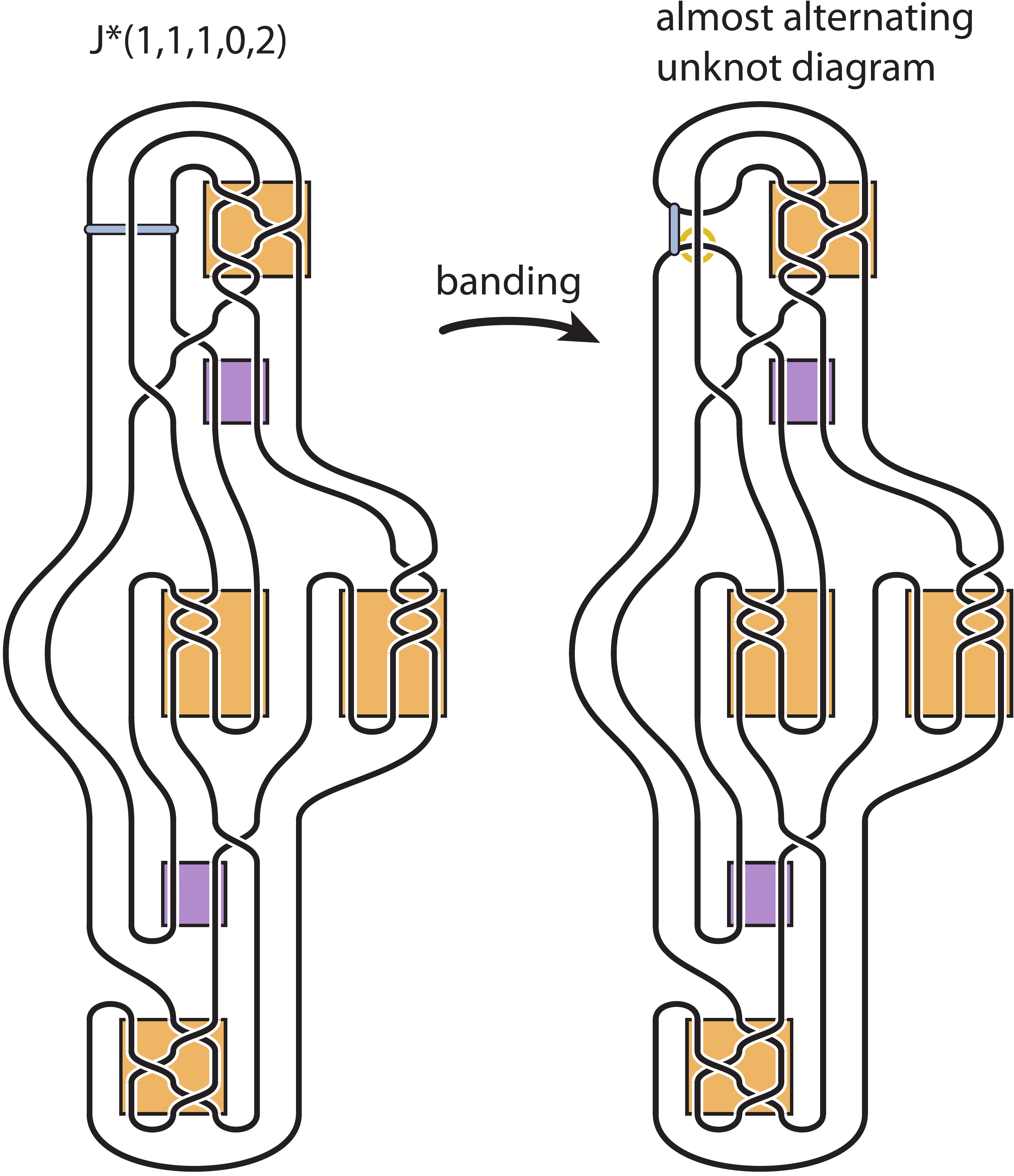}
	\caption{The alternating knot $J^*(1,1,1,0,2)$ admits a banding along an arc that crosses an alternating diagram once to produce an almost alternating diagram of the unknot.     }
	\label{fig:K11102}
\end{figure}

\begin{figure}
	\centering
	\includegraphics[width=6in]{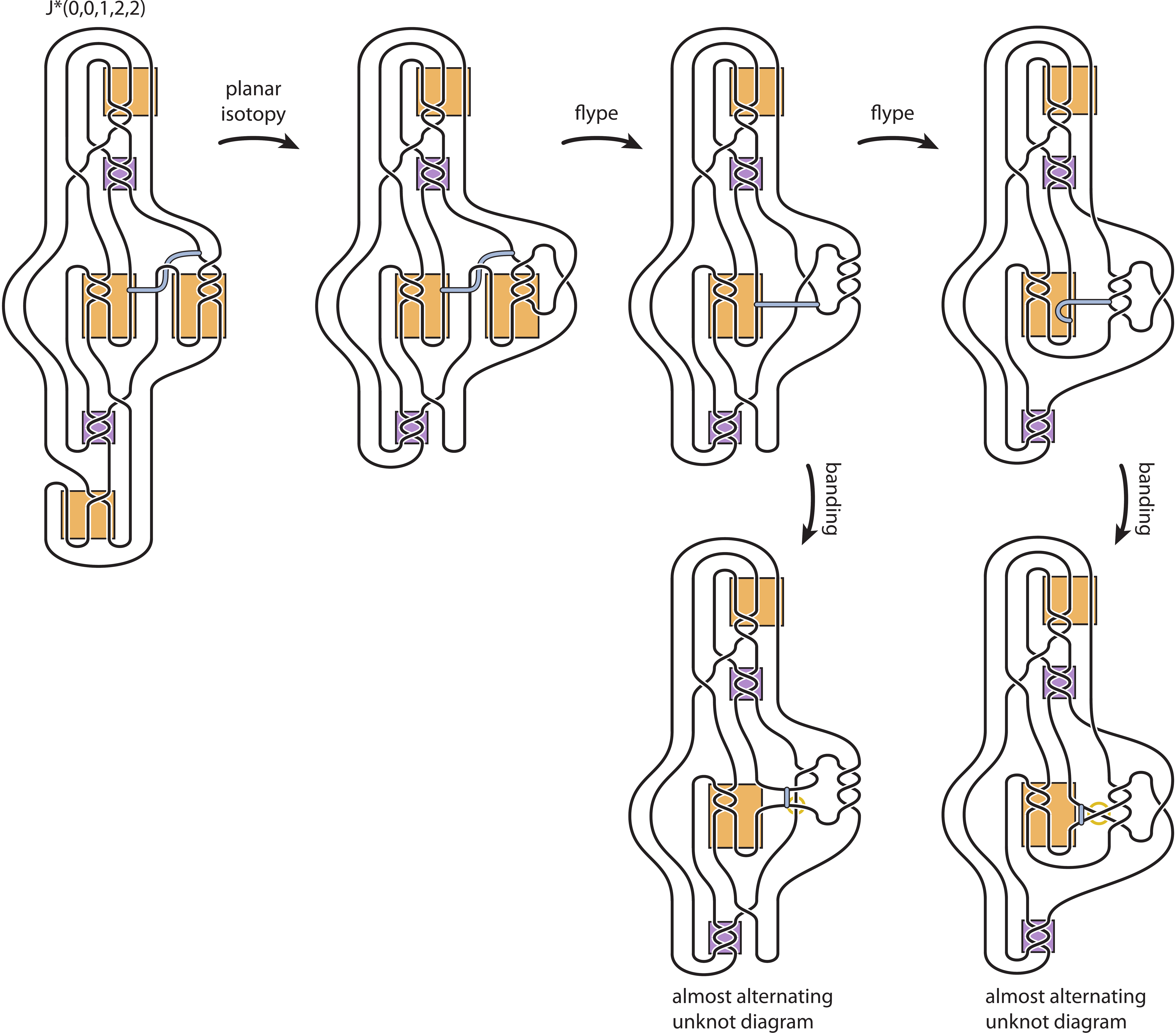}
	\caption{After planar isotopy and a flype, the alternating knot $J^*(0,0,1,2,2)$ admits a banding along a flat arc crossing an alternating diagram once to produce an almost alternating diagram of the unknot.  After a further flype, this banding is dual to the smoothing of the dealternation crossing in an almost alternating diagram of the unknot.   }
	\label{fig:K00122}
\end{figure}

\begin{conjecture}\label{conj:alternating}
 If a banding of the unknot produces an alternating link, then there is an almost alternating diagram of the unknot in which the banding is either 
 \begin{enumerate}
 \item a smoothing of the dealternation crossing or
 \item a flat banding from the dealternation crossing to an adjacent crossing. 
See Figure~\ref{fig:dealternationbandings}(Left).
\end{enumerate}
\end{conjecture}

Observe that if the first case of Conjecture~\ref{conj:alternating} occurs (see Figure~\ref{fig:dealternationbandings}(Right)), then the other smoothing of the dealternation crossing is obtained by a banding along the same core arc but with a framing that is twisted by half.  In particular, via the Montesinos Trick, this case corresponds to a knot in $S^3$ with two consecutive integral alternating surgeries.  Indeed this knot will have infinitely many alternating surgeries corresponding to alternating rational tangle replacements of the dealternation crossing. However in the second case it is not diagramatically apparent that the knot arising from the Montesinos Trick would have any other alternating surgeries.  

The links resulting from the two bandings of $J^*(1,1,1,0,2)$ that differ from the unknotting banding of Figure~\ref{fig:K11102} by a half-twist both have Jones polynomials (as one may calculate) with a span of $15$.  Thus if either were alternating, then its crossing number would be $15$ \cite{kauffman}.  However both of these links appear to have crossing number $18$.  We do not know if either of these links is alternating.  Crowell's condition \cite{crowell} that (reduced) Alexander polynomials of alternating links are alternating gives no obstruction.  Furthermore, since for each of the two links the two smoothings of the crossing arising from the half-twisted band produce the unknot and the alternating knot $J^*(1,1,1,0,2)$ of determinant $597$, the link with determinant $598$ is  necessarily {\em quasi-alternating}, see \cite{OS-dbc}.  We do not know if the one with determinant $596$ is quasi-alternating.

\begin{figure}
	\centering
	\includegraphics[width=6in]{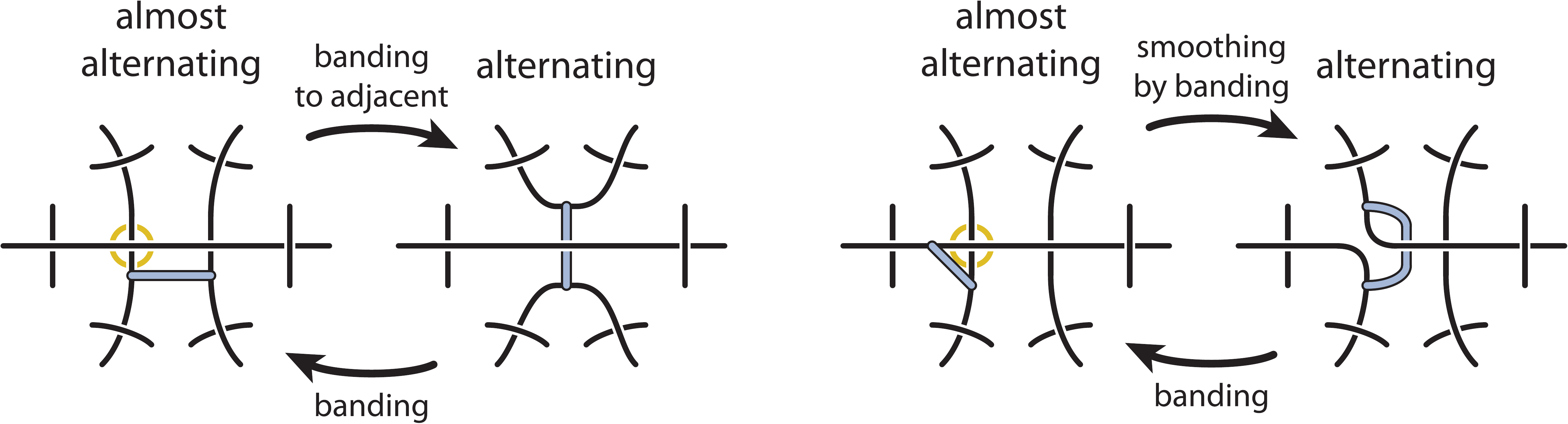}
	\caption{(Left) In an almost alternating diagram, a banding of the dealternation crossing to an adjacent crossing produces an alternating diagram.  The dual blackboard framed arc in the alternating diagram crosses the diagram once.  (Right) In an almost alternating diagram, a smoothing of the dealternation crossing by a banding is shown, though the intermediate removal of a nugatory crossing is not shown.  The dual blackboard framed arc in the alternating diagram is shown.}
	\label{fig:dealternationbandings}
\end{figure}

\subsection{Small examples in $S^1 \times S^2$}
Let $\alpha = \sigma_1^{-a_3} \sigma_2^{a_2} \sigma_1^{-a_1}$ and $\omega = \sigma_1^{a_1} \sigma_2^{-a_2} \sigma_1^{a_3+1}$.
 Then for any non-negative integers $a_1, a_2, a_3, m$, $J^{\alpha,\omega,m}$ is an almost alternating diagram of the two component unlink.  Hence for non-negative slope $p/q$, the $p/q$--lashing of $P^{\alpha,\omega,m}$ is a knot in $S^1\times S^2$ with an alternating surgery.  
Let $K'(a_3,a_2,a_1,m,b_1)$ denote the $1/b_1$--lashing in $S^1 \times S^2$.  
Using that $\omega = \alpha^{-1} \sigma_1$, one may readily adapt the surgery description obtained in Section~\ref{sec:surgdesc} to produce a surgery description of the knots $K'(a_3,a_2,a_1,m,b_1)$: change the surgery coefficient of the largest red circle in Figure~\ref{fig:fullsurgdesc-horiz} from $[0,a_1,-a_2,a_3]$ to $[0,a_1,-a_2,a_3+1]$.  Table~\ref{table:simpleS1xS2} collects the results of calculations using SnapPy of basic hyperbolic data of some of these knots and the homology of their alternating surgery.   According to \cite{NiVafaee}, knots in $S^1 \times S^2$ are spherical braids, they may be isotoped to be transverse to the $S^2$ fibers.  So in lieu of the knot genus, we also include the braid index in Table~\ref{table:simpleS1xS2}.  This is quickly calculated since the square of the winding number (with respect to the $S^1$ factor of $S^1 \times S^2$) of a non-null-homologous knot in $S^1 \times S^2$ is the order of the first homology of its framed surgeries.

\begin{table*}\centering
\caption{Examples of hyperbolic L-space knots $K'(a_3,a_2,a_1,m,b_1)$ in $S^1 \times S^2$ with alternating surgeries}
\label{table:simpleS1xS2}
\begin{tabular}{@{}rrrrrcrrcrrrrr@{}}
	\toprule
	\multicolumn{5}{c}{Parameters} & \phantom{abc} &\multicolumn{2}{c}{Top data} & \phantom{abc} & \multicolumn{4}{c}{Hyperbolic data}\\
	\cmidrule{1-5} \cmidrule{7-8} \cmidrule{10-13}
	$a_3$ & $a_2$ & $a_1$ & $m$ & $b_1$ && Braid Index & Alt.\ Surg.  && Sym & Volume & \# Tetr. & Cusp Shape \\ 
	\midrule
	$0$ & $1$ & $0$ & $1$ & $1$ && $16$ & $\Z/256$ && $\mathds{1}$ & $10.76121$ & $13$ & $0.38759 + 1.27100 i$ \\
	$0$ & $2$ & $0$ & $1$ & $1$ && $23$ & $\Z/529$ && $\mathds{1}$ & $12.58840$ & $17$ & $0.42497 + 1.42566 i$ \\
	$0$ & $1$ & $0$ & $1$ & $2$ && $23$ & $\Z/23+\Z/23$ && $\mathds{1}$ & $12.96886$ & $17$ & $0.44340 + 1.15926 i$ \\	
	$0$ & $1$ & $1$ & $1$ & $1$ && $26$ & $\Z/2 + \Z/338$ && $\mathds{1}$ & $11.29252$ & $15$ & $0.38693 + 1.26481 i$ \\
	$1$ & $1$ & $0$ & $1$ & $1$ && $28$ & $\Z/784$ && $\mathds{1}$ & $15.20006$ & $23$ & $0.35616 + 1.54654 i$ \\	
	$1$ & $1$ & $0$ & $1$ & $2$ && $40$ & $\Z/2 + \Z/800$ && $\mathds{1}$ & $17.47078$ & $23$ & $0.46933 + 1.38764 i$ \\
	\bottomrule
\end{tabular}
\end{table*}

The knot $K'(0,1,0,1,1)$ has surgery to the double branched cover of an alternating link which SnapPy identifies as the $3$--bridge link $L12a1091$.

\bibliographystyle{halpha}
\bibliography{asymmetric}

\end{document}